\documentclass[12pt]{amsart}

\usepackage{amsfonts}
\usepackage[mathscr]{eucal}
\usepackage{amssymb}
\usepackage[centertags]{amsmath}
\usepackage{amscd}
\usepackage{amsthm}
\usepackage{graphics,graphicx}
\usepackage[margin=2.5cm]{geometry}
\usepackage{dsfont}
\usepackage{bm}
\usepackage{graphics,graphicx}
\usepackage{subfigure}
\usepackage[all]{xy}

\newtheorem{thm}{Theorem}[section]
\newtheorem{prop}[thm]{Proposition}
\newtheorem{lem}[thm]{Lemma}
\newtheorem{cor}[thm]{Corollary}

\theoremstyle{definition}
\newtheorem{defn}{Definition}[section]
\newtheorem{eg}{Example}[section]
\newtheorem{rmk}[eg]{Remark}

\newcommand{\R}{\mathbb{R}}
\newcommand{\Z}{\mathbb{Z}}
\newcommand{\C}{\mathbb{C}}
\newcommand{\Q}{\mathbb{Q}}
\newcommand{\CP}[1]{\mathbb{CP}^{#1}}
\newcommand{\RP}[1]{\mathbb{RP}^{#1}}
\newcommand{\To}{\rightarrow}
\newcommand{\del}[2]{\frac{\partial #1}{\partial #2}}

\newcommand{\Log}{\mathrm{Log}}
\newcommand{\Arg}{\mathrm{Arg}}
\newcommand{\wedgestar}{\Lambda \!^*}
\def\cG{\mathcal{G}}
\def\id{\mathrm{id}}
\def\G{\mathbb{G}}
\def\fuk{\mathcal{F}uk}
\def\l{\ell}
\def\r{\mathrm{r}}
\def\Hom{\mathrm{Hom}}
\def\cF{\mathcal{F}}
\def\cP{\mathcal{P}}
\def\jmu{i}
\def\Ts{\mathsf{T}}

\date{}

\begin{document}
\title{On the Homological Mirror Symmetry conjecture for pairs of pants}
\author{Nick Sheridan}

\begin{abstract}
The $n$-dimensional pair of pants is defined to be the complement of $n+2$ generic hyperplanes in $\CP{n}$. 
We construct an immersed Lagrangian sphere in the pair of pants and compute its endomorphism $A_{\infty}$ algebra in the Fukaya category. On the level of cohomology, it is an exterior algebra with $n+2$ generators. It is not formal, and we compute certain higher products in order to determine it up to quasi-isomorphism. This allows us to give some evidence for the Homological Mirror Symmetry conjecture: the pair of pants is conjectured to be mirror to the Landau-Ginzburg model $(\C^{n+2},W)$, where $W = z_1 ... z_{n+2}$. We show that the endomorphism $A_{\infty}$ algebra of our Lagrangian is quasi-isomorphic to the endomorphism dg algebra of the structure sheaf of the origin in the mirror. This implies similar results for finite covers of the pair of pants, in particular for certain affine Fermat hypersurfaces.
\end{abstract}

\thanks{This research was partially supported by NSF Grant DMS-0652620.}

\maketitle

\tableofcontents

\section{Introduction}

\subsection{Homological Mirror Symmetry context}

In its original version, Kontsevich's Homological Mirror Symmetry conjecture \cite{kontsevich94} proposed that, if $X$ and $X^{\vee}$ are `mirror' Calabi-Yau varieties, then the Fukaya category of $X$ ($A$-model) should be equivalent, on the derived level, to the category of coherent sheaves on $X^{\vee}$ ($B$-model), and vice-versa.
Complete or partial results in this case are known for elliptic curves \cite{polzas98,polishchuk00}, abelian varieties \cite{fukaya02} (see \cite{absmith10} for the case of the four-torus), Strominger-Yau-Zaslow dual torus fibrations \cite{kontsoib01}, and K3 surfaces \cite{seidel03}.
One aim of this work is to generalize the arguments of \cite{seidel03} to the Fermat hypersurface in a projective space of arbitrary dimension -- we obtain a partial result in Theorem \ref{thm:mirrsymfermat}.

Kontsevich later proposed an extension of the conjecture to cover some Fano varieties \cite{kontsevich98}. 
The mirror of a Fano variety $X$ is a Landau-Ginzburg model $(X^{\vee},W)$, i.e., a variety $X^{\vee}$ equipped with a holomorphic function $W$ (called the superpotential). The definitions of the $A$- and $B$-models on $X$ are (roughly) the same as in the Calabi-Yau case, but the definitions on $(X^{\vee},W)$ must be altered. 
In particular, the $A$-model of $(X^{\vee},W)$ is the Fukaya-Seidel category, see \cite{seidel08}. 
The $B$-model of $(X^{\vee},W)$ is Orlov's triangulated category of singularities of $W$, see \cite{orlov04}.
Complete or partial results in the Fano case are known for toric varieties \cite{abouzaid06,abouzaid09,fltz08}, del Pezzo surfaces \cite{ako06}, and weighted projective planes \cite{ako08}.

More recently, Katzarkov and others have proposed another extension of the conjecture to cover some varieties of general type, see  \cite{katzarkov05,katzarkov10}.
The mirror of a variety $X$ of general type is again a Landau-Ginzburg model $(X^{\vee},W)$.
The definition of the $B$-model on $(X^{\vee},W)$ is as above (the definition of the $A$-model in this case is problematic, but does not concern us).
One direction of this conjecture has been verified for $X$ a curve of genus $g \ge 2$, see \cite{seidel08,efimov09}. 
Namely, the $A$-model of the genus $g$ curve is shown to be equivalent to the $B$-model of a Landau-Ginzburg mirror.
Our main result (Theorem \ref{thm:mirrsym}) gives evidence for the same direction of the conjecture in the case that $X$ is a `pair of pants' of arbitrary dimension.

\subsection{The $A$-model on the pair of pants}

Consider the smooth complex affine algebraic variety
\[\left\{\sum_{j=1}^{n+2} z_j = 0\right\} \subset \CP{n+1} \backslash \bigcup_{j=1}^{n+2} \{z_j = 0 \}.\]  
This is called the ($n$-dimensional) {\bf pair of pants} $\mathcal{P}^n$ (see \cite{mikhalkin04}). 
We equip it with an exact K\"{a}hler form by pulling back the Fubini-Study form on $\CP{n+1}$, and with a complex volume form $\eta$.
Observe that $\mathcal{P}^1$ is just $\CP{1} \setminus \{ \mbox{3 points}\}$, i.e., the standard pair of pants. 

We will consider the $A$-model on $\mathcal{P}^n$, i.e., Fukaya's $A_{\infty}$ category $\mathcal{F}uk(\mathcal{P}^n)$ (see \cite{fukaya93,fooo}). 
Recall that the objects of $\mathcal{F}uk(\mathcal{P}^n)$ are compact oriented Lagrangian submanifolds of $\mathcal{P}^n$, and the morphism space between transversely intersecting Lagrangians $L_1,L_2$ is defined as
\[ CF^*(L_1,L_2) := \bigoplus_{x \in L_1 \cap L_2} \mathbb{K}\langle x \rangle,\]
where $\mathbb{K}$ is an appropriate coefficient ring.
The $A_{\infty}$ structure maps are 
\[ \mu^d: CF^*(L_{d-1},L_d) \otimes \ldots \otimes CF^*(L_0,L_1) \To CF^*(L_0,L_d)[2-d],\]
for $d \ge 1$,
and their coefficients are defined by counts of rigid boundary-punctured holomorphic disks with boundary conditions on the Lagrangians $L_0, \ldots, L_d$.
Observe that, because the symplectic form on $\mathcal{P}^n$ is exact, the Fukaya category of exact Lagrangians is unobstructed (i.e., there is no $\mu^0$).

In general, $\mathbb{K}$ must be a Novikov field of characteristic $2$, and the morphism spaces of the  Fukaya category are $\Z_2$-graded.
If we require that the objects of our category be exact embedded Lagrangians, we remove the need for a Novikov parameter. 
If we furthermore require that our Lagrangians come equipped with a `brane' structure (a grading relative to the volume form $\eta$, and a spin structure), we can assign signs to the rigid disks whose count defines a structure coefficient of the Fukaya category, and therefore remove the need for our coefficient ring to have characteristic $2$.
The grading of Lagrangians also allows us to define a $\Z$-grading on the morphism spaces of the Fukaya category.
Thus, by restricting the objects of the Fukaya category to be exact Lagrangian branes, we can define the category with coefficients in $\C$, and with a $\Z$-grading.
For more details, see \cite{fooo} or \cite{seidel08}.

We construct an exact immersed Lagrangian sphere $L^n: S^n \To \mathcal{P}^n$ with transverse self-intersections, and a brane structure.
In the case $n=1$, we obtain an immersed circle with three self-intersections in $\mathcal{P}^1$, illustrated in Figure \ref{fig:l1} (ignore the additional labels for now).
This immersed circle also appeared in \cite{seidelg2}.

We point out that $L^n$ is not an object of the Fukaya category as just defined, because it is not embedded.
However, we will show (in Section \ref{subsec:afuk}) that one can nevertheless include $L^n$ as an `extra' object of the Fukaya category in a sensible way.

\begin{figure}
\centering
\includegraphics[width=0.5\textwidth]{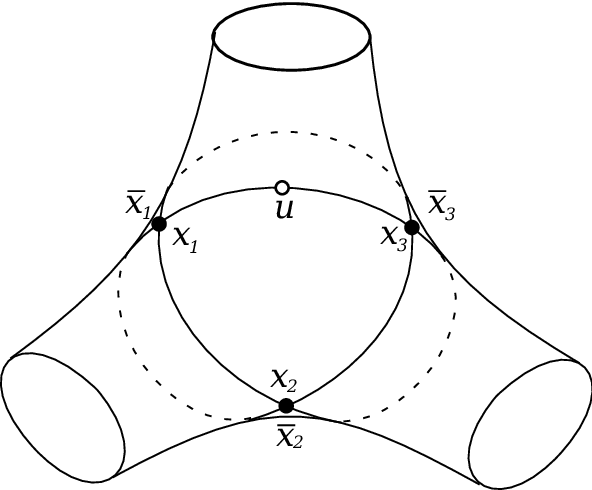}
\caption{The immersed Lagrangian $L^1:S^1 \To \mathcal{P}^1$. The image has been distorted for clarity -- for $L^1$ to be exact, the front and back triangles should have the same area.}
\label{fig:l1}
\end{figure}

We compute the Floer cohomology algebra of $L^n$:

\begin{thm}
\label{thm:zcoeffs}
\[HF^*(L^n,L^n) \cong \wedgestar \C^{n+2}\]
as $\Z_2$-graded associative $\C$-algebras.
\end{thm}

\begin{rmk}
Although both $HF^*(L^n,L^n)$ and $\wedgestar \C^{n+2}$ carry $\Z$-gradings, these gradings only agree modulo $2$.
\end{rmk}

\subsection{The $B$-model on the mirror}

The mirror of $\mathcal{P}^n$ is conjectured to be the Landau-Ginzburg model $(\C^{n+2},W)$, where
\[ W = z_1z_2\ldots z_{n+2}.\]
This paper is concerned with relating the $B$-model on $(\C^{n+2},W)$ to the $A$-model on $\mathcal{P}^n$.

Recall that the $B$-model of $(\C^{n+2},W)$ is described by Orlov's triangulated category of singularities  $D^b_{\mathrm{Sing}}(W^{-1}(0))$ (see \cite{orlov04}).
Note that $0$ is the only non-regular value of $W$.
The triangulated category of singularities is defined as the quotient of the bounded derived category of coherent sheaves, $D^bCoh(W^{-1}(0))$, by the full triangulated subcategory of perfect complexes $\mathrm{Perf}(W^{-1}(0))$.
It is a differential $\Z_2$-graded category over $\C$.

Because $\C^{n+2} = \mathrm{Spec} (R)$ is affine (where $R:= \C[z_1,\ldots,z_{n+2}]$), the triangulated category of singularities of $W^{-1}(0)$ admits an alternative description, which is more amenable to explicit computations.
Namely, it is quasi-equivalent to the category $MF(R,W)$ of `matrix factorizations' of $W$, by \cite[Theorem 3.9]{orlov04}.

An object of $MF(R,W)$ is a finite-rank free $\Z_2$-graded $R$-module $P=P^0 \oplus P^1$, together with an $R$-linear endomorphism $d_P:P \To P$ of odd degree, satisfying $d_P^2 = W \cdot \mathrm{id}_P$.
The space of morphisms from $P$ to $Q$ is the differential $\Z_2$-graded $R$-module of $R$-linear homomorphisms $f:P \To Q$, with the differential defined by
\[ d(f) := d_Q \circ f + (-1)^{|f|} f \circ d_P,\]
and composition defined in the obvious way.
This makes $MF(R,W)$ into a differential $\Z_2$-graded category over $\C$.

Under Homological Mirror Symmetry, our immersed Lagrangian sphere $L^n$ should correspond to $\mathcal{O}_0$, the structure sheaf of the origin in the triangulated category of singularities of $W^{-1}(0)$. 
This corresponds, under the above-described equivalence, to a matrix factorization of $W$, which by abuse of notation we will also denote $\mathcal{O}_0$.

It follows from the computations of \cite[Section 2]{dyckerhoff09} that, on the level of cohomology,
\[H^*\left( \mathrm{Hom}^*_{MF(R,W)}(\mathcal{O}_0,\mathcal{O}_0)\right) \cong \wedgestar \C^{n+2}\]
as $\Z_2$-graded associative $\C$-algebras.
Combining this with Theorem \ref{thm:zcoeffs} establishes an isomorphism between the endomorphism algebras of the alleged mirror objects on the level of cohomology.

The Homological Mirror Symmetry conjecture predicts more: this isomorphism of cohomology algebras should extend to a quasi-isomorphism of $A_{\infty}$ algebras.
Namely, 
\[\mathrm{Hom}^*_{MF(R,W)}(\mathcal{O}_0,\mathcal{O}_0)\]
inherits the structure of a differential $\Z_2$-graded $\C$-algebra from $MF(R,W)$, and a differential graded algebra is a special case of an $A_{\infty}$ algebra.

Our main result (proved by studying the $A_{\infty}$ deformations of the cohomology algebra) is that such a quasi-isomorphism does exist:

\begin{thm}
\label{thm:mirrsym}
There is a quasi-isomorphism
\[CF^*(L^n,L^n) \cong \mathrm{Hom}^*_{MF(R,W)}(\mathcal{O}_0,\mathcal{O}_0)\]
as $\Z_2$-graded $A_{\infty}$-algebras over $\C$.
\end{thm}

\begin{rmk}
Of course the $B$-model $D^b_{\mathrm{Sing}}(\C^{n+2},W)$ cannot be equivalent, in any sense, to the $A$-model $\mathcal{F}uk(\mathcal{P}^n)$ as we define it, because the morphism spaces in the $B$-model can be infinite-dimensional (even on the cohomology level) whereas the morphism space between two compact Lagrangians is always finite-dimensional. 
To get an $A$-model which has a hope of being equivalent to the $B$-model in some sense, we must consider the `wrapped' Fukaya category (see \cite{abseivit}), which also includes non-compact Lagrangians.
\end{rmk}

\subsection{Motivation: the $A$-model on the one-dimensional pair of pants}
\label{subsec:p1}

In this section, we consider the $1$-dimensional case. 
We hope that this will aid the reader's intuition for the subsequent arguments, and provide a link with computations that have previously appeared in the literature (in \cite[Section 10]{seidelg2}), but this section could be skipped without serious harm.

Consider the immersed Lagrangian $L^1:S^1 \To \mathcal{P}^1$ shown in Figure \ref{fig:l1}. 
We outline a description of the $A_{\infty}$ algebra $\mathcal{A} = CF^*(L^1,L^1)$ up to quasi-isomorphism.

$\mathcal{A}$ has generators $u,q$ corresponding respectively to the identity and top class in the Morse cohomology $CM^*(S^1)$, and two generators for each self-intersection point, which we label $x_1,\bar{x}_1,x_2,\bar{x}_2,x_3,\bar{x}_3$ as in Figure \ref{fig:l1}.

Because the homology class of $L^1$ is trivial in $H_1(\mathcal{P}^1)$, the generators of $\mathcal{A}$ come labeled by weights which are elements of the lattice 
\[H_1\left(\mathcal{P}^1\right) \cong \Z \langle e_1,e_2,e_3 \rangle / \langle e_1 + e_2 + e_3 \rangle,\]
so that the $A_{\infty}$ structure maps are homogeneous with respect to these weights. 
This is just because the disk contributing to such a product lifts to the universal cover, so its boundary must lift to a closed loop.
See Definition \ref{defn:weight} and Proposition \ref{prop:top} for the precise definition and argument.
Explicitly, the weight of $u,q$ is $0$, of $x_i$ is $e_i$ and of $\bar{x}_i$ is $-e_i$.
It follows that $\mu^1 = 0$.

The $A_{\infty}$ structure maps count rigid holomorphic disks, which in this case is purely combinatorial. 
Our first step is to determine the cohomology algebra of $\mathcal{A}$, which has the (associative) product defined by
 \[ a \cdot b := (-1)^{|a|} \mu^2(a, b)\]
 (using the sign conventions of \cite{seidel08}).
 
We have the following result:
\begin{lem}
The cohomology algebra of $\mathcal{A}$ is isomorphic (as $\Z_2$-graded associative $\C$-algebra) to the exterior algebra
\[ \wedgestar \C \langle e_1, e_2, e_3 \rangle\]
via the identification 
\begin{eqnarray*}
u & \mapsto & \mathds{1} \\
x_i & \mapsto & (-1)^i e_i \\
\bar{x}_i & \mapsto & (-1)^{i+1} \ast e_i \mbox{ (Hodge star with respect to volume form $e_1 \wedge e_2 \wedge e_3$)}\\
q & \mapsto & - e_1 \wedge e_2 \wedge e_3.
\end{eqnarray*}
\end{lem}
\begin{proof}\textit{ (sketch -- see \cite{seidelg2} for a more detailed proof)}
The contributions of constant disks give all products involving $u$ and $q$.
The other products come from the two triangles on the front and back of Figure \ref{fig:l1}.
For example, the triangle with vertices in cyclic order $x_1,x_2,x_3$ gives the product
\[ \mu^2(x_1,x_2) =  \bar{x}_3\]
corresponding to
\[ e_1 \cdot e_2 = \ast e_3 = e_1 \wedge e_2.\]
We will not explain how to determine the signs here -- see Section \ref{subsec:signs} (or \cite{seidelg2}) for more detail.
\end{proof}

Furthermore, we have
\[ \mu^3(x_1,x_2,x_3) = - u,\]
but the corresponding product is $0$ for any other permutation of the inputs.
This comes from the degenerate $4$-gon with vertices at $u,x_1,x_2,x_3$. 
Observe that, if we put the marked point $u$ somewhere else on $L^1$, this product would again be equal to $u$, but possibly for a different permutation of the inputs (and would be $0$ on all other permutations).

By choosing a complex volume form $\eta$ on $\mathcal{P}^1$ and computing grading of the generators, one can lift the $\Z_2$-grading of $\mathcal{A}$ (defined by the sign of the intersection point corresponding to the generator) to a $\Z$-grading. 
See \cite{seidelg2} for a formula for the grading that holds in the $1$-dimensional case.
The choice of volume form is not canonical, and hence the choice of $\Z$-grading is not canonical.

We have now shown that $\mathcal{A}$ lies in the set $\mathfrak{A}$ of $A_{\infty}$ algebras satisfying the following conditions:
\begin{itemize}
\item $\mu^1 = 0$;
\item The cohomology algebra is isomorphic to $\wedgestar \C \langle e_1,e_2,e_3 \rangle$ as $\Z_2$-graded associative $\C$-algebra;
\item The $A_{\infty}$ structure maps are homogeneous with respect to the weights as defined above;
\item The $\Z_2$-grading lifts to a $\Z$-grading as defined above.
\end{itemize}
One can show that $\mathfrak{A}$ has a one-dimensional deformation space, in the sense of \cite[Lemma 3.2]{seidel03}.
Furthermore, the deformation class of $\mathcal{A}$ in this deformation space is given by 
\[ \sum_{i,j,k = 1}^3 \mu^3(x_i,x_j,x_k) = \mu^3(x_1,x_2,x_3) = -u\]
by our previous computations. 
In particular, it is non-zero, so $\mathcal{A}$ is versal.
This determines $\mathcal{A}$ up to quasi-isomorphism, in the sense that any $A_{\infty}$ algebra lying in $\mathfrak{A}$, with non-zero deformation class, is quasi-isomorphic to $\mathcal{A}$.

\subsection{Outline of the paper} 

In Section \ref{sec:ln} we introduce some standing notation, and discuss the topology of the pair of pants $\mathcal{P}^n$. 
In particular, we introduce the coamoeba, which encodes topological information about $\mathcal{P}^n$ and is the starting point for understanding the Lagrangian immersion $L^n$.
We give the details of the construction of the Lagrangian immersion $L^n: S^n  \To \mathcal{P}^n$, and some of its properties.

In Section \ref{sec:A}, we explain how to include the Lagrangian immersion $L^n$ as an `extra' object of the Fukaya category of embedded Lagrangians in $\mathcal{P}^n$. 
We define the $A_{\infty}$ algebra $\mathcal{A} := CF^*(L^n,L^n)$, and establish some of its properties -- namely, that it is homogeneous with respect to a certain weighting of its generators, that its $\Z_2$-grading lifts to a $\Z$-grading, and that it has a certain `super-commutativity' property.

In Section \ref{sec:fukdef}, we give an alternative, Morse-Bott definition of the Fukaya category of embedded Lagrangians. 
We define the $A_{\infty}$ structure coefficients by counts of objects called `holomorphic pearly trees', which are Morse-Bott versions of the holomorphic disks usually used (and closely related to the `clusters' of \cite{cornealalonde}).
The technical parts of this section could be skipped at a first reading, but the concept of a pearly tree is important because it is the basis of our main computational technique, which is introduced in Section \ref{sec:Acalc}. 
This section could be read independently of the rest of the paper.

In Section \ref{sec:Acalc}, we introduce a Morse-Bott model $\mathcal{A}'$ for the $A_{\infty}$ algebra $\mathcal{A}$, in which the $A_{\infty}$ structure coefficients are defined by counts of objects called `flipping holomorphic pearly trees'.
We show that $\mathcal{A}'$ is quasi-isomorphic to $\mathcal{A}$. 
We can compute the $A_{\infty}$ structure maps of $\mathcal{A}'$ by explicitly identifying the relevant moduli spaces of flipping holomorphic pearly trees.
In particular, we compute that the cohomology algebra of $\mathcal{A}'$ (hence of $\mathcal{A}$) is an exterior algebra, as well as some of the higher structure maps.
We use our computation of higher structure maps to show that $\mathcal{A}'$ is versal in the class of $A_{\infty}$ algebras with cohomology algebra the exterior algebra, and the homogeneity and grading properties described in Section \ref{sec:A} (compare Section \ref{subsec:p1}).
Thus, applying deformation theory of $A_{\infty}$ algebras, $\mathcal{A}'$ (and hence $\mathcal{A}$) is completely determined up to quasi-isomorphism by the coefficients and properties that we have established.

In Section \ref{sec:matfact}, we describe the $B$-model of the mirror.
We use the techniques of \cite[Section 4]{dyckerhoff09} to construct a minimal $A_{\infty}$ model $\mathcal{B}'$ for the differential $\Z_2$-graded algebra $\mathcal{B} := \mathrm{Hom}^*_{MF(R,W)}(\mathcal{O}_0,\mathcal{O}_0)$.
We find that its cohomology algebra is an exterior algebra, and that it has the same grading and equivariance properties as $\mathcal{A}$. 
We compute higher products to show that $\mathcal{B}'$ is versal in the same class of $A_{\infty}$ algebras as $\mathcal{A}'$, and hence that it is quasi-isomorphic to $\mathcal{A}'$.
This completes the proof of Theorem \ref{thm:mirrsym}.

In Section \ref{sec:app}, we give applications of Theorem \ref{thm:mirrsym}.
In particular, we consider the Homological Mirror Symmetry conjecture for Fermat hypersurfaces.

Let $\widetilde{X}^n$ be the intersection of the Fermat Calabi-Yau hypersurface
\[ \{ z_1^{n+2} + \ldots + z_{n+2}^{n+2} = 0\} \subset \CP{n+1}\]
with the open torus $(\C^*)^{n+1} \subset \CP{n+1}$.
Let $Y^n$ be the singular variety
\[ \{W = 0\} \subset \CP{n+1}\]
(where $W = z_1z_2\ldots z_{n+2}$ as before),
and equip it with the natural action of 
\[ G_n := (\Z_{n+2})^{n+2}/\Z_{n+2}\] 
(where $\Z_{n+2}$ is the diagonal subgroup of $(\Z_{n+2})^{n+2}$) by multiplying coordinates by $(n+2)$th roots of unity.

Then we have the following:

\begin{thm}
\label{thm:mirrsymfermat}
There is a full and faithful $A_{\infty}$ embedding
\[ \mathrm{Perf}_{G_n}(Y^n) \hookrightarrow D^{\pi} \mathcal{F}uk(\widetilde{X}^n)\]
of the category of perfect complexes of $G_n$-equivariant sheaves on $Y^n$ into the derived Fukaya category of $\widetilde{X}^n$.
\end{thm}

We conjecture that this embedding is an equivalence.

{\bf Acknowledgments.} I thank my advisor, Paul Seidel, as well as Mohammed Abouzaid and Grigory Mikhalkin, for stimulating conversations and a number of crucial insights into this work. 
I also thank James Pascaleff and Paul Seidel for reading drafts of this paper in detail and making many useful suggestions. 
I also thank Siu-Cheong Lau for pointing out an error in the proof of Corollary \ref{cor:cohsigns}. 
I also thank Denis Auroux and Katrin Wehrheim for some very helpful discussions, as well as MSRI for the great atmosphere at the tropical and symplectic geometry workshops and conferences of 2009 and 2010, where part of this work was carried out.

\section{The Lagrangian immersion $L^n: S^n  \To \mathcal{P}^n$}
\label{sec:ln}

The aim of this section is to describe the immersed Lagrangian sphere $L^n: S^n  \To \mathcal{P}^n$. 
In Section \ref{subsec:coamoeba} we introduce some standing notation, and describe the topology of the pair of pants $\mathcal{P}^n$. 
We introduce the notion of the {\bf coamoeba} of the pair of pants, which is the starting point for visualising the Lagrangian immersion $L^n$.

In Section \ref{subsec:ln} we construct the Lagrangian immersion $L^n: S^n \To \mathcal{P}^n$ and establish some of its properties.

\subsection{Topology of $\mathcal{P}^n$ and coamoebae}
\label{subsec:coamoeba}

Let $[k]$ denote the set $\{1,2,\ldots,k\}$.
For a subset $K \subset [k]$, let $|K|$ be its number of elements and $\bar{K}\subset [k]$ its complement.
Let $\widetilde{M}$ be the $(n+2)$-dimensional lattice
\[\widetilde{M} := \Z \langle e_1,\ldots,e_{n+2} \rangle.\]
For $K \subset [n+2]$, let $e_K$ denote the element
\[ e_K := \sum_{j \in K} e_j \in \widetilde{M}.\]
Let $M$ be the $(n+1)$-dimensional lattice
\[ M := \widetilde{M}/ \langle e_{[n+2]}\rangle.\]
We will use the notation 
\[M_{P} := M \otimes_{\Z} P\]
for any $\Z$-module $P$.
We will not distinguish notationally between a lattice element $e_K \in \widetilde{M}$ and its image in $M$.
We define maps
\begin{eqnarray*}
\widetilde{\Log}: \widetilde{M}_{\C^*} & \To & \widetilde{M}_{\R}, \\
\widetilde{\Log} (z_1,\ldots,z_{n+2}) &:=& (\log|z_1|, \ldots,\log|z_{n+2}|) \\
\widetilde{\Arg}: \widetilde{M}_{\C^*} & \To & \widetilde{M}_{\R}/2\pi \widetilde{M}, \\
\widetilde{\Arg} (z_1,\ldots,z_{n+2}) &:=& (\arg(z_1), \ldots,\arg(z_{n+2})).
\end{eqnarray*}
These descend to maps
\begin{eqnarray*}
\Log: M_{\C^*} & \To & M_{\R}, \\
\Arg: M_{\C^*} & \To & M_{\R}/2\pi M. \\
\end{eqnarray*}

We can identify 
\[\widetilde{M}_{\C^*} = \C^{n+2} \setminus \bigcup_j \{z_j = 0\}\]
and the quotient by the diagonal $\C^*$ action,
\[M_{\C^*} = \CP{n+1} \setminus D\]
where we denote the divisors $D_j := \{z_j = 0\}$ for $j = 1,\ldots,n+2$, and $D$ is the union of all $D_j$.
Thus we have
\[\mathcal{P}^n = \left\{ \sum_{j=1}^{n+2} z_j = 0\right\} \subset M_{\C^*}.\]

\begin{defn}
The closure of the image $\Arg(\mathcal{P}^n)$ is called the {\bf coamoeba} (also, sometimes, the {\bf alga}) of $\mathcal{P}^n$, and we will denote it $\mathcal{C}^n$ (see, e.g., \cite{fhkv08,passarenilsson}).
\end{defn}

Now we will give a description of the coamoeba $\mathcal{C}^n$ for all $n$. 
It will be described in terms of a certain polytope, which we first describe.

\begin{defn}
Let $Z_n$ be the zonotope generated by the vectors $e_j$ in $M_{\R}$, i.e.,
\[Z_n = \left\{\sum_{j=1}^{n+2} \theta_j e_j :\theta_j \in [0,1]\right\} \subset M_{\R}\]
(this is the projection of the cube $[0,1]^{n+2}$ in $\widetilde{M}_{\R}$).
\end{defn}

\begin{defn}
\label{defn:qncells}
The cells of $\partial Z_n$ are indexed by triples of subsets $J,K,L \subset [n+2]$ such that
\begin{itemize}
\item $J \sqcup K \sqcup L = [n+2]$;
\item $J \neq \phi$ and $K \neq \phi$.
\end{itemize}
Namely, we define the cell
\begin{eqnarray*}
 U_{JKL}&:= &\left\{\sum_{i=1}^{n+2} \theta_i e_i: \theta_j = 0 \mbox{ for $j \in J$, } \theta_k = 1 \mbox{ for $k \in K$, } \theta_l \in [0,1] \mbox{ for $l \in L$}\right\} \\
& \subset& \partial Z_n.
\end{eqnarray*}
\end{defn}

We note that
\[\mathrm{dim}\left(U_{JKL}\right) = |L|,\]
and $U_{J'K'L'}$ is part of the boundary of $U_{JKL}$ if and only if
\[ J \subseteq J', K \subseteq K', \mbox{ and } L \supsetneq L'.\]
In particular, the vertices of $Z_n$ are the $0$-cells $U_{\bar{K},K,\phi} = \{e_K\}$, and are indexed by proper, non-empty subsets $K \subset [n+2]$.

\begin{prop}
\label{prop:coamoeba}
$\mathcal{C}^n \subset M_{\R}/2\pi M$ is the complement of the image of the interior of $\pi Z_n$.
\end{prop}
\begin{proof}
$\mathcal{C}^n$ is the closure of the set of those 
\[\bm{\theta} = \sum_j \theta_j e_j\]
such that there exist $r_j$ satisfying
\[ \sum_{j=1}^{n+2} \exp(r_j+i\theta_j) = 0.\]
In other words, the convex cone spanned by the vectors $\exp(i \theta_j)$ contains $0$.

Therefore the complement of $\mathcal{C}^n$ consists of exactly those $\bm{\theta}$ such that the coordinates $\theta_1,\ldots,\theta_{n+2}$ are contained in an interval of length $<\pi$. 
By adding a common constant we may assume all $\theta_j$ lie in $[0,\pi)$. 
Thus the complement of $\mathcal{C}^n$ is exactly the image of the interior of $\pi Z_n$.
\end{proof}

\begin{rmk}
\label{rmk:verts}
As we saw in Definition \ref{defn:qncells}, the vertices of $\partial (\pi Z_n)$ are the points $\pi e_K$ where $K \subset [n+2]$ is proper and non-empty.
Observe that the vertices $\pi e_K,\pi e_{\bar{K}}$ get identified because 
\[\pi e_K - \pi e_{\bar{K}} \in 2\pi M.\]
\end{rmk}

We can draw pictures in the lower-dimensional cases (see Figure \ref{fig:coamoebas}).

\begin{prop}
\label{prop:arghomeq}
The map $\Arg: \mathcal{P}^n \To \mathcal{C}^n$ is a homotopy equivalence.
In particular, $\mathcal{P}^n$ has the homotopy type of an $(n+1)$-torus with a point removed.
\end{prop}
\begin{proof}
We choose to work in affine coordinates 
\[\tilde{z}_j := \frac{z_j}{z_{n+2}} \mbox{ for $j = 1, \ldots, n+1$}\]
on $\CP{n+1} \setminus D$.
So
\[\mathcal{P}^n \cong \{1+\tilde{z}_1+\ldots + \tilde{z}_{n+1} = 0\} \subset (\C^*)^{n+1}.\]
It is shown in \cite{hattori75} that there exists a subset $W \subset \mathcal{P}^n$, such that the inclusion $W \hookrightarrow \mathcal{P}^n$ is a homotopy equivalence, and the projection 
\[\Arg: W \To M_{\R}/2\pi M\]
is a homotopy equivalence onto its image, which is
\[ \Arg(W) = \left\{(\tilde{\theta}_1,\ldots,\tilde{\theta}_{n+1}): \mbox{ at least one } \tilde{\theta}_j = \pi \right\} \subset M_{\R}/2\pi M.\]
It is easy to see that the inclusion
\[ \Arg(W) \hookrightarrow \mathcal{C}^n\]
is a homotopy equivalence (both are strong deformation retracts of $(M_{\R}/2\pi M) \setminus (0,0,\ldots,0)$). Hence, we have a commutative diagram
\[ \xymatrix{
W \ar[r] \ar[d] & \mathcal{P}^n \ar[d]^{\mathrm{Arg}}\\
\Arg(W) \ar[r] & \mathcal{C}^n}
\]
in which all arrows but the one labeled `$\Arg$' are known to be homotopy equivalences.
It follows that $\Arg: \mathcal{P}^n \To \mathcal{C}^n$ is also a homotopy equivalence.
\end{proof}

\begin{cor}
\label{cor:pi1}
For $n>1$, there are natural isomorphisms
\[\pi_1(\mathcal{P}^n) \cong H_1(\mathcal{P}^n) \cong M.\]
When $n=1$, we still have a natural isomorphism $H_1(\mathcal{P}^1) \cong M$, but the fundamental group is no longer abelian. 
Instead, there is a natural isomorphism
\[ \pi_1(\mathcal{P}^1) \cong \langle a,b,c | abc \rangle.\]
\end{cor}

\begin{figure}
\centering
\subfigure[The coamoeba of $\mathcal{P}^1$.]{
\includegraphics[width=0.3\textwidth]{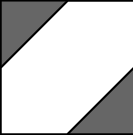}}
\hfill
\subfigure[The coamoeba of $\mathcal{P}^2$. This picture lives in $(S^1)^3$, drawn as a cube with opposite faces identified, and we are removing the zonotope illustrated, which looks somewhat like a crystal.]{
\includegraphics[width=0.45\textwidth]{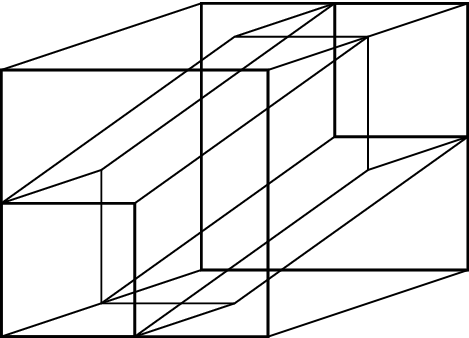}
\label{fig:c2}}
\caption{$\mathcal{C}^1$ and $\mathcal{C}^2$.
\label{fig:coamoebas}}
\end{figure}

\subsection{Construction of the Lagrangian immersion $L^n: S^n  \To \mathcal{P}^n$}
\label{subsec:ln}

We observe that the Lagrangian $L^1: S^1 \To \mathcal{P}^1$ can be seen rather simply in the coamoeba. 
It corresponds to traversing the hexagon which forms the boundary of the coamoeba (see Figure \ref{fig:l1co}).
The two triangles that make up the coamoeba correspond to the holomorphic triangles that give the product structure on Floer cohomology.

\begin{figure}
\centering
\includegraphics[width=0.5\textwidth]{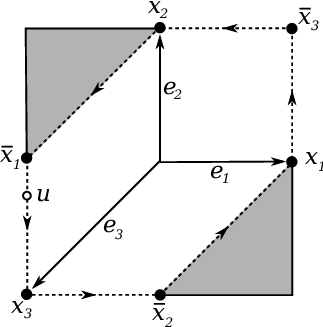}
\caption{The projection of $L^1$ to $\mathcal{C}^1$.}
\label{fig:l1co}
\end{figure}

We will show that a similar picture exists for higher dimensions. 
Namely, by Proposition \ref{prop:coamoeba}, we know that the boundary of $\mathcal{C}^n$ is a polyhedral $n$-sphere that intersects itself at its vertices.
In this section, we will explain how to lift this immersed polyhedral $n$-sphere to an immersed Lagrangian $n$-sphere in $\mathcal{P}^n$.

\begin{rmk}
This is not the first time that the coamoeba has been used to study Floer cohomology. 
It appeared in \cite{fhkv08} (with the name `alga'), where it was used as motivation to construct Landau-Ginzburg mirrors to some toric surfaces. 
It was conjectured in \cite{futueda10} that this picture generalizes to higher dimensions. 
There is a connection between the `tropical coamoeba' of the Landau-Ginzburg mirror $(X,w)$ of projective space, introduced in \cite{futueda10}, and our construction, but we will not go into it.
\end{rmk}

Consider the real projective space 
\[\RP{n} = \left\{\sum_j z_j = 0, z_j \in \R \right\} \subset \left\{\sum_j z_j = 0 \right\} \subset \CP{n+1}.\]
Clearly it is Lagrangian and invariant with respect to the $S_{n+2} \times \Z_2$ action, so by an equivariant version of the Weinstein Lagrangian neighbourhood theorem, there is an $S_{n+2} \times \Z_2$-equivariant symplectic embedding of the radius-$\eta$ disk cotangent bundle
\[D^*_{\eta} \RP{n} \hookrightarrow \left\{\sum_j z_j = 0\right\} \subset \CP{n+1}\]
for some sufficiently small $\eta >0$. 
We may choose this embedding to be $J$-holomorphic along the zero section with respect to the almost-complex structure induced by the standard symplectic form and metric on $D^*_{\eta}\RP{n}$. The $\Z_2$-invariance says that complex conjugation acts on $D^*_{\eta}\RP{n}$ by $-1$ on the covector.

Our immersed sphere $L^n$ will land inside this neighbourhood.
Now consider the double cover of $\RP{n}$ by $S^n$.
Think of $S^n$ as 
\[ S^n = \left\{ \sum_j x_j^2 = 1 \right\} \bigcap \left\{\sum_j x_j = 0\right\}  \subset \R^{n+2},\]
and denote the real hypersurfaces
\[ D_j^{\R} := \{x_j = 0\} \subset S^n.\]
Then the double cover just sends $(x_1,\ldots,x_{n+2}) \mapsto [x_1:\ldots:x_{n+2}]$.
This extends to a double cover $D^*_{\eta} S^n \To D^*_{\eta} \RP{n}$. 
Composing this with the inclusion $D^*_{\eta} \RP{n} \To \CP{n}$ gives a map $i : D^*_{\eta} S^n \To \CP{n}$.

\begin{lem}
\label{lem:transdiv}
Suppose that $f:S^n \To \R$ is a smooth function whose gradient vector field (with respect to the round metric on $S^n$) is transverse to the real hypersurfaces $D_j^{\R}$.
Then for sufficiently small $\epsilon > 0$, the image of the graph $\Gamma(\epsilon df) \subset T^* S^n$ lies inside $D^*_{\eta} S^n$, and its image under the map $i$ into $\CP{n}$ avoids the divisors $D_j$.
\end{lem}
\begin{proof}
Note that the graph of $\epsilon df$ in $D^*_{\eta} S^n$ is the time-$\epsilon$ flow of the zero-section by the Hamiltonian vector field corresponding to $f$, which is exactly $J(\nabla f)$, where $J$ is the standard complex structure on $\CP{n}$ (we observe that the round metric on $S^n$ is exactly the metric induced by the Fubini-Study form and standard complex structure). 
Given a point $q \in D_j^{\R}$, we can holomorphically identify a neighbourhood of its image in $\CP{n}$ with a neighbourhood of $0$ in $\C^n$, in such a way that a neighbourhood of $q$ in $S^n$ gets identified with a neighbourhood of $0$ in $\R^n \subset \C^n$. 
We can furthermore arrange that the divisor $D^j$ corresponds to the first coordinate being $0$.

When we flow $\R^n$ by $J(\nabla f)$, the imaginary part of the first coordinate will be strictly positive (respectively negative) because $\nabla f$ is transverse to $D_j^{\R}$, in the positive (respectively negative) direction.
Therefore the first component can not be zero, so the image avoids $D_j$.
\end{proof}

\begin{defn}
\label{defn:f}
Let $g: \R \To \R$ be a smooth function so that 
\begin{enumerate}
\item $g'(x)>0$;
\item $g(-x)=-g(x)$;
\item $g(x) = x$ for $|x| < \delta$;
\item $g'(x)$ is a strictly decreasing function of $|x|$ for $|x|>\delta$;
\item $g'(x) < \delta$ for $|x| > 2\delta$,
\end{enumerate}
where $0<\delta \ll 1$ (see Figure \ref{fig:g}).
\begin{figure}
\centering
\includegraphics[width=0.5\textwidth]{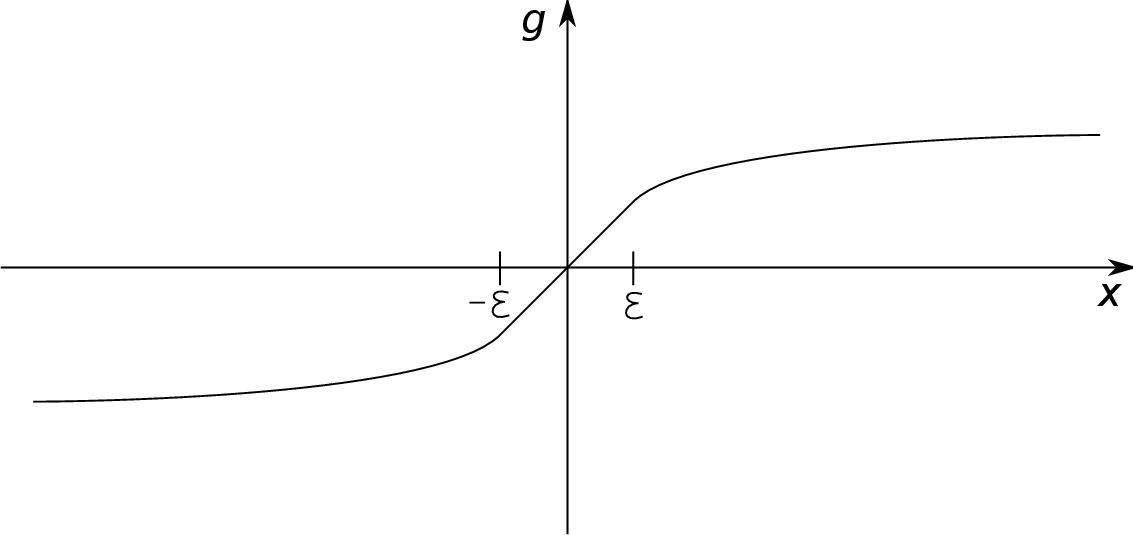}
\caption{The function $g$.}
\label{fig:g}
\end{figure}
We define $f: S^n \To \R$ by restricting the function
\begin{eqnarray*}
\tilde{f}: \R^{n+2} & \To & \R \\
\tilde{f}(x_1,\ldots,x_{n+2}) &=& \sum_{j=1}^{n+2} g(x_j),
\end{eqnarray*}
recalling that $S^n$ sits inside $\R^{n+2}$ as above.
\end{defn}

\begin{lem}
\label{lem:defnf}
$\nabla f$ is transverse to all of the hypersurfaces $D_j^{\R}$ in a positive sense.
\end{lem}
\begin{proof}
One can compute that $\nabla f$ is the projection of the vector
\[ \sum_{j=1}^{n+2} f_j \del{}{x_j} \in T \R^{n+1}\]
to $T S^n$, where $\R^{n+1} = \{\sum_j x_j = 0\} \subset \R^{n+2}$ and
\[ f_j := g'(x_j) - \frac{\sum_{k=1}^{n+2} g'(x_k)}{n+2}. \]
By the construction of $g$, one can check that $f_j > 0$ whenever $|x_j| < \delta$. 
The result follows.
\end{proof}

\begin{defn}
Let $L^n_{\epsilon}: S^n \To \CP{n}$ be the graph of $\epsilon df$ in $\CP{n}$, for $\epsilon>0$ sufficiently small.
Note that it lies in $\mathcal{P}^n$ by Lemmas \ref{lem:transdiv} and \ref{lem:defnf}, and is Lagrangian because it is the graph of an exact one-form.
We will frequently fix an $\epsilon$ and write $L^n$.
\end{defn}

\begin{rmk}
\label{rmk:propln}
$L^n$ is $S_{n+2}$-invariant (because $f$ and our Weinstein neighbourhood are).
Furthermore, because $f(-x) = -f(x)$, $df$ is invariant under the $\Z_2$-action 
\[(x,\alpha) \mapsto (a(x), -a^* \alpha)\]
where $a: S^n \To S^n$ is the antipodal map.
Recall that complex conjugation $\tau: \mathcal{P}^n \To \mathcal{P}^n$ induces the $\Z_2$-action $(x,\alpha) \mapsto (x,-\alpha)$ in $D^*_{\eta}S^n$, so $\tau \circ L^n = L^n \circ a$. 
In other words, the image of $L^n$ is preserved by complex conjugation, but it acts via the antipodal map on the domain $S^n$.
\end{rmk}

\begin{prop}
\label{prop:convamoe}
Define the maps
\begin{eqnarray*}
 \iota_{\epsilon}: S^n &\To& M_{\R}/2\pi M,\\
 \iota_{\epsilon} & := & \Arg \circ L^n_{\epsilon},
 \end{eqnarray*}
and
\[ q: \partial Z_n \To M_{\R}/2\pi M\]
(the standard inclusion). 
Then there exist homotopy equivalences $p_{\epsilon}:S^n \To \partial Z_n$, defined for $\epsilon>0$ sufficiently small, such that 
\[ \lim_{\epsilon \To 0} \| \iota_{\epsilon} - q \circ p_{\epsilon} \|_{C^0} = 0.\]
In other words, $\iota_{\epsilon}$ converges absolutely, modulo reparametrisation, to $\partial Z_n$.
\end{prop}
\begin{proof}
We consider a cellular decomposition of $S^n$ which is dual to the cellular decomposition induced by the hypersurfaces $D_j^{\R}$, and is isomorphic to the cellular decomposition of $\partial Z_n$ defined in Definition \ref{defn:qncells}. 
We will show that the image of each cell in the decomposition, under $\iota_{\epsilon}$, converges to the corresponding cell in $\partial Z_n$.

\begin{defn}
\label{defn:sncells}
We define a cellular decomposition of $S^n$ whose cells are indexed by triples of subsets $J,K,L \subset [n+2]$ such that
\begin{itemize}
\item $J \sqcup K \sqcup L = [n+2]$;
\item $J \neq \phi$ and $K \neq \phi$.
\end{itemize}
Namely, we define the cell
\[ V_{JKL} := \left\{ (x_1,\ldots,x_{n+2}) \in S^n: x_j = \max_i\{x_i\} \mbox{ for all $j \in J$}, x_k = \min_i\{x_i\} \mbox{ for all $k \in K$}\right\}\]
(this is dual to the cellular decomposition with cells
\[ W_{JKL} := \left\{ x_j \ge 0 \mbox{ for }j \in J, x_k \le 0 \mbox{ for }k \in K, \mbox{ and }x_l = 0\mbox{ for }l \in L \right\},\]
induced by the hypersurfaces $D_j^{\R}$).
\end{defn}

We now have
\[ \mathrm{dim}(V_{JKL}) = |L|,\]
and $V_{J'K'L'}$ is part of the boundary of $V_{JKL}$ if and only if
\[ J \subseteq J', K \subseteq K' ,\mbox{ and } L \supsetneq L'.\]
Thus, this cellular decomposition is isomorphic to that of $\partial Z_n$ by cells $U_{JKL}$, described in Definition \ref{defn:qncells}.
See Figure \ref{fig:dualcells} for the picture in the case $n=2$.

Our Lagrangian is obtained from the immersion $S^n \To \CP{n}$ by pushing off with the vector field $J(\nabla f)$. 
Thus, by Lemma \ref{lem:defnf}, it is approximately equal (to order $\epsilon^2$) to the composition of the map
\begin{eqnarray*}
S^n & \To & \left\{ \sum_j z_j = 0, z_j \neq 0\right\} \subset \C^{n+2} \\
(x_1,\ldots,x_{n+2}) & \mapsto & (x_1,\ldots,x_{n+2}) + i \epsilon (f_1,\ldots,f_{n+2})
\end{eqnarray*}
with the projection to $\CP{n} \setminus D = \mathcal{P}^n$. 
Thus we have
\begin{eqnarray*}
\iota_{\epsilon}(x_1,\ldots,x_n) & = & (\arg(x_1 + i\epsilon f_1), \ldots, \arg(x_{n+2} + i\epsilon f_{n+2})) + \mathcal{O}(\epsilon^2).
\end{eqnarray*}

Now, when $|x_l|$ is sufficiently large, we have 
\[ \arg(x_l + i \epsilon f_l + \mathcal{O}(\epsilon^2)) \approx \arg(x_l) = 0\mbox{ or }\pi.\]
When $|x_l|$ is sufficiently small, we have
\[ \arg(x_l + i \epsilon f_l + \mathcal{O}(\epsilon^2)) \in (0,\pi)\]
because $f_l > 0$ (by Lemma \ref{lem:defnf}). 
More precisely, we have the following:

\begin{lem}
\label{lem:args}
If we choose $\epsilon>0$ sufficiently small, then we have:
\begin{itemize}
\item If $|x_l| \ge \sqrt{\epsilon}$, then $\arg(x_l + i\epsilon f_l + \mathcal{O}(\epsilon^2)) = \arg(x_l) + \mathcal{O}(\sqrt{\epsilon})$, where $\arg(x_l) = 0$ or $\pi$; 
\item If $|x_l| \le \sqrt{\epsilon}$, then $\arg(x_l + i\epsilon f_l + \mathcal{O}(\epsilon^2)) \in (0,\pi)$, because $f_l$ is strictly positive for $|x_l|$ sufficiently small (by Lemma \ref{lem:defnf}).
\end{itemize}
\end{lem}

Observe that, on the cell $V_{JKL}$, we have
\[ x_j \ge \sqrt{\epsilon} \mbox{ for $j \in J$, and } x_k \le -\sqrt{\epsilon} \mbox{ for $k \in K$,}\]
because $\sum_l x_l^2 = 1$ and $\sum_l x_l = 0$.
Therefore, by Lemma \ref{lem:args},
\begin{eqnarray*}
\arg(x_j + i\epsilon f_j) + \mathcal{O}(\epsilon^2) &=& \mathcal{O}(\sqrt{\epsilon}) \mbox{ for $j \in J$, }\\
\arg(x_k + i\epsilon f_k) + \mathcal{O}(\epsilon^2) &=& \pi +  \mathcal{O}(\sqrt{\epsilon}) \mbox{ for $k \in K$, and}\\
\arg(x_l + i\epsilon f_l) + \mathcal{O}(\epsilon^2) & \in & (0,\pi) + \mathcal{O}(\sqrt{\epsilon}) \mbox{ for $l \in L$.}
\end{eqnarray*}
It follows that $\iota_{\epsilon}(V_{JKL})$ lies in an $\mathcal{O}(\sqrt{\epsilon})$-neighbourhood of $U_{JKL}$.

We are now able to define the map 
\[p_{\epsilon}: S^n \To \partial Z_n\]
to be a cellular map which identifies the cellular decompositions $V_{JKL}$ and $U_{JKL}$ (hence is a homotopy equivalence), and such that 
\[\| \iota_{\epsilon} - q \circ p_{\epsilon} \|_{C^0} = \mathcal{O}(\sqrt{\epsilon}).\]
We assume inductively that a map with these properties has been defined on all cells of dimension $<d$, then extend it to the cells of dimension $d$ relative to their boundaries.
\end{proof}

Now observe that, because $f(-x) = -f(x)$, $df(-x) = -df(x)$ (identifying tangent spaces by the antipodal map), so the only points where $L^n$ has a self-intersection are where $df=0$, i.e., critical points of $f$. A self-intersection point looks locally like the intersection of the graph of $df$ with the graph of $-df$, which is transverse because $f-(-f) = 2f$ is Morse.
We will now describe the critical points and Morse flow of $f$.

\begin{lem}
\label{lem:fcells}
If $x_j>x_k\ge0$, then
\[ (\nabla f) \left(\frac{x_k}{x_j}\right)>0.\]
Similarly, if $x_j<x_k\le 0$, then
\[ (\nabla f) \left(\frac{x_k}{x_j}\right) < 0.\]
\end{lem}
\begin{proof}
We prove the first statement.
If $x_j>x_k \ge 0$ then, by the construction of $g$, $g'(x_j) < g'(x_k)$. 
It follows that $f_j < f_k$, and hence that $f_j x_k < f_k x_j$, using the notation from the proof of Lemma \ref{lem:defnf}. 
Thus,
\begin{eqnarray*} 
(\nabla f) \left(\frac{x_k}{x_j}\right) & = & \sum_{l=1}^{n+2} f_l \del{}{x_l} \left( \frac{x_k}{x_j} \right) \mbox{ (since $x_k/x_j$ is constant in the radial direction)} \\
&=& \frac{1}{x_j^2}  (f_k x_j - f_j x_k ) \\
& > & 0.
\end{eqnarray*}
The proof of the second statement is similar.
\end{proof}

\begin{cor}
\label{cor:critf}
There is one critical point $p_K$ of $f$ for each proper, non-empty subset $K \subset [n+2]$, defined by
\[ V_{\bar{K},K,\phi} = \{p_K\}.\]
Explicitly, $p_K$ has coordinates (recalling $\sum_j x_j = 0$) 
\[ x_j = \left\{ \begin{array}{ll}
			-\frac{1}{|K|} & j \in K,\\
			+\frac{1}{|\bar{K}|} & j \in \bar{K},
			\end{array}\right.
\]
up to a positive rescaling so that $\sum_j x_j^2 = 1$.
Observe that $\Arg$ maps $p_K$ to the vertex $\pi e_K$ of $\partial Z_n$.
\end{cor}
\begin{proof}
Critical points of $f$ cannot lie on the hypersurfaces $D_j^{\R}$, since $\nabla f$ is transverse to the hypersurfaces.
Suppose that $x_j>x_k>0$. 
Then by Lemma \ref{lem:fcells},
\[ (\nabla f) \left(\frac{x_k}{x_j}\right) > 0,\]
so $\nabla f \neq 0$.
Hence, at a critical point of $f$, all positive coordinates $x_j$ are equal.
By a similar argument, all negative coordinates are equal.
It follows that the points $p_K$ are the only possiblities for critical points of $f$.

To prove that each $p_K$ is indeed a critical point, observe that by $S_{n+2}$ symmetry, the Morse flow of $f$ must preserve the equalities
\begin{eqnarray*}
x_k =x_l& & \mbox{ for all $k,l \in K$, and}\\
x_k = x_l && \mbox{ for all $k, l \in \bar{K}$}.
\end{eqnarray*}
The set of points satisfying these equalities is exactly $\{p_K,p_{\bar{K}}\}$, hence the Morse flow preserves these points.
Thus each $p_K$ is a critical point of $f$.
\end{proof}

\begin{lem}
\label{lem:morsecells}
Let $\phi: \R \times S^n \To S^n$ denote the flow of $\nabla f$ with respect to the round metric on $S^n$, so that $\phi(0,\cdot) = id$.
Given a proper, non-empty subset $K \subset [n+2]$, we define
\[\mathcal{S}(p_K) := \left\{q \in S^n: \lim_{t \To \infty} \phi(t,q) = p_K\right\}\subset S^n,\]
the stable manifold of $p_K$, and
\[\mathcal{U}(p_K) := \left\{q \in S^n: \lim_{t \To -\infty} \phi(t,q) = p_K\right\}\subset S^n,\]
the unstable manifold of $p_K$.
Then
\[ \mathcal{S}(p_K) = \{(x_1,\ldots,x_{n+2}) \in S^n: \{k \in [n+2]: x_k = \min_l\{x_l\} \} = K\}\]
and
\[ \mathcal{U}(p_K)  = \{ (x_1,\ldots,x_{n+2}) \in S^n: \{k\in [n+2]: x_k = \max_l\{x_l\} \} = \bar{K}\}.\]
\end{lem}
\begin{proof}
We prove the first statement. 
Suppose we are given $q = (x_1,\ldots,x_{n+2}) \in S^n$.
Let
\[ \lim_{t \To \infty} \phi(t,q) := p_{K},\]
and
\[K':= \{k\in [n+2]: x_k = \min_l\{x_l\}\}. \] 
We will show that $K=K'$.

First observe that, by $S_{n+2}$ symmetry, any equality of the form $x_j = x_k$ is preserved under the forward and backward flow of $\nabla f$.
Consequently any inequality of the form $x_j > x_k$ is also preserved under the (finite-time) flow.
It follows that $K' \subset K$.

We prove that $K \subset K'$ by contradiction: suppose that $j \notin K'$ but $j \in K$.
After flowing for some time, $x_j$ would have to be negative (in order to converge to $p_{K}$).
Then for any $k \in K'$ we would have $x_k<x_j<0$, so by Lemma \ref{lem:fcells} we have
\[ (\nabla f)\left( \frac{x_j}{x_k} \right) < 0.\]
Thus, the ratio $x_j/x_k$ is bounded above away from $1$, so even in the limit $t \To \infty$, $x_j$ can not approach the minimum value $x_k = \min_l\{x_l\}$.
This is a contradiction, hence $K \subset K'$.

Therefore $K=K'$. 
This completes the proof of the first statement.
The proof of the second statement is analogous.
\end{proof}

\begin{cor}
\label{cor:morseindf}
The critical point $p_K$ of $f$ has Morse index
\[ \mu_{\mathrm{Morse}}(p_K) = n+1-|K|.\]
\end{cor}
\begin{proof}
The Morse index of $p_K$ is the dimension of the stable manifold of $p_K$, which by Lemma \ref{lem:morsecells} is $n+1-|K|$.
\end{proof}

\begin{rmk}
Observe that, as a consequence of Lemma \ref{lem:morsecells}, 
\[V_{JKL} = \overline{ \mathcal{U}(\bar{J}) \cap \mathcal{S}(K)}\]
(see Figure \ref{fig:dualcells}).
\end{rmk}

\begin{figure}
\centering
\includegraphics[width=\textwidth]{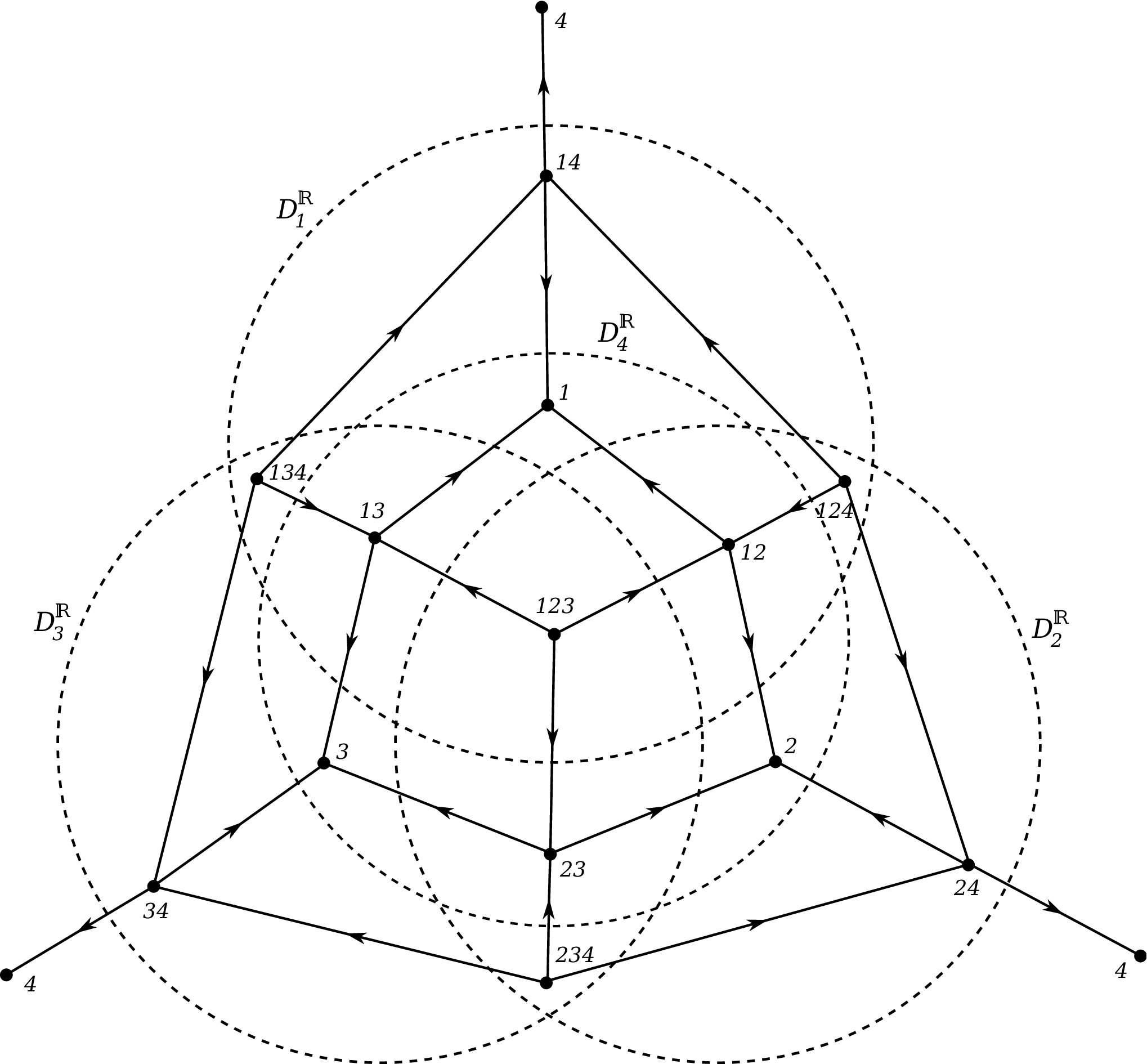}
\caption{The dual cell decompositions for $n=2$. The dashed circles represent the hypersurfaces $D_j^{\R}$ as labeled. Each region is labeled with the list of coordinates that are negative in that region (e.g., the label `$124$' means that $x_1<0,x_2<0,x_3>0,x_4<0$ in that region). The arrows represent the index-1 Morse flow lines of $\nabla f$. The dots represent critical points of $f$. The picture really lives on a sphere, and the three points labeled `$4$' should be identified (at infinity). 
Observe that the flowlines correspond to the edges of the polyhedron $\partial Z_2$, illustrated in Figure \ref{fig:c2}.}
\label{fig:dualcells}
\end{figure}

\section{The $A_{\infty}$ algebra $\mathcal{A} := CF^*(L^n,L^n)$}
\label{sec:A}

This section is concerned with the definition and properties of the $A_{\infty}$ algebra $\mathcal{A}^n := CF^*(L^n,L^n)$. 
We will simply write `$\mathcal{A}$' rather than `$\mathcal{A}^n$' unless we wish to draw attention to the dimension. 

In Section \ref{subsec:afuk}, we will explain why $L^n$, despite being an immersion rather than an embedding, can be regarded as an `extra' object of the Fukaya category of $\mathcal{P}^n$, as defined in \cite[Chapters 8 -- 12]{seidel08}.
This section can {\bf not} be read independently of that reference.
In Sections \ref{subsec:weights} -- \ref{subsec:signs}, we establish certain properties of $\mathcal{A}$.

\subsection{Including $L^n$ as an `extra' object of $\mathcal{F}uk(\mathcal{P}^n)$}
\label{subsec:afuk}

In \cite[Chapters 8 -- 12]{seidel08}, it is shown how to define the Fukaya category of a symplectic manifold $(X,\omega)$ with the following properties and structures:

\begin{itemize}
\item $\omega = d \theta$ is exact;
\item $X$ is equipped with an almost-complex structure $J_0$ in a neighbourhood of infinity, compatible with $\omega$;
\item $X$ is convex at infinity, in the sense that there is a bounded below, proper function $h: X \To \R$ such that
\[ \theta = - dh \circ J_0.\]
\end{itemize}

These assumptions are actually not quite the same as those in \cite{seidel08}, but the arguments and definitions work in the same way.

In particular, $X = \mathcal{P}^n$ has these properties: we equip it with the standard (integrable) complex structure $J_0$, then the restriction of the Fubini-Study form to $\mathcal{P}^n$ is given by $\omega = d \theta$, where $\theta = -dh \circ J_0$, and
\begin{eqnarray*}
h: \mathcal{P}^n & \To & \R, \\
 h([z_1:\ldots:z_{n+2}]) &=& \log \left( \frac{\sum_{j=1}^{n+2} |z_j|^2 }{\left(\prod_{j=1}^{n+2} |z_j|^2 \right)^{\frac{1}{n+2}} } \right)
 \end{eqnarray*}
is proper and bounded below.

With this data, the Fukaya category of compact, exact, embedded, oriented Lagrangians $L$ can be defined over a field of characteristic $2$, and with $\Z_2$ gradings (the `preliminary' Fukaya category of \cite[Chapters 8, 9]{seidel08}).
If $X$ is furthermore equipped with a complex volume form $\eta$ (note: we will not take a quadratic complex volume form as in \cite{seidel08}, because we assume our Lagrangians to be oriented), then the Fukaya category of compact, exact, embedded, oriented Lagrangian branes $L^{\#}$ can be defined over $\C$, and the $\Z_2$ grading can be lifted to a $\Z$ grading. 

We define the Fukaya category of $\mathcal{P}^n$ to include an `extra' object corresponding to the Lagrangian immersion $L^n: S^n \To \mathcal{P}^n$. 

\begin{rmk}
A theory of Lagrangian Floer cohomology for immersed Lagrangians has been worked out in \cite{joyce08} using Kuranishi structures, but we will give a definition that is compatible with the definition of \cite{seidel08} using explicit perturbations, with the aim of using it to make computations in Section \ref{sec:Acalc}.
\end{rmk}

First, we note that $H^1(S^n) = 0$ for $n>1$, so $L^n$ is automatically exact (this is an additional restriction in the case $n=1$ -- see the caption to Figure \ref{fig:l1}).

Now we explain the modifications necessary to the definition of the (preliminary) Fukaya category given in \cite[Chapters 8, 9]{seidel08}, to include the object $L^n$. 

\begin{rmk}

We will not mention brane structures, orientations and gradings for the purposes of this Section \ref{subsec:afuk}, because they work exactly the same as in \cite[Chapters 11, 12]{seidel08}.
We observe that $H^1(S^n)=0$ for $n>1$, so $L^n$ admits a grading (the case $n=1$ is easily checked). 
$S^n$ is also spin, so $L^n$ admits a brane structure. 
These observations, together with the modifications described in this section that show we can include $L^n$ as an extra object of the preliminary Fukaya category, allow us to include $L^n$ as an extra object in the `full' ($\Z$-graded, with $\C$ coefficients) Fukaya category of $\mathcal{P}^n$.
\end{rmk}

\begin{defn}
We define an object $L$ of the (preliminary) Fukaya category to be an exact Lagrangian immersion
\[ L: N \To \mathcal{P}^n\]
of some closed, oriented $n$-manifold $N$ into $\mathcal{P}^n$, which is {\bf either} an embedding {\bf or} the Lagrangian immersion $L^n:S^n \To \mathcal{P}^n$.
\end{defn}

\begin{defn}
We define 
\[\mathcal{H} := C^{\infty}_c(\mathcal{P}^n,\R),\]
the space of smooth, compactly supported functions on $\mathcal{P}^n$ (the space of Hamiltonians), and $\mathcal{J}$, the space of smooth almost-complex structures on $\mathcal{P}^n$ compatible with $\omega$, and equal to the standard complex structure $J_0$ outside of some compact set.
\end{defn}

\begin{defn}
\label{defn:gen}
For each pair of objects $(L_0,L_1)$, we define a Floer datum $(H_{01},J_{01})$ consisting of 
\[H_{01} \in C^{\infty}([0,1],\mathcal{H}) \mbox{ and }J_{01} \in C^{\infty}([0,1],\mathcal{J})\] 
satisfying the following property: if $\phi^t$ denotes the flow of the Hamiltonian vector field of the (time-dependent) Hamiltonian $H_{01}$, then the image of the time-$1$ flow $\phi^1 \circ L_0$ is transverse to $L_1$. 
One then defines a {\bf generator} of $CF^*(L_0,L_1)$ to be a path $y:[0,1] \To \mathcal{P}^n$ which is a flowline of the Hamiltonian vector field of $H_{01}$, together with a pair of points $(\tilde{y}_0,\tilde{y}_1) \in N_0 \times N_1$ such that $L_0(\tilde{y}_0) = y(0)$ and $L_1(\tilde{y}_1) = y(1)$.
One defines $CF^*(L_0,L_1)$ to be the $\C$-vector space generated by its generators. 
\end{defn}

The definition of a perturbation datum on a boundary-punctured disk with Lagrangian labels is the same as in \cite[Section 9h]{seidel08}. 

\begin{defn}
Given a perturbation datum on a boundary-punctured disk $S$ with Lagrangian boundary labels $(L_0,\ldots,L_k)$, some of which may be immersed, we define an {\bf inhomogeneous pseudo-holomorphic disk} to be a smooth map $u:S \To \mathcal{P}^n$ such that
\begin{itemize}
\item $u(C) \in \mathrm{im}(L_C)$ for each boundary component $C$ with label $L_C$, and
\item $u$ satisfies the perturbed holomorphic curve equation \cite[Equation (8.9)]{seidel08} with respect to the perturbation datum,
\end{itemize}
together with a continuous lift $\tilde{u}_C$ of the map $u|_C: C \To \mathrm{im}(L_C)$ to $N_C$:
\[\xymatrix{
& N_C \ar[d]^{L_C}\\
C\ar[ur]^{\tilde{u}_C} \ar[r]_-{u|_C} & \mathrm{im}(L_C)}
\]
for each boundary component $C$ with label $L_C: N_C \To \mathcal{P}^n$.
\end{defn}

\begin{rmk}
Note that the lift $\tilde{u}_C$ exists automatically if $L_C$ is an embedding. 
When $L_C$ is an immersion, the existence of $\tilde{u}_C$ tells us that the boundary map $u|_C$ does not `switch sheets' of the immersion along $C$.
\end{rmk}

\begin{defn}
\label{defn:ascon}
Given generators
\[ y_j \in CF^*(L_{j-1},L_j) \mbox{ for $j = 1,\ldots,k$, } ,\] 
and
\[y_0 \in CF^*(L_0,L_k),\]
we say that an inhomogeneous pseudo-holomorphic disk has {\bf asymptotic conditions} given by $(y_0,\ldots,y_k)$ if, on the strip-like end $\epsilon_j$ corresponding to the $j$th puncture, we have 
\begin{eqnarray*}
 \lim_{s \To + \infty} u(\epsilon_j(s,t)) &=& y_j(t), \\
 \lim_{s \To + \infty} \tilde{u}(\epsilon_j(s,0))& = & (\tilde{y}_j)_0, \mbox{and} \\
 \lim_{s \To + \infty} \tilde{u}(\epsilon_j(s,1))& = & (\tilde{y}_j)_1.\\
\end{eqnarray*}
(and the analogous condition with $s \To - \infty$ when $j = 0$).
We define the moduli space $\mathcal{M}_{S}(y_0,\ldots,y_k)$ to be the set of inhomogeneous pseudo-holomorphic disks with asymptotic conditions given by the generators $(y_0,\ldots,y_k)$.
\end{defn}

To show that $\mathcal{M}_S(y_0,\ldots,y_k)$ is a smooth manifold, we must modify the functional analytic framework of \cite[Section 8i]{seidel08} slightly. 
Namely, we fix $p>2$, and define a Banach manifold $\mathcal{B}_S(y_0,\ldots,y_k)$ as follows.

A point in $\mathcal{B}_S$ consists of:
\begin{itemize}
\item a map $u \in W^{1,p}_{loc}( S , \mathcal{P}^n)$, satisfying $u(C) \in \mathrm{im}(L_C)$;
\item continuous lifts $\tilde{u}_C$ of the continuous maps $u|_C: C \To \mathrm{im}(L_C)$ to $N_C$,
for each boundary component $C$ of $S$,
\end{itemize}
such that $u$ and $\tilde{u}_C$ are asymptotic to the generators $y_j$ along the strip-like ends, in the sense of Definition \ref{defn:ascon}.
Observe that $W^{1,p}$ functions are continuous at the boundary, so the lifting condition makes sense. 

Let $\bm{u} = (u,(\tilde{u}_C)) \in \mathcal{B}_S$ be represented by a smooth map. 
We define charts for the Banach manifold structure in a neighbourhood of $\bm{u}$. 
For each boundary component $C$ of $S$, we have a continuous Lagrangian embedding of vector bundles,
\[ TN_C \overset{(L_C)_*}{\lhook\joinrel\relbar\joinrel\relbar\joinrel\relbar\joinrel\rightarrow} (L_C)^*T \mathcal{P}^n.\]
Thus, we have a continuous Lagrangian embedding
\[ (\tilde{u}_C)^* TN_C \hookrightarrow (\tilde{u}_C)^* (L_C)^* T \mathcal{P}^n \cong (u^* T\mathcal{P}^n)|_C .\]
We define the tangent space to $\mathcal{B}_S(y_0,\ldots,y_k)$ at $\bm{u}$ to be the Banach space
\[T_{\bm{u}} \mathcal{B}_S(y_0,\ldots,y_k) := W^{1,p}(S, u^*T \mathcal{P}^n, \tilde{u}_C ^* TN_C)\]
(with the $W^{1,p}$-norm).
We choose an exponential map $\exp: T\mathcal{P}^n \To \mathcal{P}^n$ that makes the Lagrangian labels totally geodesic, and denote by $\widetilde{\exp}_N: TN \To N$ the corresponding exponential map on each Lagrangian label. 
We then define a map
\[ \phi_{\bm{u}} : T_{\bm{u}} \mathcal{B}_S \To \mathcal{B}_S\]
so that $\phi_{\bm{u}}(\xi)$ consists of the map $\exp(u,\xi(u))$, together with boundary lifts $\widetilde{\exp}_{N_C}(\tilde{u}_C, \xi(\tilde{u}_C))$.
This defines a chart of the Banach manifold structure in a neighbourhood of $\bm{u}$.

\begin{rmk}
Note that we can {\bf not} define a Banach manifold of locally $W^{1,p}$ maps from $S$ to $\mathcal{P}^n$, sending boundary component $C$ to $\mathrm{im}(L_C)$, then impose the lifting condition separately -- this would not define a Banach manifold because the image of $L_C$ may be singular (if $L_C = L^n$). 
\end{rmk}

We now define a Banach bundle $\mathcal{E}_S$ over $\mathcal{B}_S$, and a smooth section given by the perturbed $\bar{\partial}$-operator, as in \cite[Section 8i]{seidel08}. 
The section is Fredholm, because its linearization is a Cauchy-Riemann operator with totally real boundary conditions given by $\tilde{u}_C^* TN_C$. 
Thus, assuming regularity, the moduli space $\mathcal{M}_S(y_0,\ldots,y_k)$ is a smooth manifold with dimension equal to the Fredholm index. 
We can extend these arguments to show that the moduli space $\mathcal{M}_{\mathcal{S}^{k+1}}(y_0,\ldots,y_k)$ of inhomogeneous pseudo-holomorphic disks with arbitrary modulus is also a smooth manifold.

Finally, we must check that Gromov compactness holds. 
The author is not aware of a proof of Gromov compactness with immersed Lagrangian boundary conditions in the literature, but we can give an ad hoc proof in our special case by passing to a cover of $\mathcal{P}^n$. 
Namely, by Corollary \ref{cor:cover}, there is a cover $\widetilde{\mathcal{P}}^n$ of $\mathcal{P}^n$ in which every lift $\widetilde{L}^n$ of $L^n$ is embedded, so all of our lifted boundary conditions are embedded Lagrangians. 
Any family of inhomogeneous pseudo-holomorphic disks in $\mathcal{P}^n$ lifts to a family in $\widetilde{\mathcal{P}}^n$.
Standard Gromov compactness for the family of lifted disks in $\widetilde{\mathcal{P}}^n$, with boundary on the embedded lifts of Lagrangians, implies compactness for the family in $\mathcal{P}^n$. 

Everything else works as in \cite{seidel08}, so this allows us to define the Fukaya category of $\mathcal{P}^n$ with the extra object $L^n$, and show that the $A_{\infty}$ associativity relations hold.

We now consider the $A_{\infty}$ algebra $\mathcal{A} = CF^*(L^n,L^n)$. 
We would like to choose the Floer datum for the pair $(L^n,L^n)$ so that the underlying vector space of $\mathcal{A}$ is as small as possible. 

\begin{lem}
\label{lem:h}
There exists a Hamiltonian $H \in \mathcal{H}$ such that $(L^n)^* H$ is a Morse function on $S^n$ with exactly two critical points, and $X_H|_{\mathrm{im}(L^n)}$ vanishes only at those critical points.
\end{lem}
\begin{proof}
First define $H$ in a neighbourhood of the self-intersections of $\mathrm{im}(L^n)$, in such a way that $X_H$ is transverse to both branches of the image. 
This defines $(L^n)^* H$ on a neighbourhood of the critical points $p_K$ of $f$ (see Corollary \ref{cor:critf}). 
This function can easily be extended to a Morse function on $S^n$ with the desired properties, then extended to a neighbourhood of $\mathrm{im}(L^n)$, then to all of $\mathcal{P}^n$ using a cutoff function.
\end{proof}

\begin{cor}
\label{cor:arank}
For an appropriate choice of Floer datum, $CF^*(L^n,L^n)$ has generators $p_K$ indexed by all subsets $K \subset [n+2]$.
\end{cor}
\begin{proof}
We scale the $H$ of Lemma \ref{lem:h} so that it is $\ll \epsilon$ (the parameter in the definition of $L^n = L^n_{\epsilon}$), and use it as the Hamiltonian part of our Floer datum for $(L^n,L^n)$. Let $X_H$ denote the corresponding Hamiltonian vector field. 
Now if $\phi^1$ is the time-$1$ flow of $X_H$, we can arrange that $\phi^1(L^n(p)) = L^n(q)$ if and only if {\bf either} 
\[ p=q \mbox{ and } X_H(L^n(p)) = 0,\]
{\bf or}
\[(p,q) \mbox{ corresponds to a pair $(p',q')$ such that $p' \neq q'$ and $L^n(p') = L^n(q')$}\]
(note that the assumption that $H \ll \epsilon$ ensures that the transverse self-intersections $L^n(p') = L^n(q')$ persist under the flow of one branch of $L^n$ by $X_H$).

In the first case, we get generators corresponding to the critical points of the Morse function $(L^n)^* H$. 
We denote the generator corresponding to the minimum, respectively maximum, by $p_{\phi}$, respectively $p_{[n+2]}$. 
In the second case, we get generators corresponding to pairs $(p',q') = (p_K,p_{\bar{K}})$ where $K \subset [n+2]$ is proper and non-empty, by Corollary \ref{cor:critf}. 
We denote the generator corresponding to $(p_K,p_{\bar{K}})$ by $p_{K}$, by slight abuse of notation.
\end{proof}

\subsection{Weights in $M$}
\label{subsec:weights}

\begin{defn}
\label{defn:weight}
(Compare \cite[Section 8b]{seidel03}) Whenever we have an immersed Lagrangian $L:N \To X$ (such that the image of $H_1(N)$ in $H_1(X)$ is trivial), we can assign a weight $w(y) \in H_1(X)$ to each generator $y$ of $CF^*(L,L)$.
Namely, choose a path from $\tilde{y}_1$ to $\tilde{y}_0$ in $N$, and define $w(y)$ be the homology class obtained by composing the image of this path in $X$ with the path $y$ (see Definition \ref{defn:gen}).
\end{defn}

\begin{prop}
\label{prop:weight}
In our case, we have
\[w(p_K) = e_K \in M \cong H_1(\mathcal{P}^n).\]
\end{prop}
\begin{proof}
By Proposition \ref{prop:arghomeq} and Proposition \ref{prop:convamoe}, $\Arg$ induces a homotopy equivalence between $(\mathcal{P}^n, L^n)$ and $(\mathcal{C}^n, \partial (\pi Z_n))$.
Thus, when $K$ is proper and non-empty, $w(p_K)$ is the class of a path from $\pi e_{\bar{K}}$ to $\pi e_K$ in $H_1(\mathcal{C}^n) \cong M$, which is exactly $e_K$. 
When $K = \phi$ or $[n+2]$ it is clear that $w(p_K) = 0$.
\end{proof}

\begin{cor}
\label{cor:cover}
There exists a finite cover $\widetilde{\mathcal{P}}^n \To \mathcal{P}^n$ in which every lift $\widetilde{L}^n$ of $L^n$ is embedded.
\end{cor}
\begin{proof}
Recall that $\pi_1(\mathcal{P}^n) \cong M$ by Corollary \ref{cor:pi1}.
Consider the group homomorphism
\begin{eqnarray*}
 \rho: M &\To& \Z_{n+2} \\
 \rho(u) &=& e_{[n+2]} \cdot u
\end{eqnarray*}
(this is well defined because $\rho(e_{[n+2]}) \equiv 0 \mbox{ (mod $(n+2)$)}$).
There is a corresponding $(n+2)$-fold cover of $\mathcal{P}^n$, and we have
\[ \rho(w(p_K)) = \rho(e_K) = |K| \neq 0 \mbox{ (mod $(n+2)$)}\]
for all proper non-empty $K \subset [n+2]$, so the two lifts of $L^n$ coming together at an intersection point are distinct.
\end{proof}

\begin{prop}
\label{prop:top}
The $A_{\infty}$ structure maps $\mu^k$ are homogeneous with respect to the weight $w$.
In other words, the coefficient of $p_{K_0}$ in $\mu^k(p_{K_1},\ldots,p_{K_k})$ is non-zero only if
\[ \sum_{j=1}^k e_{K_j} = e_{K_0}.\]
\end{prop}
\begin{proof}
If the coefficient of $p_{K_0}$ in $\mu^k(p_{K_1},\ldots,p_{K_k})$ is non-zero, then there is a topological disk in $\mathcal{P}^n$ with boundary on the image of $L^n$, 
\[u:(D,\partial D) \To (\mathcal{P}^n,\mathrm{im}(L^n)),\]
whose boundary changes `sheets' of $L^n$ exactly at the self-intersection points $p_{K_0},p_{K_1},\ldots,p_{K_k}$ in that order (ignoring any appearance $p_{\phi}$ or $p_{[n+2]}$ on the list). 
This disk must lift to the universal cover, hence its boundary lifts to a loop in the universal cover.

The boundary always lies on lifts of $L^n$, which are indexed by the fundamental group $M$ (think of the homotopy-equivalent picture of $M_{\R} \setminus \{ \pi Z_n + 2\pi M\}$, with the lifts of $L^n$ being $\partial (\pi Z_n) + 2 \pi M$).
When the boundary changes sheets at a point $p_K$, the index of the sheet in $M$ changes by $w(p_K)$ (observe that the points $p_{\phi}$ and $p_{[n+2]}$, at which no sheet-changing occurs, have weight $0$).

Therefore, if the boundary of our disc changes sheets at $p_{K_0},p_{K_1},\ldots,p_{K_k}$, and comes back to the sheet it started on, we must have
\[ -w(p_{K_0}) + \sum_{j=1}^k w(p_{K_j}) = 0.\]
\end{proof}

\begin{cor}
\label{cor:tact}
The character group of $M$,
\[\mathbb{T} := \mathrm{Hom}(M,\C^*),\]
acts on $\mathcal{A}$ via
\[ \alpha \cdot p := \alpha(w(p)) p.\]
The $A_{\infty}$ structure on $\mathcal{A}$ is equivariant with respect to this action.
\end{cor}

\subsection{Grading}
\label{subsec:grad}

Recall that, to lift the $\Z_2$-grading on the Fukaya category to a $\Z$-grading, we must equip $\mathcal{P}^n$ with a complex volume form $\eta$. 
We assume that:
\begin{itemize}
\item $\eta$ is compatible with complex conjugation $\tau: \mathcal{P}^n \To \mathcal{P}^n$, in the sense that $\tau^* \eta = \bar{\eta}$;
\item $\eta$ extends to a meromorphic $(n,0)$-form on $\CP{n}$, with a pole of order $n_j$ along the divisor $D_j$ (with the usual convention that a zero of order $k$ is a pole of order $-k$). 
\end{itemize}
We set
\[ \bm{n} := \sum_{j=1}^{n+2} n_j e_j \in \widetilde{M}.\]
Observe that
\begin{eqnarray*}
\bm{n} \cdot e_{[n+2]} &=& \sum_{j=1}^{n+2} n_j \\
 &=& \mathrm{deg}(K_{\CP{n}}) \\
&=& n+1.
\end{eqnarray*} 
Observe that there is no canonical choice for $\eta$, so our $\Z$-grading will not be canonical.

\begin{prop}
\label{prop:grading}
The $\Z$-grading on $\mathcal{A}$ defined by $\eta$ is
\[i(p_K) = (2\bm{n} - e_{[n+2]}) \cdot e_K.\]
In other words, the coefficient of $p_{K_0}$ in $\mu^k(p_{K_1},\ldots,p_{K_k})$ is non-zero only if
\[  i(p_{K_0}) = 2-k + \sum_{j=1}^k i(p_{K_j}).\]
\end{prop}
\begin{proof}
Recall that the volume form $\eta$ defines a function
\[ \psi: \mathrm{Gr} (T \mathcal{P}^n) \To S^1,\]
where $\mathrm{Gr} (T \mathcal{P}^n)$ is the Lagrangian Grassmannian of $\mathcal{P}^n$ (i.e., the fibre bundle over $\mathcal{P}^n$ whose fibre over a point $p$ is the set of Lagrangian subspaces of $T_p \mathcal{P}^n$). 
If $V \subset T_p \mathcal{P}^n$ is a Lagrangian subspace, then $\psi (V)$ is defined by choosing a real basis $v_1,\ldots,v_n$ for $V$ and defining
\[ \psi(V) := \arg(\eta(v_1,\ldots,v_n)).\]

A grading on $L^n$ is a function $\alpha^{\#}: S^n \To \R$ such that
\[ \pi \alpha^{\#}(x) = \psi(L^n_*(T_x S^n)) \]
(see \cite{seidel99}). 
Recall from the construction of $L^n$ that, away from the hypersurfaces $D_j^{\R}$, the immersion $L^n:S^n \To \CP{n}$ is close to the double cover of the real locus, $\iota: S^n \To \CP{n}$.
So away from the hypersurfaces $D_j^{\R}$, 
\[\psi(L^n_*(T_x S^n)) \approx \psi(\iota_*(T_x S^n)) = 0 \mbox{ or }\pi,\]
because we assumed $\eta$ was invariant under complex conjugation, so $\psi(T\RP{n})$ is real. 
Therefore, away from the hypersurfaces $D_j^{\R}$, $\alpha^{\#}$ is approximately an integer.

The hypersurfaces $D_j^{\R}$ split $S^n$ into regions $S^n_K$ indexed by proper non-empty subsets $K \subset [n+2]$.
Namely, $S^n_K$ is the region where $x_j<0$ for $j \in K$ and $x_j > 0$ for $j \notin K$, and contains the unique critical point $p_K$ of $f$.
Suppose that $\alpha^{\#} \approx \alpha^{\#}_K \in \Z$ in the region $S^n_K$.

How does $\alpha^{\#}_K$ change as we cross a hypersurface $D_j^{\R}$? 
Let $p$ be a point on $D_j^{\R}$, away from the other hypersurfaces $D_k^{\R}$. 
Let us choose a holomorphic function $q$ in a neighbourhood of $\iota(p)$ in $\CP{n}$, compatible with complex conjugation (i.e., $q(\tau(z)) = \overline{q(z)}$), and such that $D_j = \{q = 0\}$.
Because $\eta$ has a pole of order $n_j$ along $D_j$, we have
\[ \eta = q^{-n_j} \eta',\]
where $\eta'$ is a holomorphic volume form compatible with complex conjugation.

In the same way that $\eta$ defines the function $\psi$, $\eta'$ defines a function 
\[ \psi': \mathrm{Gr} (T \CP{n}) \To S^1\]
in a neighbourhood of $\iota(p)$. 
Whereas $\psi$ is not defined on $D_j$, because $\eta$ has a pole there, the function $\psi'$ is defined and continuous on $D_j$, because $\eta'$ is holomorphic.

We have
\[ \psi = \psi' + \arg(q^{-n_j})\]
away from $D_j$. 
We can define real functions $\beta^{\#}_{\epsilon}$ on a neighbourhood of $p$ in $S^n$, for $\epsilon \ge 0$ sufficiently small, so that
\begin{eqnarray*}
\pi \beta_{\epsilon}^{\#}(x) &=& \psi'((L^n_{\epsilon})_*(T_x S^n)).
\end{eqnarray*}
Because $L^n_0 = \iota$, and $\eta'$ is compatible with complex conjugation, $ \beta_{0}^{\#}$ is a constant integer.
Furthermore, away from $D_j^{\R}$, $L^n_{\epsilon} \approx \iota$, so $\beta_{\epsilon}^{\#} \approx \beta_0^{\#}$. 
It follows that $\beta_{\epsilon}^{\#}$ approximately does not change as we cross $D_j^{\R}$. 
So the change in $\alpha^{\#}_K$ as we cross the hypersurface $D_j^{\R}$ comes only from the term $\arg(q^{-n_j})$.

We saw in Proposition \ref{prop:convamoe} that $\Arg \circ L^n$ approximates the boundary of the zonotope $Z^n$. 
Thus, as we cross $D_j^{\R}$, moving from $S^n_K$ to $S^n_{K \sqcup \{j\}}$, $\Arg \circ L^n$ changes from $\pi e_K$ to $\pi e_{K \sqcup \{j\}}$, changing by $\pi e_j$.
It follows that $\arg(q^{-n_j})$ decreases by $ \pi n_j$. 
Therefore, $\alpha^{\#}$ approximately decreases by $n_j$.
So we may assume that 
\[\alpha^{\#}_K = -\bm{n} \cdot e_K.\]

To calculate the index of the generator $p_K$, we observe that the two sheets of $L^n$ that meet at $p_K$ are locally the graphs of the exact $1$-forms $df$ and $-df$.
It follows by \cite[2d(v)]{seidel99} that the obvious path connecting the tangent spaces of the two sheets in the Lagrangian Grassmannian has Maslov index equal to the Morse index 
\[ \mu_{\mathrm{Morse}}(p_K) = n+1 - |K| \mbox{ (see Corollary \ref{cor:morseindf}).}\]
We also need to take into account the grading shift of $\alpha^{\#}_K - \alpha^{\#}_{\bar{K}}$ between the two sheets.
Using \cite[2d(ii)]{seidel99}, we have
\begin{eqnarray*} 
i(p_K) &=& \mu_{\mathrm{Morse}}(p_K) -\alpha^{\#}_K + \alpha^{\#}_{\bar{K}} \\
 &=& n+1-|K| + \bm{n} \cdot e_K - \bm{n} \cdot e_{\bar{K}} \\
&=& n+1 - e_{[n+2]} \cdot e_K + \bm{n} \cdot (e_K - e_{[n+2]} + e_K)\\
&=& (2 \bm{n} - e_{[n+2]}) \cdot e_K \mbox{ (since $\bm{n} \cdot e_{[n+2]} = n+1$).}
\end{eqnarray*}
We also note that this equation works for $p_{\phi}$ and $p_{[n+2]}$, which have their usual gradings of $0$ and $n$ respectively.

The dimension formula for moduli spaces of holomorphic polygons now says that the dimension of the moduli space of $(k+1)$-gons with boundary on $L^n$, a positive puncture at $p_{K_0}$, and negative punctures at $p_{K_1},\ldots,p_{K_k}$ is
\[ \mathrm{dim}(\mathcal{M}_{\mathcal{S}}(p_{K_0},\ldots,p_{K_k})) = k-2+ i(p_{K_0}) - \sum_{j=1}^k i(p_{K_j}).\]
Since we are counting the $0$-dimensional component of the moduli space to determine our $A_{\infty}$ structure coefficients, this dimension should be $0$.
This proves the stated formula, i.e., that $i$ defines a valid $\Z$-grading on $\mathcal{A}$.

We also observe that $i$ lifts the $\Z_2$-grading: the two sheets of $L^n$ that meet at $p_K$ are locally the graphs of the exact $1$-forms $df$ and $-df$, hence the sign of the intersection is
\[ n+1 + \mu_{\mathrm{Morse}}(p_K) \equiv |K| \equiv (2 \bm{n} - e_{[n+2]}) \cdot e_K \mbox{ (mod $2$)}.\]
\end{proof}

\begin{cor}
\label{cor:grading}
The $A_{\infty}$ structure on $\mathcal{A}$ admits the fractional grading
\[|p_K| := \frac{n}{n+2}|K| \in \Q,\]
in the sense that the coefficient of $p_{K_0}$ in $\mu^k(p_{K_1},\ldots,p_{K_k})$ is non-zero only if
\[ 2-k + \sum_{j=1}^k \frac{n}{n+2}|K_j| = \frac{n}{n+2}|K_0|.\]
\end{cor}
\begin{proof}
For any such non-zero product, we have
\[ -e_{K_0} + \sum_{j=1}^k e_{K_j} = q e_{[n+2]}\]
for some $q \in \Z$ (Proposition \ref{prop:top} says that the image of this sum in $M$ is $0$, hence it is a multiple of $e_{[n+2]}$ in $\widetilde{M}$).
It then follows from Proposition \ref{prop:grading} that
\begin{eqnarray*}
i(p_{K_0}) &=& 2-k+\sum_{j=1}^k i(p_{K_j}).
\end{eqnarray*}

Hence, we ought to have
\begin{eqnarray*}
 k-2 &=& (2 \bm{n} - e_{[n+2]}) \cdot \left(-e_{K_0} + \sum_{j=1}^k e_{K_j} \right) \\
&=& (2 \bm{n} - e_{[n+2]}) \cdot q e_{[n+2]}\\
&=& nq \mbox{ (since $\bm{n} \cdot e_{[n+2]} = n+1$)}\\
&=& \frac{n}{n+2} e_{[n+2]} \cdot q e_{[n+2]} \\
&=& \frac{n}{n+2} e_{[n+2]} \cdot \left(-e_{K_0} + \sum_{j=1}^k e_{K_j} \right)
\end{eqnarray*}
from which the result follows.
\end{proof}

\begin{cor}
\label{cor:prodzero}
The $A_{\infty}$ products $\mu^k$ are non-zero only when $k = 2+nq$ (where $q \in \Z_{\ge 0}$).
\end{cor}
\begin{proof}
This follows from the final set of equations in the proof of Corollary \ref{cor:grading}.
\end{proof}

\begin{rmk}
\label{rmk:2nqsum}
We observe that, when $k = 2 + nq$, we must also have
\[ \sum_{j=1}^{2+nq} e_{K_j} = e_{K_0} + q e_{[n+2]}\]
(note: this is an equation in $\widetilde{M}$, not $M$).
\end{rmk}

\begin{cor}
\label{cor:m1m2top}
$\mu^1$ is trivial, and
\begin{eqnarray*}
\mu^2(p_{K_1},p_{K_2}) &=& \left\{ \begin{array}{cl}
                                       a(K_1,K_2)p_{K_1 \sqcup K_2} & \mbox{if } K_1 \cap K_2 = \phi \\
                                       0 & \mbox{otherwise,}
\end{array}\right.
\end{eqnarray*}
where $a(K_1,K_2)$ are some integers.
\end{cor}
\begin{proof}
The fact that $\mu^1 = 0$ follows immediately from Corollary \ref{cor:prodzero}.

For the second part of the Proposition, suppose that the coefficient of $p_{K_0}$ in $\mu^2(p_{K_1},p_{K_2})$ is non-zero. 
It follows from Proposition \ref{prop:top} that
\[ e_{K_1}+e_{K_2} = e_{K_0}\]
in $M$, and from Corollary \ref{cor:grading} that
\[ |K_1| + |K_2| = |K_0|.\]
Therefore $K_0 = K_1 \sqcup K_2$, and the result is proven.
\end{proof}

\subsection{Signs}
\label{subsec:signs}

The main aim of this section is to prove that the cohomology algebra of $\mathcal{A}$ is graded commutative. 
The basic reason for this is that complex conjugation $\tau: \mathcal{P}^n \To \mathcal{P}^n$ maps $L^n$ to itself. 
Given a holomorphic disk $u:S \To \mathcal{P}^n$ contributing to the product $a \cdot b$, the corresponding disk $\bar{u}:=\tau \circ u :\bar{S} \To \mathcal{P}^n$ (where $\bar{S}$ denotes the disk $S$ with the conjugate complex structure) contributes to the product $b \cdot a$ with the appropriate relative Koszul sign.

Throughout this section, we use the sign conventions of \cite{seidel08}.

\begin{defn}
\label{defn:op}
Given an $A_{\infty}$ category $\mathcal{C}$, we define its {\bf opposite category} $\mathcal{C}^{op}$ to be the category with the same objects, the `opposite' morphisms 
\[ \mathrm{hom}_{\mathcal{C}^{op}}(A,B) := \mathrm{hom}_{\mathcal{C}}(B,A),\]
and compositions defined by
\[ \mu^k_{op}(x_1,\ldots,x_k) := (-1)^{\ast} \mu^k(x_k,\ldots,x_1),\]
where
\[ \ast = \sum_{j<l}(i(x_j) + 1) \cdot (i(x_l) + 1).\]
It is an exercise to check that $\mathcal{C}^{op}$ is an $A_{\infty}$-category.
\end{defn}

The following proposition is due to \cite{sol10}, and is also proved in \cite[Appendix B]{Sheridan2017}:

\begin{prop}
\label{prop:opfuk}
Let $X = (X,\omega,\eta)$ be an exact symplectic manifold with boundary with symplectic form $\omega$, and complex volume form $\eta$.
Define $\bar{X} := (X,-\omega,\bar{\eta})$.
There is a quasi-isomorphism of $A_{\infty}$-categories
\[c^X:\mathcal{F}uk(X)^{op} \To \mathcal{F}uk(\bar{X}).\]
\end{prop}

Recall that we equip $L^n$ with a grading and Pin structure $P^\#$ to turn it into a Lagrangian brane $L^\#$. 
If $n>1$ then $P^\#$ is unique, but if $n=1$ there are two possible choices, and we choose $P^\#$ to be the \emph{non}-trivial Pin structure in that case.

Complex conjugation defines an isomorphism $\tau: \cP^n \xrightarrow{\sim} \bar{\cP}^n$. 
We have $\tau \circ L^n = L^n \circ a$, where $a: S^n \to S^n$ is the antipodal map, because the function $f: S^n \to \R$ is odd by construction. 
If $P^\#$ denotes our chosen Pin structure on $S^n$, then a choice of isomorphism $P^\# \cong a^* P^\#$ determines an isomorphism of Lagrangian branes, $j:L^\# \xrightarrow{\sim} \tau L^\#$. 
This determine an algebra isomorphism
\begin{align*}
\label{eqn:theiso} \Hom^*_{\fuk(\cP^n)}(L^\#,L^\#) & \xrightarrow{c^{\cP^n}} \Hom^*_{\fuk(\bar{\cP}^n)^{op}}(L^\#,L^\#) \\
& \xrightarrow{\tau} \Hom^*_{\fuk(\cP^n)^{op}}(\tau L^\#, \tau L^\#) \nonumber \\
& \xrightarrow{j} \Hom^*_{\fuk(\cP^n)^{op}}(L^\#,L^\#). \nonumber
\end{align*}

\begin{lem}[= Lemma \ref{lem:thesign}]
\label{lem:thesign2}
This isomorphism sends
\[ p_K \mapsto (-1)^{1 + \bm{n} \cdot e_K} \cdot p_K.\]
\end{lem}

\begin{rmk}
If $n=1$ and $P^\#$ is the trivial Pin structure, then Lemma \ref{lem:thesign2} is false.
\end{rmk}

\begin{cor}
\label{cor:cohsigns}
The cohomology algebra of $\mathcal{A}$, with the (associative) product
\[ p_{J} \cdot p_{K} := (-1)^{|K|} \mu^2(p_{J},p_{K}),\]
is supercommutative:
\[ p_{J} \cdot p_{K} = (-1)^{|J|\cdot|K|} p_{K} \cdot p_{J}.\]
\end{cor}
\begin{proof}
It follows from Lemma \ref{lem:thesign2}, together Corollary \ref{cor:m1m2top} and the definition of the opposite category, that whenever $J$ and $K$ are disjoint we have
\[ \mu^2(p_K, p_J) = (-1)^\dagger \cdot \mu^2(p_J, p_K),\]
where
\begin{align*}
 \dagger &= \left(1 + \bm{n} \cdot e_J\right) + \left(1 + \bm{n} \cdot e_K \right) + \left(1+\bm{n} \cdot e_{J \sqcup K}\right) + (1+|J|) \cdot (1+|K|) \\
 &= |J|\cdot |K| + |J| + |K|.
 \end{align*}
It follows that
\[ p_J \cdot p_K = (-1)^{|J|\cdot |K|} p_K \cdot p_J \]
as required.
\end{proof}

\section{A Morse-Bott definition of the Fukaya category}
\label{sec:fukdef}

The Fukaya $A_{\infty}$ category was introduced in \cite{fukaya93}. 
There are a number of approaches to transversality issues in its definition -- virtual perturbations are used in \cite{fooo}, and explicit perturbations of the holomorphic curve equation are used in \cite{seidel08}. 

In this section, we describe a `Morse-Bott' approach which is a modification of the approach in \cite{seidel08}, combining it with the approach of \cite{abouzaidplumb}. 
The outline of this approach has appeared in \cite[Section 7]{seidelg2}, and is related to the `clusters' of \cite{cornealalonde}. 
However, the geometric situation we consider is simpler than that of \cite{cornealalonde}, namely we work only in exact symplectic manifolds with convex boundary, which for example rules out disk and sphere bubbling.

Our treatment follows \cite[Sections 8 -- 12]{seidel08} closely, explaining at each stage how our construction differs. 
We make use of concepts and terminology from \cite{seidel08} (including abstract Lagrangian branes, strip-like ends and perturbation data) with minimal explanation.
We explain, in Section \ref{subsec:compat}, why our definition of the Fukaya category is quasi-equivalent to that given in \cite{seidel08}.

This section deals only with the Fukaya category of embedded Lagrangians.
In particular, the Lagrangian immersion $L^n: S^n  \To \mathcal{P}^n$ does not fit into this framework.
However, the concepts introduced in this section are the basis for the Morse-Bott computation of $\mathcal{A} = CF^*(L^n,L^n)$ that will be explained in Section \ref{subsec:fpearlyt}.

\subsection{The domain: pearly trees}
\label{subsec:pearlytrees}

In this section, we recall the Deligne-Mumford-Stasheff compactification of the moduli space of disks with boundary punctures, and define the analogous moduli space of pearly trees and its compactification.

Suppose that $k \ge 2$, and $\bm{L} := (L_0,\ldots,L_k)$ is a tuple of Lagrangians in $X$.
We denote by $\mathcal{R}(\bm{L})$ the moduli space of disks with $k+1$ boundary marked points,  modulo biholomorphism, with the components of the boundary between marked points labeled $L_0,\ldots,L_k$ in order. 
The marked point between $L_k$ and $L_0$ is `positive', and all other marked points are `negative'.
We call $\bm{L}$ a set of {\bf Lagrangian labels} for our boundary-marked disk (for the purposes of this section, it is not important that the labels correspond to Lagrangians in $X$ -- we need only assign certain labels to the boundary components and keep track of which of the labels are identical).

\begin{defn}
We denote by $\mathcal{S}(\bm{L}) \To \mathcal{R}(\bm{L})$ the universal family of boundary-punctured disks with Lagrangian labels $\bm{L}$, so that the fibre $\mathcal{S}_r$ over a point $r \in \mathcal{R}(\bm{L})$ is the corresponding disk, with its boundary marked points removed. 
\end{defn}

We define
\[ Z^{\pm} := \R^{\pm} \times [0,1]\]
with the standard complex structure (where $\R^+, \R^-$ are the positive and negative half-lines respectively). 
We will use $s$ to denote the $\R^{\pm}$ coordinate and $t$ to denote the $[0,1]$ coordinate.
We make a {\bf universal choice of strip-like ends} for the family $\mathcal{S}(\bm{L}) \To \mathcal{R}(\bm{L})$, which consists of fibrewise holomorphic embeddings 
\[ \epsilon_j: \mathcal{R}(\bm{L}) \times Z^{\pm} \To \mathcal{S}(\bm{L})\]
to a neighbourhood of the $j$th puncture, for each $j=0,1,\ldots,k$, where the sign $\pm$ is opposite to the sign of the puncture.

\begin{defn}
A {\bf directed $k$-leafed planar tree} $T$ is a directed tree with $k$ semi-infinite `incoming' edges and one semi-infinite `outgoing' edge, together with a proper embedding into $\R^2$.
Isotopic embeddings are regarded as equivalent.
We denote by $V(T)$ the set of vertices of $T$, by $E(T)$ the set of edges, and by $E_i(T) \subset E(T)$ the set of internal (compact) edges.
We say that $T$ has {\bf Lagrangian labels} $\bm{L}$ if the connected components of $\R^2 \setminus T$ are labeled by the Lagrangians of $\bm{L}$, in order. 
A Lagrangian labeling $\bm{L}$ of $T$ induces a labeling $\bm{L}_v$ of the regions surrounding each vertex $v \in V(T)$ (see Figure \ref{fig:treelabels}).
We call a vertex {\bf stable} if it has valence $\ge 3$, and {\bf semi-stable} if it has valence $\ge 2$. 
We call the tree $T$ stable (respectively semi-stable) if all of its vertices are stable (respectively semi-stable).
\end{defn}

\begin{figure}
\centering
\subfigure[A $k$-leafed stable tree $T_S$ is said to have Lagrangian labels $\bm{L}$ if the connected components of $\R^2 \setminus T$ are labeled by the Lagrangians of $\bm{L}$, in order. In this figure, $\bm{L} = (L_0,L_0,L_0,L_1,L_2,L_2,L_1,L_0,L_3)$.
A Lagrangian labeling $\bm{L}$ of $T_S$ induces a labeling $\bm{L}_v$ of the regions surrounding each vertex $v$. In this figure, the induced labeling of the regions surrounding the topmost vertex is $\bm{L}_v = (L_0,L_0,L_1,L_1,L_0,L_3)$.]{
\includegraphics[width=0.45\textwidth]{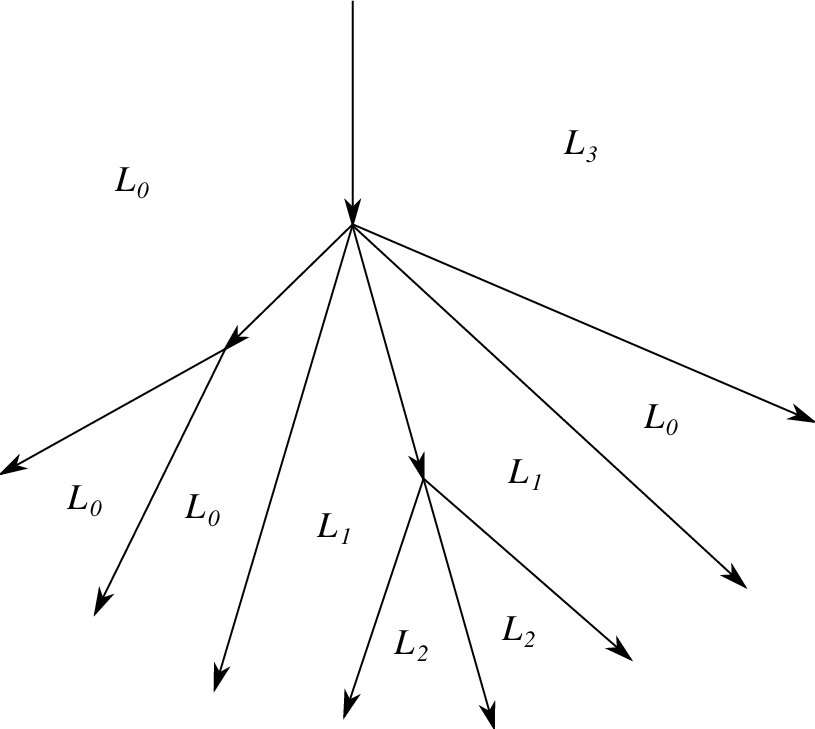}
\label{fig:treelabels}}
\hfill
\subfigure[A pearly tree $S$, with underlying tree $T_S$ and Lagrangian labels as in Figure \ref{fig:treelabels}. Observe that all edges have the same label on either side, while external strips have different labels on either side.]{
\includegraphics[width=0.45\textwidth]{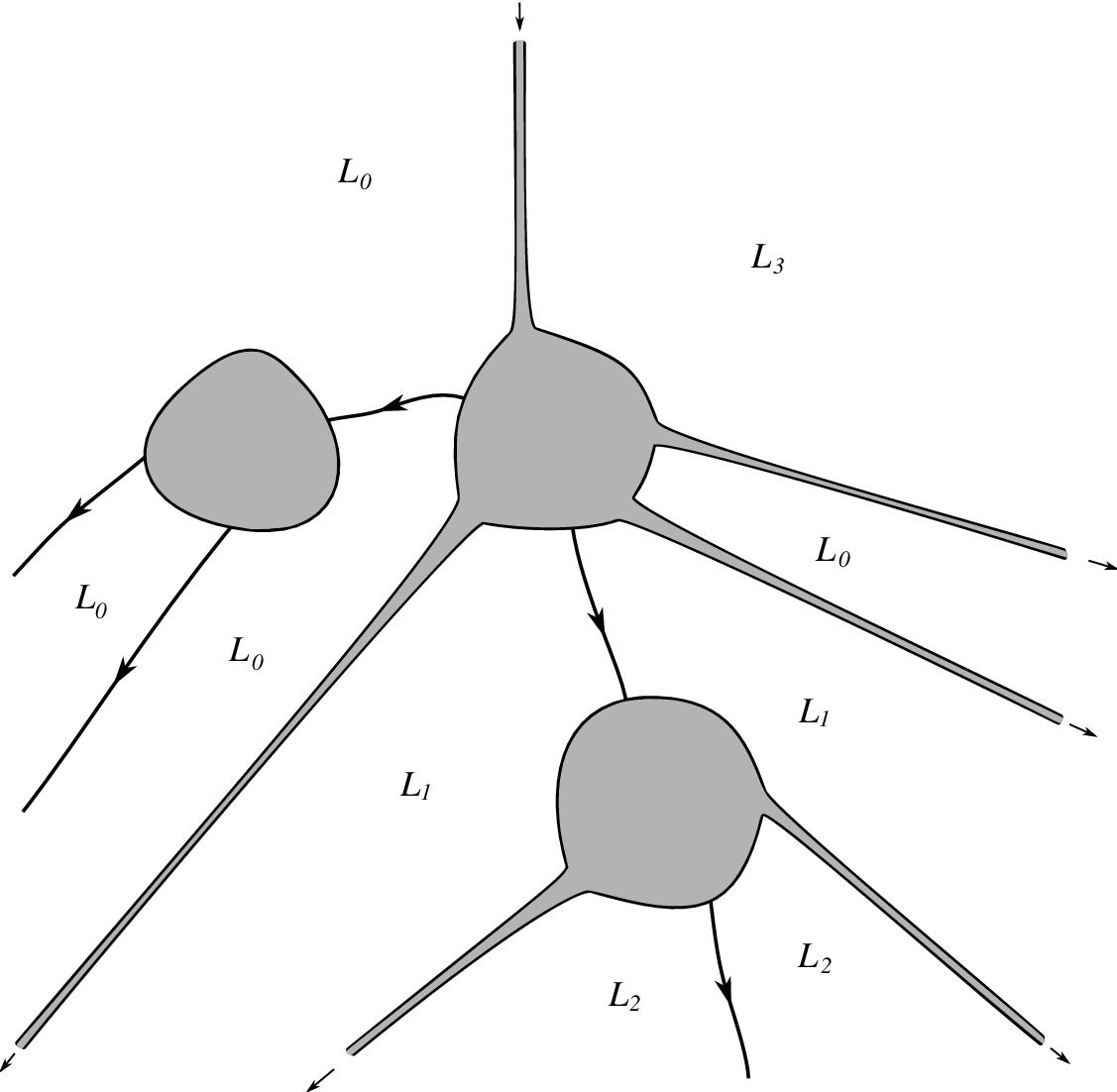}
\label{fig:pearlytree}}
\caption{Pearly trees with Lagrangian labels.
\label{fig:labels}}
\end{figure}

We define
\[ \bar{\mathcal{R}}_T(\bm{L}) := \left(\prod_{v\in V(T)} \mathcal{R}(\bm{L}_v) \right) \times (-1,0]^{E_{i}(T)}.\]
In other words, $\bar{\mathcal{R}}_T(\bm{L})$ consists of the data of the planar tree $T$, a boundary-marked disk $r_v \in \mathcal{R}(\bm{L}_v)$ for each vertex $v$, and a {\bf gluing parameter} $\rho_e \in (-1,0]$ for each internal edge $e$.

Given an internal edge $e$ of $T$ with gluing parameter $\rho_e \in (-1,0)$, we can glue the disks $r_v$ at either end of $e$ together along their strip-like ends with gluing parameter $\rho_e$ (corresponding to the `length' of the gluing region being $l_e := -\log(-\rho_e)$), to obtain an element of $\mathcal{R}_{T/e}(\bm{L})$ (where $T/e$ denotes the tree obtained from $T$ by contracting the edge $e$).
This defines a {\bf gluing map}
\[ \varphi_{T,e}: \{ r \in \bar{\mathcal{R}}_T (\bm{L}): \rho_e \in (-1,0)\} \To \bar{\mathcal{R}}_{T/e}(\bm{L}).\]

\begin{defn}
We denote by $\bar{\mathcal{R}}(\bm{L})$ the Deligne-Mumford-Stasheff compactification of $\mathcal{R}(\bm{L})$ by stable disks:
\[ \bar{\mathcal{R}}(\bm{L}) := \left(\coprod_{T}  \bar{\mathcal{R}}_T(\bm{L})\right) /\sim,\]
where
\[ r \sim \varphi_{T,e}(r) \]
whenever defined. 
Given a boundary-punctured disk $S$ with modulus $r \in \mathcal{R}(\bm{L})$, we call the union of all strip-like ends and gluing regions (under all possible gluing maps) the {\bf thin} part of $S$, and its complement the {\bf thick} part.
\end{defn}

\begin{rmk}
$\bar{\mathcal{R}}$ is the compactification of $\mathcal{R}$ by allowing the gluing parameters $\rho_e$ to take the value $0$. 
This corresponds to allowing the lengths of the gluing regions $l_e$ to be infinite.
$\bar{\mathcal{R}}(\bm{L})$ has the structure of a smooth $(k-2)$-dimensional manifold with corners (where $k := |\bm{L}| - 1$).
The codimension-$d$ boundary strata are indexed by trees $T$ with $d$ internal edges.
Namely, $T$ corresponds to the subset of $\bar{\mathcal{R}}_T$ where all $d$ gluing parameters $\rho_e$ are equal to $0$.
\end{rmk}

\begin{defn}
We denote by $\bar{\mathcal{S}}(\bm{L}) \To \bar{\mathcal{R}}(\bm{L})$ the partial compactification of the universal family $\mathcal{S}(\bm{L}) \To \mathcal{R}(\bm{L})$ of boundary-punctured disks by stable boundary-punctured disks.
\end{defn}

In \cite{seidel08}, the coefficients of the $A_{\infty}$ structure maps
\[ \mu^k: CF^*(L_{k-1},L_k) \otimes \ldots \otimes CF^*(L_0,L_1) \To CF^*(L_0,L_k)\]
are defined by counts of (appropriately perturbed) holomorphic curves $u: \mathcal{S}_r(\bm{L}) \To X$ for some $r \in \mathcal{R}(\bm{L})$.
The structure of the codimension-$1$ boundary of $\bar{\mathcal{R}}(\bm{L})$ leads to the $A_{\infty}$ associativity equations.

When no two of the Lagrangians in $\bm{L}$ coincide, we define the $A_{\infty}$ structure maps in exactly the same way.
However, when some of the Lagrangians in $\bm{L}$ coincide, we alter this definition. 

\begin{defn}
\label{defn:pearlytree}
A {\bf pearly tree} $S$ with Lagrangian labels $\bm{L}$ is specified by the following data:
\begin{itemize}
\item A stable directed $k$-leafed planar tree $T_S$ (the {\bf underlying tree} of $S$) with Lagrangian labels $\bm{L}$, such that the labels on either side of an internal edge are identical;
\item For each vertex $v$, a point $r_v \in \mathcal{R}(\bm{L}_v)$;
\item For each internal edge $e$, a length parameter $l_e \in [0,\infty)$.
\end{itemize}
We denote by $V(S)$ the set of vertices of the tree $T_S$, and by $E_L(S)$ the set of edges of $T_S$ with both sides labeled $L$ (internal or external).
For each vertex $v \in V(S)$, we define $S_v$ to be the boundary-marked disk with modulus $r_v$, with all marked points between distinct Lagrangians punctured (but all marked points between identical Lagrangians remain). 
These are the `{\bf pearls}'.
We define 
\[S^p := \coprod_{v \in V(S)} S_v.\]
For each internal edge $e$, we define $S_e := [0,l_e]$. 
For each external edge $e$ with opposite sides labeled by the same Lagrangian, we define $S_e := \R^{\pm}$, depending on the orientation of the edge.
For each Lagrangian $L \in \bm{L}$, we define
\[ S^e(L) := \coprod_{e \in E_L(S)} S_e,\]
and $S^e$ to be the disjoint union of $S^e(L)$ over all $L$.
For each $L \in \bm{L}$, we define $F_L(S)$ to be the set of flags of $T_S$ with both sides labeled by the same Lagrangian $L$.
We define $F(S)$ to be the union of all $F_L(S)$.
For each $f \in F_L(S)$, there is a corresponding marked point on a boundary component of $S^p$ with Lagrangian label $L$, which we denote by $m(f) \in S^p$. 
Also corresponding to $f$, there is a point $b(f) \in S^e(L)$, which is the boundary point of the edge corresponding to the flag $f$. 
We finally define
\[ S:= (S^p \sqcup S^e)/\sim\]
where
\[m(f) \sim b(f) \mbox{ for all $f \in F(S)$}\]
(see Figure \ref{fig:pearlytree}).
\end{defn}

We now define a topology on the moduli space of pearly trees.

Suppose we are given a stable directed $k$-leafed planar tree $T$ with Lagrangian labels $\bm{L}$.
If the labels on opposite sides of an edge are distinct, we call the edge a {\bf strip edge}, and if they are identical, we call it a {\bf Morse edge}. 
We denote by $E_{i,s}(T) \subset E(T)$ the internal strip edges, and $E_{i,M}(T) \subset E(T)$ the internal Morse edges.
We define
\[ \mathcal{R}^{pt}_T(\bm{L}) := \left(\prod_{v\in V(T)}  \mathcal{R}(\bm{L}_v)\right) \times (-1,0)^{E_{i,s}(T)} \times (-1,1)^{E_{i,M}(T)}\]
(`$pt$'  stands for `pearly tree').

As before, for any internal edge $e \in E_{i}(T)$, we have a `gluing map'
\[ \varphi_{T,e}:\{r \in \mathcal{R}^{pt}_T (\bm{L}): \rho_e \in (-1,0)\} \To  \mathcal{R}_{T/e}^{pt}(\bm{L}).\]
The only difference from the previous construction is that the gluing parameter $\rho_e$ now takes values in $(-1,1)$, rather than $(-1,0)$, for $e$ an internal Morse edge. 

\begin{defn}
\label{defn:pearltop}
We define $\mathcal{R}^{pt}(\bm{L})$, the moduli space of {\bf pearly trees} with Lagrangian labels $\bm{L}$:
\[ \mathcal{R}^{pt}(\bm{L}) := \left(\coprod_{T}  \mathcal{R}^{pt}_T(\bm{L}) \right)/\sim,\]
where
\[r \sim \varphi_{T,e}(r) \]
whenever defined.
A point $r \in \mathcal{R}^{pt}(\bm{L})$ corresponds to a pearly tree $S_r$ as follows: we glue along any edge with gluing parameter $<0$, so that we get a tree $T_S$ whose only internal edges are Morse edges with gluing parameter $\rho_e \in [0,1)$. 
We regard these as edges having length parameter
\[ l_e := -\log(1-\rho_e)\]
(see Figure \ref{fig:morseedge}). 
This defines a topology on the moduli space $\mathcal{R}^{pt}(\bm{L})$. 
Again, we define the {\bf thin} part of $S^p$ to be the union of all strip-like ends and gluing regions (including a strip neighbourhood of each boundary marked point), and the {\bf thick} part of $S^p$ to be its complement.
\end{defn}

\begin{rmk}
We could have defined $\mathcal{R}^{pt}(\bm{L})$ without any reference to strip edges at all, since we can glue along all strip edges. 
However this would not allow us to define the thick and thin regions, and we will need to consider strip edges soon anyway when we define the compactification of $\mathcal{R}^{pt}(\bm{L})$.
\end{rmk}

\begin{defn}
We denote by 
\begin{eqnarray*}
\mathcal{S}^p(\bm{L}) & \To & \mathcal{R}^{pt}(\bm{L}), \\
\mathcal{S}^e_L(\bm{L}) & \To & \mathcal{R}^{pt}(\bm{L}) \mbox{ (for $L \in \bm{L}$), and} \\
\mathcal{S}^{pt}(\bm{L}) & \To & \mathcal{R}^{pt}(\bm{L}) 
\end{eqnarray*} 
the universal families with fibre $S^p_r, S^e_r(L)$ and $S_r$ respectively, over a point $r \in \mathcal{R}^{pt}(\bm{L})$.
\end{defn}

\begin{defn}
We define a universal choice of strip-like ends for the family $\mathcal{S}^{pt}(\bm{L}) \To \mathcal{R}^{pt}(\bm{L})$ to consist of the embeddings
\[ \epsilon_j: \mathcal{R}^{pt}(\bm{L}) \times Z^{\pm} \To \mathcal{S}^{pt}(\bm{L})\]
for each external strip edge, coming from our universal choice of strip-like ends for families of boundary-punctured disks, and
\[ \epsilon_j : \mathcal{R}^{pt}(\bm{L}) \times \R^{\pm} \To \mathcal{S}^{pt}(\bm{L})\]
which are parametrisations of the corresponding external Morse edges (where the sign $\pm$ is determined by the orientation of the edge).
\end{defn}

\begin{figure}
\centering
\includegraphics[width=0.9\textwidth]{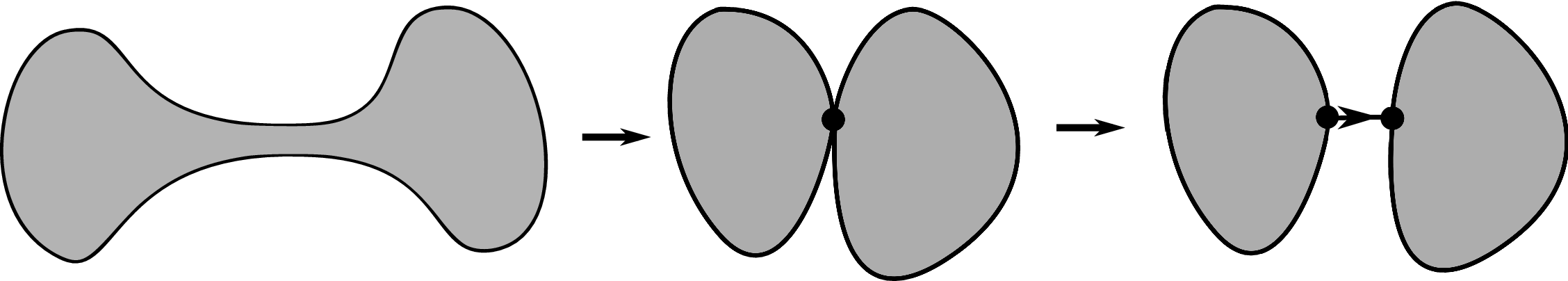}
\caption{In this figure, we show what happens as the gluing parameter $\rho_e$ for a Morse edge in a pearly tree passes from negative to positive.
On the left, $\rho_e<0$, and we have a `thin' region in our disk, corresponding to the edge $e$.
As $\rho_e \To 0^{-}$, the thin region's length becomes infinite, until at $\rho_e = 0$ we have a stable disk (middle picture).
On the right, $\rho_e > 0$, and we have two distinct disks connected by an edge of length $l_e = -\log(1-\rho_e)$.
As $\rho_e \To 0^{+}$, the edge's length goes to $0$, until at $\rho_e = 0$ we have the same stable disk. 
\label{fig:morseedge}}
\end{figure}

\begin{defn}
\label{defn:medge}
Given a tree $T_S$ as above, and a subset $B \subset E(T_S)$, we define $\mathcal{R}^{pt}(T_S,B) \subset \mathcal{R}^{pt}(\bm{L})$ to be the images of pearly trees $S$ with underlying tree $T_S$, with gluing parameter $\rho_e = 0$ for $e \in B$ and $\rho_e >0$ for $e \notin B$ (of course this depends on the Lagrangian labels, but we omit $\bm{L}$ from the notation for readability). 
Each pearly tree $r \in \mathcal{R}^{pt}(\bm{L})$ lies in a unique subset $\mathcal{R}^{pt}(T_S,B)$.
\end{defn}

\begin{defn}
Given $(T_S,B)$ as in Definition \ref{defn:medge}, we define the universal family
\[ \mathcal{S}^{pt}(T_S,B) \To \mathcal{R}^{pt}(T_S,B).\]
\end{defn}

We now define the compactification of $\mathcal{R}^{pt}(\bm{L})$.
Let
\[ \bar{\mathcal{R}}^{pt}_T(\bm{L}) := \left(\prod_{v\in V(T)}  \mathcal{R}(\bm{L}_v)\right) \times (-1,0]^{E_{i,s}(T)} \times (-1,1]^{E_{i,M}(T)}.\]
Note that $ \bar{\mathcal{R}}^{pt}_T(\bm{L})$ contains $ \mathcal{R}^{pt}_T(\bm{L})$ as a dense open subset.

\begin{defn}
\label{defn:stabpearlytree}
We define the the compactification of $\mathcal{R}^{pt}(\bm{L})$, the moduli space of {\bf stable pearly trees},
\[ \bar{\mathcal{R}}^{pt}(\bm{L}) := \left(\coprod_{T}  \bar{\mathcal{R}}^{pt}_T (\bm{L})\right)/\sim,\]
where
\[r \sim \varphi_{T,e}(r) \]
whenever defined.
We also define the universal family $\bar{\mathcal{S}}^{pt}(\bm{L}) \To \bar{\mathcal{R}}^{pt}(\bm{L})$ of stable pearly trees.
\end{defn}

\begin{rmk}
In the spaces $\bar{\mathcal{R}}^{pt}_T$, the gluing parameters of strip (respectively Morse) edges can take the value $0$ (respectively $1$). 
This corresponds to the length of the gluing region $l_e$ becoming infinite (respectively, the length of the edge $l_e$ becoming infinite).
Thus, we are essentially compactifying by allowing the pearls to be stable disks, and the Morse edges to have infinite length.
$\bar{\mathcal{R}}^{pt}(\bm{L})$ has the structure of a smooth $(k-2)$-manifold with corners. 
The codimension-$d$ boundary strata are indexed by trees $T$ with Lagrangian labels $\bm{L}$ and $d$ internal edges. 
Namely, the boundary stratum corresponding to $T$ is the image of the subset of $\bar{\mathcal{R}}^{pt}_T(\bm{L})$ where all gluing parameters $\rho_e$ are $0$ for strip edges and $1$ for Morse edges. 
\end{rmk}

\begin{rmk}
$\bar{\mathcal{R}}^{pt}(\bm{L})$ is obtained from the usual Deligne-Mumford-Stasheff compactification $\bar{\mathcal{R}}(\bm{L})$ by adding a `collar' along each boundary stratum corresponding to a tree with a Morse edge in it.
\end{rmk}

Na\"{i}vely, the structure coefficients of the usual Fukaya category count rigid holomorphic disks $u: \mathcal{S}_r \To X$ for some $r \in \mathcal{R}(\bm{L})$. 
In reality, we must perturb the $J$-holomorphic curve equation to achieve transversality, in particular when two of the Lagrangian boundary conditions coincide.
In \cite{seidel08}, the equation is perturbed by allowing modulus- and domain-dependent almost-complex structures and Hamiltonian perturbations. 

We would like to alter the definition of the Fukaya category so that the structure coefficients are counts of rigid `holomorphic pearly trees' $u: S_r \To X$ for some $r \in \mathcal{R}^{pt}(\bm{L})$.
Na\"{i}vely, a holomorphic pearly tree is a map which is holomorphic on the pearls and given by the Morse flow of some Morse function on the corresponding Lagrangian on each edge. 
Again, in reality, we have to perturb the holomorphic curve and Morse flow equations by modulus- and domain-dependent perturbations in order to achieve transversality. 
We describe how to do this in Sections \ref{subsec:floerdata}-\ref{subsec:pertdatafam}.

\subsection{Floer data and morphism spaces}
\label{subsec:floerdata}

Recall, from Section \ref{subsec:afuk}, that we define the Fukaya category of a symplectic manifold $(X,\omega)$ with the following properties and structures:

\begin{itemize}
\item $\omega = d \theta$ is exact;
\item $X$ is equipped with an almost-complex structure $J_0$, compatible with $\omega$;
\item $X$ is convex at infinity, in the sense that there is a bounded below, proper function $h: X \To \R$ such that
\[ \theta = - dh \circ J_0;\]
\item $X$ is equipped with a complex volume form $\eta$ (note: we will not take a quadratic complex volume form as in \cite{seidel08}, because we will assume our Lagrangians to be oriented).
\end{itemize}

An object of the Fukaya category of $X$ is a compact, exact, embedded Lagrangian brane $L^{\#}$ (we will neglect the superscript $\#$, denoting the brane structure, for notational convenience).

\begin{defn}
We define 
\[\mathcal{H} := C^{\infty}_c(X,\R),\]
the space of smooth, compactly supported functions on $X$ (think of this as the space of Hamiltonians), and $\mathcal{J}$, the space of smooth almost-complex structures on $X$ compatible with $\omega$, and equal to the standard complex structure $J_0$ outside of some compact set.
For future use, for each Lagrangian $L$, we define
\[ \mathcal{V}_L := C^{\infty}(L,TL),\]
the space of smooth vector fields on $L$.
\end{defn}

\begin{defn}
\label{defn:floerdata}
For each {\bf distinct} pair of objects $(L_0,L_1)$, we choose a {\bf Floer datum} $(H_{01},J_{01})$ consisting of 
\[H_{01} \in C^{\infty}([0,1],\mathcal{H}) \mbox{ and }J_{01} \in C^{\infty}([0,1],\mathcal{J})\] 
satisfying the following property: if $\phi^t$ denotes the flow of the Hamiltonian vector field of the (time-dependent) Hamiltonian $H_{01}$, then the time-$1$ flow $\phi^1(L_0)$ is transverse to $L_1$. 
One then defines a {\bf generator} of $CF^*(L_0,L_1)$ to be a path $y:[0,1] \To X$ which is a flowline of the Hamiltonian vector field of $H_{01}$, such that $y(0) \in L_0$ and $y(1) \in L_1$ (these correspond to the transverse intersections of $\phi^1(L_0)$ with $L_1$). 
One defines $CF^*(L_0,L_1)$ to be the $\C$-vector space generated by its generators.
It is $\Z$-graded, as explained in \cite[Chapter 11, 12]{seidel08}. 
\end{defn}

In \cite{seidel08}, the case $L_0 = L_1$ is treated identically, but we will do something different.

\begin{defn}
A {\bf Floer datum} for a pair of identical Lagrangians $(L,L)$ is a Morse-Smale pair $(h_L,g_L)$ consisting of a Morse function $h_L:L \To \R$ and a Riemannian metric $g_L$ on $L$.
One then defines $CF^*(L,L) := C_M^*(L)$, the $\C$-vector space generated by critical points of $h_L$. 
It is $\Z$-graded by the Morse index.
\end{defn}

\begin{rmk}
\label{rmk:limit}
Intuitively, one should think of this as a limiting case of Definition \ref{defn:floerdata}.
Namely, we could choose the almost-complex structure part of the perturbation datum to be a time-independent $J \in \mathcal{J}$ which, when combined with $\omega$, induces a Riemannian metric whose restriction to $L$ is $g_L$.
We could then choose the Hamiltonian part of the perturbation datum to be a time-independent function $\epsilon H$, where $H|_{L} = h_L$, and consider the limit $\epsilon \To 0$. 
\end{rmk}

\begin{defn}
Given a set of Lagrangian labels $\bm{L} = (L_0,\ldots,L_k)$, an {\bf associated set of generators} is a tuple
\[ \bm{y} = (y_0,\ldots,y_k),\]
where $y_j$ is a generator of $CF^*(L_{j-1},L_j)$ for each $1 \le j \le k$, and $y_0$ is a generator of $CF^*(L_0,L_k)$.
We denote the grading of a generator $y$ by $i(y)$, and define
\[ i(\bm{y}) := i(y_0) - \sum_{j=1}^k i(y_j).\]
\end{defn}

\subsection{Perturbation data for fixed moduli}
\label{subsec:pertdatafix}

For the purposes of this section, let $S$ be a pearly tree with Lagrangian labels $\bm{L}$ and fixed modulus $r \in \mathcal{R}^{pt}(\bm{L})$. 

\begin{defn}
A {\bf perturbation datum} for $S$ consists of the data $(K,J,V)$, where:
\begin{itemize}
\item $K \in \Omega^1(S^p,\mathcal{H})$;
\item $J \in C^{\infty}(S^p,\mathcal{J})$;
\item $V$ is a tuple of maps $V_L \in C^{\infty}(S^e(L),\mathcal{V}_L)$ for each $L \in \bm{L}$,
\end{itemize}
such that
\[ K(\xi)|_{L_C} = 0 \mbox{ for all $\xi \in TC \subset T(\partial S^p)$}\]
for each boundary component $C$ of a pearl in $S$ with Lagrangian label $L_C$.

We also impose a requirement that the perturbation datum be compatible with the Floer data on the strip-like ends, in the following senses:
\[\epsilon_j^*K = H_{j-1,j}(t)dt, \,\,\, J(\epsilon_j(s,t)) = J_{j-1,j}(t)\]
on each external strip edge;
\[ V_{L_j}(\epsilon_j(s)) = \nabla h_{L_j}\]
on each external Morse edge.
\end{defn}

\begin{defn}
\label{defn:perthol}
Given a pearly tree $S$ with Lagrangian labels $\bm{L}$ and a perturbation datum $(K,J,V)$, a {\bf holomorphic pearly tree} (or more properly, an inhomogeneous pseudo-holomorphic pearly tree) in $X$ with domain $S$ is a collection $\bm{u}$ of smooth maps 
\begin{eqnarray*}
u_p: S^p & \To&  X \mbox{ and }\\
u_L: S^e(L) & \To & L \mbox{ for all $L$ in $\bm{L}$,}
\end{eqnarray*}
satisfying
\begin{eqnarray*}
u_p(C) &\in& L_C \mbox{ for each boundary component $C$ of $S^p$ with label $L_C$;}\\
u_p(m(f)) &=& u_L(b(f)) \mbox{ for all $f \in F_L(S)$, for all $L$;} \\
(Du_p - Y)^{0,1} &=& 0 \mbox{ on $S^p$;}\\
Du_L-V &=& 0 \mbox{ on $S^e(L)$, for all $L$,}
\end{eqnarray*}
where, for $\xi \in TS$, $Y(\xi)$ is the Hamiltonian vector field of the function $K(\xi)$. 
Note that the second condition says exactly that $\bm{u}$ defines a continuous map $S \To X$. 
\end{defn}

\begin{defn}
Given $\bm{y} = (y_-,y_+)$, where $y_{\pm}$ are generators of $CF^*(L_0,L_1)$, we define the moduli space $\mathcal{M}_Z(\bm{y})$ of solutions of the holomorphic pearly tree equation with domain $Z = \R \times [0,1]$ (if $L_0 \neq L_1$) or $\R$ (if $L_0 = L_1$), translation-invariant perturbation datum given by the corresponding Floer datum, and asymptotic conditions
\[ \lim_{s \To \pm \infty} u(s,t) = y_{\pm}(t)\]
if $L_0 \neq L_1$, and the same without the $t$ variable if $L_0 = L_1$.
We define $\mathcal{M}^*_Z(\bm{y}) := \mathcal{M}_Z(\bm{y})/\R$, where $\R$ acts by translation in the $s$ variable.
\end{defn}

It is standard (see \cite{fhs95,oh97}) that the moduli spaces $\mathcal{M}^*_Z(\bm{y})$ are smooth manifolds for generic choice of Floer data, and their dimension is $i(\bm{y})-1$.

\begin{defn}
\label{defn:asympcond}
Suppose that $k \ge 2$. 
Given a pearly tree $S$ with Lagrangian labels $\bm{L} = (L_0,\ldots,L_k)$, associated generators $\bm{y} = (y_0,\ldots,y_k)$, and a perturbation datum, we consider the moduli space $\mathcal{M}_{S}(\bm{y})$ of holomorphic pearly trees with domain $S$, such that
\[ \lim_{s \To + \infty} u(\epsilon_j(s,t)) = y_j(t)\]
and
\[ \lim_{s \To - \infty} u(\epsilon_0(s,t)) = y_0(t)\]
on external strip edges, and the same (without the $t$ variable) on external Morse edges.
\end{defn}

We wish to show that the moduli spaces $\mathcal{M}_S(\bm{y})$ form smooth, finite-dimensional manifolds for a generic choice of perturbation datum. 

\begin{defn}
\label{defn:banman}
Fix $2 < p < \infty$ and define the Banach manifold $\mathcal{B}_S(\bm{y})$ to consist of collections of maps
\[ \bm{u} = (u_p, \bm{u}_{\bm{L}}) \in W^{1,p}_{loc}(S^p,X) \times \prod_{L \in \bm{L}} W^{1,p}_{loc}(S^e(L),L)\]
such that 
\[ u_p(C) \in L_C\]
for each boundary component $C$ of $S^p$ with label $L_C$, and $\bm{u}$ converges in $W^{1,p}$-sense to $y_j$ on the $j$th strip-like end. 
These boundary and asymptotic conditions make sense because $W^{1,p}$ injects into the space of continuous functions.
Henceforth we omit the $\bm{y}$ from the notation for readability.
Note that the tangent space to $\mathcal{B}_S$ is
\[ T _{\bm{u}} \mathcal{B}_S = W^{1,p}(S^p, u_p^* TX, u_p^*TL_C) \oplus \bigoplus_{L \in \bm{L}} W^{1,p}(S^e(L),u_L^*TL),\]
where for the first component we have used the notation $W^{1,p}(S^p, E, F)$ for the space of $W^{1,p}$ sections of a vector bundle $E$ over $S$, whose restriction to the boundary lies in the distribution $F \subset E|_{\partial S^p}$.
\end{defn}

\begin{defn}
The maps $\bm{u} \in \mathcal{B}_S$ are not necessarily continuous at the points where edges join onto pearls.
We define
\[ \bm{L}^{F(S)}:= \prod_{L \in \bm{L}} L^{F_L(S)}.\]
Then there are evaluation maps
\begin{eqnarray*}
\bm{ev}_m: \mathcal{B}_S&\To& \bm{L}^{F(S)} \\
\bm{ev}_m(\bm{u}) & := & (u_p(m(f)))_{f \in F(S)}
\end{eqnarray*}
and
\begin{eqnarray*}
\bm{ev}_b: \mathcal{B}_S &\To& \bm{L}^{F(S)} \\
\bm{ev}_b(\bm{u}) & := & (u_L(b(f)))_{ f \in F_L(S)}.
\end{eqnarray*}
We define
\begin{eqnarray*}
\bm{ev} : \mathcal{B}_S & \To & \bm{L}^{F(S)} \times \bm{L}^{F(S)} \\
\bm{ev} &:= &(\bm{ev}_m, \bm{ev}_b).
\end{eqnarray*}
We also define
\[\Delta^{S} \subset \bm{L}^{F(S)} \times \bm{L}^{F(S)}\]
to be the diagonal.
An element $\bm{u} \in \mathcal{B}_S$ is continuous at the points where edges join onto pearls if and only if $\bm{u} \in \bm{ev}^{-1}(\Delta^S)$.
We define the linearization of $\bm{ev}$,
\[ D(\bm{ev}): T_{\bm{u}} \mathcal{B}_S \To T_{\bm{ev}(\bm{u})} \left(\bm{L}^{F(S)} \times \bm{L}^{F(S)} \right).\]
Given a point $\bm{u} \in \bm{ev}^{-1}(\Delta^S)$, we define the projection of the linearization to the normal bundle of the diagonal,
\begin{eqnarray*}
D^{ev}_{S,\bm{u}}: T_{\bm{u}} \mathcal{B}_S & \To & T_{\bm{ev}_m(\bm{u})} \bm{L}^{F(S)},\\
D^{ev}_{S,\bm{u}} &:=& D(\bm{ev}_m) - D(\bm{ev}_b).
\end{eqnarray*}
\end{defn}

\begin{defn}
\label{defn:lin}
Define the Banach vector bundle $\mathcal{E}_S(\bm{y}) \To \mathcal{B}_S(\bm{y})$ whose fibre over $\bm{u}$ (again omitting the $\bm{y}$ from the notation) is the space
\[ (\mathcal{E}_S)_{\bm{u}} := L^p(S^p,\Omega^{0,1}_S \otimes u_p^* TX) \oplus \bigoplus_{L \in \bm{L}} L^p(S^e(L),u_L^*TL). \]
There is a smooth section
\begin{eqnarray*}
d_S: \mathcal{B}_S & \To &  \mathcal{E}_S \\
d_S(\bm{u}) & = &((Du_p-Y)^{0,1}, (D\bm{u}_{\bm{L}} - V)).
\end{eqnarray*}
We denote the linearization of $d_S$ at $\bm{u}$ by
\[ D^h_{S,\bm{u}} : T _{\bm{u}}\mathcal{B}_S \To  (\mathcal{E}_S)_{\bm{u}}\]
(the `$h$' stands for `holomorphic').
\end{defn}

Note that $\mathcal{M}_S(\bm{y}) = (\bm{ev},d_S)^{-1}(\Delta^S, \bm{0})$ (where $\bm{0}$ denotes the zero section of the Banach vector bundle $\mathcal{E}_S(\bm{y})$).

\begin{defn}
Given $\bm{u} \in \mathcal{M}_S(\bm{y})$, we denote by
\begin{eqnarray*}
D_{S,\bm{u}}: T _{\bm{u}}\mathcal{B}_S &\To& T_{\bm{ev}_m(\bm{u})} \bm{L}^{F(S)} \oplus (\mathcal{E}_S)_{\bm{u}} 
\end{eqnarray*}
the projection of the linearization 
\[ D_{\bm{u}}(\bm{ev},d_S)\]
to the normal bundle of $(\Delta^S,\bm{0})$.
It is given by
\[D_{S,\bm{u}} = D^{ev}_{S,\bm{u}} \oplus D^h_{S,\bm{u}}.\]
We say that $\bm{u} \in \mathcal{M}_S(\bm{y})$ is {\bf regular} if $D_{S,\bm{u}}$ is surjective, and that $\mathcal{M}_S(\bm{y})$ is regular if every $\bm{u} \in \mathcal{M}_S(\bm{y})$ is regular. 
\end{defn}

It is standard that the operator $D^h_{S,\bm{u}}$ is Fredholm (compare \cite[Section 8i]{seidel08} for the pearls, and \cite[Section 2.2]{schwarz} for the edges).
Therefore, $D_{\bm{u}}(\bm{ev},d_S)$ is Fredholm also, because the codomain of $\bm{ev}$ is finite-dimensional. 
So the map $(\bm{ev},d_S)$ is Fredholm.
Thus, if $\mathcal{M}_S(\bm{y})$ is regular, then it is a smooth manifold with dimension given by the Fredholm index of $D_{S,\bm{u}}$ at each point.

It will follow from our arguments in Section \ref{subsec:transv} that, for a generic choice of perturbation datum, $\mathcal{M}_S(\bm{y})$ is regular.

\subsection{Perturbation data for families}
\label{subsec:pertdatafam}

To define the Fukaya category, we must count moduli spaces of holomorphic pearly trees with varying domain, rather than a fixed domain as in Section \ref{subsec:pertdatafix}.
The first step is to define perturbation data for the whole family $\mathcal{S}^{pt}(\bm{L}) \To \mathcal{R}^{pt}(\bm{L})$.
The following definition is the appropriate notion of a smoothly varying family of perturbation data for each fibre $S_r$.

\begin{defn}
\label{defn:pertfam}
A {\bf perturbation datum} for the family $\mathcal{S}^{pt}(\bm{L}) \To \mathcal{R}^{pt}(\bm{L})$ consists of the data $(K,J,V)$, where:
\begin{itemize}
\item $K \in \Omega^1_{\mathcal{S}^p/\mathcal{R}^{pt}}(\mathcal{S}^p,\mathcal{H})$;
\item $J \in C^{\infty}(\mathcal{S}^p,\mathcal{J})$;
\item $V$ is a tuple of maps $V_L \in C^{\infty}(\mathcal{S}^e_L, \mathcal{V}_L)$ for each $L \in \bm{L}$,
\end{itemize}
such that the restriction of $(K,J,V)$ to each fibre $S_r$ is a perturbation datum. 
We furthermore require some additional, somewhat artificial, conditions to deal with the structure of the moduli space near a point with an edge of length $0$ (the situation illustrated in Figure \ref{fig:morseedge}). 
Namely, for any edge $e$, we require:
\begin{itemize}
\item $V|_{S_e} = 0$ whenever $l_e \in [0,1]$;
\item the perturbation data do not change as $l_e$ varies between $0$ and $1$ (keeping all other parameters fixed);
\item $V|_{S_e} = \nabla h_L$ whenever $l_e \ge 2$;
\item $K \equiv 0$, and $J$ is constant, on a neighbourhood of each Morse edge of length $0$. 
To see what this means, look at Figure \ref{fig:morseedge}: we require that $K \equiv 0$ and $J$ has one fixed value on the long strip on the left, and in a neighbourhood of the boundary marked points at opposite ends of the edge on the right.
\end{itemize}
\end{defn}

We impose the condition $V|_{S_e} = 0$ on edges of length $l_e \le 1$ because it makes the following Lemma true (a similar trick is used in \cite{abouzaidplumb}):

\begin{lem}
\label{lem:constedge}
Suppose that we have chosen a perturbation datum in accordance with Definition \ref{defn:pertfam}, and that $S = S_r$ is a pearly tree with an edge $e$ of length $l_e < 1$. 
Let $S' = \mathcal{S}^{pt}_{r'}$ denote the pearly tree that is identical to $S$, except we shrink the edge $e$ to have length $l_e = 0$. 
Then there is a canonical isomorphism
\[ \mathcal{M}_S(\bm{y}) \equiv \mathcal{M}_{S'}(\bm{y})\]
(where both are defined using the restriction of the perturbation datum on $\mathcal{S}^{pt}$ to the fibres $S, S'$).
\end{lem}
\begin{proof}
The result is clear from the holomorphic pearly tree equation (see Definition \ref{defn:perthol}): because $V|_{S_e} = 0$ for $l_e \in [0,1]$, the corresponding map $\bm{u}|_{S_e}: [0,l_e] \To L$ is necessarily constant. 
Thus the part of the holomorphic pearly tree equation on the edge $e$ reduces to a point constraint, regardless of $l_e$.
Because the perturbation datum does not change as we vary $l_e \in [0,1]$, the equation on the rest of $S$ does not change, so $\mathcal{M}_S(\bm{y})$ and $\mathcal{M}_{S'}(\bm{y})$ can be canonically identified.
\end{proof}

\begin{defn}
\label{defn:modvar}
Given a set of Lagrangian labels $\bm{L} = (L_0,\ldots,L_k)$, associated generators $\bm{y}$,
and a perturbation datum, we consider the moduli space 
\[ \mathcal{M}_{\mathcal{S}^{pt}}(\bm{y}) := \{ (r,\bm{u}): r \in \mathcal{R}^{pt}(\bm{L}) \mbox{ and } \bm{u} \in \mathcal{M}_{S_r}(\bm{y})\}.\]
\end{defn}

We now aim to show that $\mathcal{M}_{\mathcal{S}^{pt}}(\bm{y})$ is a manifold (whether it is possible to construct a smooth manifold structure is unclear, but this is irrelevant for the purposes of defining the Fukaya category). 
The complicated part of this is to understand what happens in a neighbourhood of the Morse edges of zero length, because the nature of the domain changes at those points.
We start by explaining what happens away from the Morse edges of zero length (i.e., when the modulus $r \in \mathcal{R}^{pt}(T_S,B)$ where $B = \phi$).

\begin{defn}
\label{defn:extlin}
Let $U \subset \mathcal{R}^{pt}(\bm{L})$ be a small connected open subset which makes the strip-like ends constant and avoids a neighbourhood of the pearly trees with some Morse edge of length $0$.
We define the trivial Banach fibre bundle $\mathcal{B}_{\mathcal{S}^{pt}|_U} (\bm{y}) \To U$ whose fibre over $r \in U$ is the Banach manifold $\mathcal{B}_{S_r}(\bm{y})$ defined in Definition \ref{defn:banman}. 
There is a Banach vector bundle $\mathcal{E}_{\mathcal{S}^{pt}|_U}(\bm{y}) \To \mathcal{B}_{\mathcal{S}^{pt}|_U}(\bm{y})$ whose restriction (omitting the $\bm{y}$ from the notation) to $\mathcal{B}_{S_r}$ is the Banach vector bundle $\mathcal{E}_{S_r}$ defined in Definition \ref{defn:lin}.
It has a smooth section $d_{\mathcal{S}^{pt}|_U}$ given, over $\mathcal{B}_{S_r}$, by the section $d_{S_r}$ of Definition \ref{defn:lin}.
We have
\[ \mathcal{M}_{\mathcal{S}^{pt}|_U} (\bm{y})= (\bm{ev}|_U,d_{\mathcal{S}^{pt}|_U})^{-1}(\Delta^S, \bm{0})\]
(note that the codomain of $\bm{ev}$ depends on the underlying tree $T_S$ of $S_r$; our requirement that $U$ be connected and avoid Morse edges of length $0$ ensures that $T_S$ is constant on $U$).
Given $(r,\bm{u}) \in \mathcal{M}_{\mathcal{S}^{pt}}(\bm{y})$ with $r \in U$, we denote the linearization of $d_{\mathcal{S}^{pt}|_U}$ at $(r,\bm{u})$ by 
\[D^h_{\mathcal{S}^{pt}|_U,r,\bm{u}}: T_{(r,\bm{u})} \left(\mathcal{B}_{\mathcal{S}^{pt}|_U} \right) \To (\mathcal{E}_{S_r})_{\bm{u}},\]
where we note that
\[T_{(r,\bm{u})} \left(\mathcal{B}_{\mathcal{S}^{pt}|_U}\right) = T_r \mathcal{R}^{pt} \oplus T_{\bm{u}} \mathcal{B}_{S_r}.\]
\end{defn}

\begin{rmk}
The component
\[ T_{\bm{u}} \mathcal{B}_{S_r} \To (\mathcal{E}_{S_r})_{\bm{u}}\]
is just the linearized operator $D^h_{S_r,{\bm{u}}}$ from Definition \ref{defn:lin}. 
The component
\[T_r \mathcal{R}^{pt} \To (\mathcal{E}_{S_r})_{\bm{u}}\]
corresponds to derivatives of the holomorphic curve equation (Definition \ref{defn:perthol}) with respect to changes of the modulus $r$.
\end{rmk}

\begin{defn}
We denote by
\[ D_{\mathcal{S}^{pt}|_U, r,\bm{u}}: T_{(r,\bm{u})} \left(\mathcal{B}_{\mathcal{S}^{pt}|_U}\right) \To 
T_{\bm{ev}_m(\bm{u})} \bm{L}^{F(S)} \oplus (\mathcal{E}_{S_r})_{\bm{u}} \]
the projection of the linearization 
\[D_{r,\bm{u}}(\bm{ev}|_U,d_{\mathcal{S}^{pt}}|_U)\]
to the normal bundle of $(\Delta^S,\bm{0})$. 
It is given by
\[ D_{\mathcal{S}^{pt}|_U, r,\bm{u}}= D^{ev}_{S_r,{\bm{u}}} \oplus D^h_{\mathcal{S}^{pt}|_U,r,{\bm{u}}}.\]
If $S_r$ has no edges of length $0$, we say that $(r,{\bm{u}})$ is a {\bf regular} point of $\mathcal{M}_{\mathcal{S}^{pt}}(\bm{y})$ if $D_{\mathcal{S}^{pt}|_U,r,\bm{u}}$ is surjective (for some open neighbourhood $U$ of $r$ as above). 
We say that the moduli space $\mathcal{M}_{\mathcal{S}^{pt}|_U}(\bm{y})$ is regular if every ${\bm{u}} \in \mathcal{M}_{\mathcal{S}^{pt}|_U}(\bm{y})$ is regular.
\end{defn}

\begin{prop}
\label{prop:fredind}
The operator $D_{\mathcal{S}^{pt}|_U, r,\bm{u}}$ is Fredholm of index
\[ \mathrm{ind}(D_{\mathcal{S}^{pt}|_U,r,\bm{u}}) = k-2 + i(\bm{y})\]
when $U$ avoids a neighbourhood of all pearly trees with edges of length $0$.
\end{prop}
\begin{proof}
See \cite[Section 12d]{seidel08} for the pearl component -- the inclusion of the Morse flowlines is a trivial addition.
\end{proof}

It follows that, if $\mathcal{M}_{\mathcal{S}^{pt}|_U}(\bm{y})$ is regular, then it is a smooth manifold with dimension equal to the Fredholm index of $D_{\mathcal{S}^{pt}|_U}$ given above. 
The transition maps between the spaces $\mathcal{B}_{\mathcal{S}^{pt}|_U}$ are not necessarily smooth, so in general it is not possible to define a Banach manifold `$\mathcal{B}_{\mathcal{S}^{pt}|_U}$' over an arbitrarily large open set $U$ avoiding a neighbourhood of the Morse edges of length $0$. 
However, elliptic regularity ensures that the transition maps between spaces $\mathcal{M}_{\mathcal{S}^{pt}|_U}(\bm{y})$ are smooth in the regular case, hence they can be patched together to obtain a smooth manifold $\mathcal{M}_{\mathcal{S}^{pt}|_U}(\bm{y})$ over an arbitrarily large open set $U$ avoiding a neighbourhood of the Morse edges of length $0$ (compare \cite[Remark 9.4]{seidel08}).

Now we must deal with the Morse edges of length $0$, i.e., the case that the modulus $r \in \mathcal{R}^{pt}(T_S,B)$, where $B \neq \phi$ (in the notation of Definition \ref{defn:medge}).

\begin{defn}
\label{defn:tsb}
We define the moduli space
\[ \mathcal{M}_{\mathcal{S}^{pt}(T_S,B)}(\bm{y}) := \{(r,{\bm{u}}) \in \mathcal{M}_{\mathcal{S}^{pt}}(\bm{y}): r \in \mathcal{R}^{pt}(T_S,B)\}.\]
\end{defn}

In order to construct a manifold structure on the moduli space $\mathcal{M}_{\mathcal{S}^{pt}}(\bm{y})$, we are going to arrange that all of the moduli spaces $\mathcal{M}_{\mathcal{S}^{pt}(T_S,B)}(\bm{y})$ are regular, then use them to construct charts for the manifold structure on $\mathcal{M}_{\mathcal{S}^{pt}}(\bm{y})$.

\begin{defn}
Let $U \subset \mathcal{R}^{pt}(T_S,B)$ be a small connected open subset which makes the strip-like ends constant and avoids a neighbourhood of the pearly trees with some Morse edge not in $B$ having length $0$. 
We define $\mathcal{B}_{\mathcal{S}^{pt}(T_S,B)|_U}, \mathcal{E}_{\mathcal{S}^{pt}(T_S,B)|_U}, d_{\mathcal{S}^{pt}(T_S,B)|_U}$ by restricting $\mathcal{B}_{\mathcal{S}^{pt}|_U}, \mathcal{E}_{\mathcal{S}^{pt}|_U}, d_{\mathcal{S}^{pt}|_U}$to $\mathcal{R}^{pt}(T_S,B)$.
We have
\[ \mathcal{M}_{\mathcal{S}^{pt}(T_S,B)|_U}(\bm{y}) = (\bm{ev}|_U, d_{\mathcal{S}^{pt}(T_S,B)|_U})^{-1}(\Delta^S,\bm{0}).\]
The projection of the linearization 
\[ D_{r,\bm{u}}(\bm{ev}|_U, d_{\mathcal{S}^{pt}(T_S,B)|_U})\]
to the normal bundle of $(\Delta^S,\bm{0})$ is the restriction of $D_{\mathcal{S}^{pt}|_U,r,\bm{u}}$ to the codimension-$|B|$ subspace
\[T_r \mathcal{R}^{pt} (T_S,B) \oplus T_{\bm{u}} \mathcal{B}_{S_r} \subset T_r \mathcal{R}^{pt} \oplus T_{\bm{u}} \mathcal{B}_{S_r}.\]
We denote it by $D_{\mathcal{S}^{pt}(T_S,B)|_U,r,\bm{u}}$.
By Proposition \ref{prop:fredind}, it is Fredholm of index
\[ \mathrm{ind}(D_{\mathcal{S}^{pt}(T_S,B)|_U,r,\bm{u}}) = k-2+i(\bm{y}) - |B|.\]
\end{defn}

\begin{defn}
\label{defn:tsbreg}
We say that $(r,{\bm{u}})$ is a regular point of $\mathcal{M}_{\mathcal{S}^{pt}}(\bm{y})$ if $r \in \mathcal{R}^{pt}(T_S,B)$ and the operator $D_{\mathcal{S}^{pt}(T_S,B)|_U,r,\bm{u}} $ is surjective (for some open neighbourhood $U \subset \mathcal{R}^{pt}(T_S,B)$ of $r$ as above). 
We say that the moduli space $\mathcal{M}_{\mathcal{S}^{pt}}(\bm{y})$ is regular if every $(r,{\bm{u}}) \in \mathcal{M}_{\mathcal{S}^{pt}}(\bm{y})$ is regular.
\end{defn}

It follows that, if $\mathcal{M}_{\mathcal{S}^{pt}}(\bm{y})$ is regular, then each moduli space $\mathcal{M}_{\mathcal{S}^{pt}(T_S,B)|_U}$ is a smooth manifold with dimension equal to the Fredholm index of $D_{\mathcal{S}^{pt}(T_S,B)}$ given above.

Assuming regularity, we now construct charts for a manifold structure on $\mathcal{M}_{\mathcal{S}^{pt}}(\bm{y})$.

\begin{defn}
Let $U \subset \mathcal{R}^{pt}(T_S,B)$ be a small connected open subset which makes the strip-like ends constant and avoids a neighbourhood of the pearly trees with some Morse edge not in $B$ having length $0$.
Given $\epsilon>0$, denote by $U_{\epsilon} \subset \mathcal{R}^{pt}$ the image of the map
\[ U \times (-\epsilon,\epsilon)^B \To \mathcal{R}^{pt}\]
obtained by interpreting the parameter in $(-\epsilon,\epsilon)$ corresponding to the edge $e \in B$ as a gluing parameter $\rho_e$ for $e$.
Note that $U_{\epsilon}$ is open in $\mathcal{R}^{pt}$.
\end{defn}

\begin{prop}
\label{prop:charts}
Suppose that $\mathcal{M}_{\mathcal{S}^{pt}}$ is regular.
Then for some $\epsilon>0$ sufficiently small, there is a homeomorphism
\[ \mathcal{M}_{\mathcal{S}^{pt}(T_S,B)|_U} \times (-\epsilon,\epsilon)^B \To \mathcal{M}_{\mathcal{S}^{pt}|_{U_{\epsilon}}}\]
which makes the following diagram commute:
\[ \xymatrix{
\mathcal{M}_{\mathcal{S}^{pt}(T_S,B)|_U} \times (-\epsilon,\epsilon)^B \ar[r] \ar[d] & \mathcal{M}_{\mathcal{S}^{pt}|_{U_{\epsilon}} } \ar[d]\\
U \times (-\epsilon,\epsilon)^B \ar[r] & U_{\epsilon}.}
\]
\end{prop}
\begin{proof}
If two pearls are joined by a Morse edge $e$ of length zero, then they form a nodal disk. 
In a neighbourhood of the node, the Hamiltonian perturbation is identically $0$ and the almost-complex structure is constant, by the conditions we placed on our perturbation datum. 
A standard gluing argument shows that there is a family of pearls with gluing parameter $\rho_e \in (-\epsilon,0]$, converging to this nodal disk.
A standard compactness argument shows that any sequence of pearls with gluing parameter $\rho_e \To 0^-$ converges to such a nodal disk.
More generally, allowing for multiple Morse edges of length $0$, one can show that there is a homeomorphism
\[ \mathcal{M}_{\mathcal{S}^{pt}(T_S,B)|_U} \times (-\epsilon,0]^B \To \mathcal{M}_{\mathcal{S}^{pt}|_{\mathrm{im}(U \times (-\epsilon, 0]^B)}}\]
for some $\epsilon>0$ sufficiently small.

It then follows from Lemma \ref{lem:constedge} that this map extends to a homeomorphism
\[\mathcal{M}_{\mathcal{S}^{pt}(T_S,B)|_U} \times (-\epsilon,\epsilon)^B \To \mathcal{M}_{\mathcal{S}^{pt}|_{U_{\epsilon}}}\]
with the desired properties.
\end{proof}

We have an open cover of $\mathcal{R}^{pt}$ by the sets of the form $U_{\epsilon}$, for some $U \subset \mathcal{R}^{pt}(T_S,B)$ and some $T_S,B$.
Therefore, we have an open cover of $\mathcal{M}_{\mathcal{S}^{pt}}(\bm{y})$ by sets $\mathcal{M}_{\mathcal{S}^{pt}|_{U_{\epsilon}}}(\bm{y})$ which are homeomorphic to smooth manifolds of dimension $k-2 + i(\bm{y})$. 
So they are the charts of a topological manifold structure on $\mathcal{M}_{\mathcal{S}^{pt}}$.
We have proven:

\begin{prop}
\label{prop:msman}
If $\mathcal{M}_{\mathcal{S}^{pt}}(\bm{y})$ is regular, then it has the structure of a topological manifold of dimension
\[ \mathrm{dim}(\mathcal{M}_{\mathcal{S}^{pt}}(\bm{y})) = k-2 + i(\bm{y}).\]
\end{prop}

\begin{rmk}
\label{rmk:orient}
One can show that the embeddings of Proposition \ref{prop:charts} respect orientations, and hence that the manifold $\mathcal{M}_{\mathcal{S}^{pt}}(\bm{y})$ is oriented.
\end{rmk}

\subsection{Consistency and compactness}
\label{subsec:consistent}

\begin{defn}
A universal choice of perturbation data is a choice of perturbation datum for each family $\mathcal{S}^{pt}(\bm{L})$ (for all choices of Lagrangian labels $\bm{L}$).
\end{defn}

\begin{defn} (Compare \cite[Section 9i]{seidel08}) Given a tree $T$ with Lagrangian labels $\bm{L}$, the gluing construction defines a map to a collar neighbourhood of the boundary stratum corresponding to $T$:
\[ \left \{ r \in \bar{\mathcal{R}}^{pt}_T: \rho_e \in (-\epsilon,0]\mbox{ for $e$ a strip edge, and }\rho_e \in (1-\epsilon,1]\mbox{ for $e$ a Morse edge}\right\} \To \bar{\mathcal{R}}^{pt}.\]
Because the perturbation data are standard along the strip-like ends (given by the Floer data), we can glue the perturbation data on the families $\mathcal{S}^{pt}(\bm{L}_v)$, for each vertex $v$ of $T$, together to obtain a perturbation datum $(K_T,J_T,V_T)$ on this collar neighbourhood.
Furthermore, this perturbation datum extends smoothly to the boundary stratum corresponding to $T$.
We say that a universal choice of perturbation data is {\bf consistent} if the perturbation datum $(K,J,V)$ on $\mathcal{S}^{pt}(\bm{L})$ also extends smoothly to the compactification $\bar{\mathcal{S}}^{pt}(\bm{L})$, and agrees with the perturbation datum $(K_T,J_T,V_T)$ on the boundary stratum corresponding to $T$, for all such $\bm{L}$ and $T$.
\end{defn}

\begin{prop}
Consistent universal choices of perturbation data exist.
\end{prop}
\begin{proof}
The proof is essentially the same as \cite[Lemma 9.5]{seidel08}.
\end{proof}

\begin{defn}
\label{defn:stablepearly}
Suppose we have made a consistent universal choice of perturbation data, and all moduli spaces are regular.
Let $\bm{L}$ be a set of Lagrangian labels and $\bm{y}$ an associated set of generators.
A {\bf stable holomorphic pearly tree} consists of the following data:
\begin{itemize}
\item A semi-stable directed planar tree $T$ with Lagrangian labels $\bm{L}$;
\item For each edge $e$ of $T$, a generator $y_e \in CF^*(L_{r(e)},L_{l(e)})$, where $L_{r(e)}, L_{l(e)}$ are the Lagrangian labels to the right and left of $e$ respectively, such that the generators are given by $\bm{y}$ for the external edges;
\item For each stable vertex $v$ (i.e., $v$ has valence $\ge 3$), an element
\[ (r_v,{\bm{u}}_v) \in \mathcal{M}_{\mathcal{S}^{pt}(\bm{L}_v)}(\bm{y}_v),\]
where $\bm{y}_v$ denotes the set of chosen generators for the edges adjacent to $v$;
\item for each vertex $v$ of valence $2$, an element 
\[ u_v \in \mathcal{M}^*_Z(\bm{y}_v).\]
\end{itemize}
We define $\mathcal{M}_{\bar{\mathcal{S}}^{pt}}^T(\bm{y})$ to be the set of all equivalence classes of stable holomorphic pearly trees modeled on the tree $T$.
\end{defn}

\begin{defn}
We define the moduli space 
\[ \bar{\mathcal{M}}_{\bar{\mathcal{S}}^{pt}} (\bm{y}) := \coprod_T \mathcal{M}_{\bar{\mathcal{S}}^{pt}}^T(\bm{y}) \]
of stable holomorphic pearly trees, as a set.
\end{defn}

\begin{prop}
\label{prop:gromcomp}
$\bar{\mathcal{M}}_{\bar{\mathcal{S}}^{pt}}(\bm{y})$ has the structure of a compact topological manifold with corners. 
Its codimension-$d$ strata are the sets $\mathcal{M}_{\bar{\mathcal{S}}^{pt}}^T$ where $T$ has $d$ internal edges. 
In particular, the open stratum (corresponding to the one-vertex tree) is the moduli space $\mathcal{M}_{\mathcal{S}^{pt}}(\bm{y})$.
\end{prop}
\begin{proof}
Observe that each stratum
\[\mathcal{M}_{\bar{\mathcal{S}}^{pt}}^T(\bm{y})\]
has the structure of a smooth manifold, since it is a product of smooth manifolds. 
By standard gluing arguments, there are maps
\[ \mathcal{M}_{\bar{\mathcal{S}}^{pt}}^T(\bm{y}) \times (-\epsilon,0]^{E(T)} \To \bar{\mathcal{M}}_{\bar{\mathcal{S}}^{pt}} (\bm{y}).\]
We define the topology on $\bar{\mathcal{M}}_{\bar{\mathcal{S}}^{pt}} (\bm{y})$ so that all of these maps are continuous. 
This defines a manifold-with-corners structure on the moduli space of stable holomorphic pearly trees.

We prove compactness by considering each underlying tree type $T_S$ for a pearly tree separately. 
Given $T_S$, consider the moduli space of stable holomorphic pearly trees such that, if we contract all edges of length $0$, we get a tree of type $T_S$. 
The space of possible stable pearls corresponding to vertices of $T_S$ is compact, by standard Gromov compactness as in \cite{frauenfelder08}. 
Similarly, the space of possible broken Morse flowlines corresponding to edges of $T_S$ is compact, by standard compactness results in Morse theory as in \cite[Section 2.4]{schwarz}.
Thus, the full moduli space is a closed subset (defined by the incidence conditions of marked boundary points on pearls and ends of edges) of the compact set of all possible pearl and edge maps.
By considering all possible tree types $T_S$, we obtain a covering of $\bar{\mathcal{M}}_{\bar{\mathcal{S}}^{pt}}(\bm{y})$ by a finite number of compact sets, hence the moduli space is compact.
\end{proof}

\subsection{Transversality}
\label{subsec:transv}

\begin{prop}
The moduli spaces $\mathcal{M}_{\mathcal{S}^{pt}}(\bm{y})$ are regular for generic consistent universal choices of perturbation data.
\end{prop}
\begin{proof}
Make a consistent universal choice of perturbation data. 
For each set of Lagrangian labels $\bm{L}$, we show that it is possible to modify the perturbation data $(K,J,V)$ slightly to make our moduli spaces regular. 
In fact it is sufficient only to perturb $(K,J)$, assuming we have already chosen the Floer data $(h_L,g_L)$ to be Morse-Smale for each $L$. 
Our situation is very similar to that considered in \cite[Section 9k]{seidel08}.

A deformation of $(K,J)$ is given by a choice of:
\begin{itemize}
\item $\delta K \in \Omega^1_{\mathcal{S}^p/\mathcal{R}^{pt}}(\mathcal{S}^p,\mathcal{H})$;
\item $\delta J \in C^{\infty}(\mathcal{S}^p,T_J \mathcal{J})$,
\end{itemize}
such that $(\delta K,\delta J)$ vanish on the strip-like ends and $\delta K(\xi)|_{L_C} = 0$ for each $\xi \in TC$, where $C$ is a boundary component of a pearl and $L_C$ its Lagrangian label.

We choose an open set $\Omega \subset \mathcal{S}^p$ such that, for each $r \in \mathcal{R}^{pt}$, $\Omega \cap \mathcal{S}^p_r$ lies within the `thick' region of Definition \ref{defn:pearltop}, and intersects each connected component of the thick region in a non-empty, connected set that intersects each boundary component (see Figure \ref{fig:Omega}). 
To retain consistency of our perturbation datum, we require that $(\delta K,\delta J)$ are zero outside $\Omega$, and extend smoothly to a pair $(\overline{\delta K}, \overline{\delta J})$ defined on $\bar{\mathcal{S}}^p$ which vanish to infinite order along the boundary.

Let $\mathcal{T}$ denote the space of all such $(\delta K, \delta J)$. 
Given $t \in \mathcal{T}$, we can exponentiate it to an actual perturbation datum, and we define
\[ \mathcal{M}^t_{\mathcal{S}^{pt}}(\bm{y})\]
to be the moduli space of holomorphic pearly trees with respect to this perturbation datum.
We define the universal moduli space 
\[ \mathcal{M}^{univ}_{\mathcal{S}^{pt}}(\bm{y}) := \{ (t,r,{\bm{u}}): t \in \mathcal{T}, (r,{\bm{u}}) \in \mathcal{M}^t_{\mathcal{S}^{pt}}(\bm{y})\}.\]

We have the associated universal linearized operators
\[ D^{univ}_{\mathcal{S}^{pt},r,\bm{u}} : \mathcal{T} \oplus T_r \mathcal{R}^{pt} \oplus T_{\bm{u}} \mathcal{B}_{S_r} \To T_{\bm{ev}_m({\bm{u}})} \bm{L}^{F(S)} \oplus (\mathcal{E}_{S_r})_{\bm{u}},\]
given by
\[D^{univ}_{\mathcal{S}^{pt},r,\bm{u}} = D^{def}_{\mathcal{S}^{pt},r,{\bm{u}}}\oplus  D_{\mathcal{S}^{pt},r,{\bm{u}}},\]
where $D_{\mathcal{S}^{pt},r,{\bm{u}}}$ is as defined in Definition \ref{defn:extlin} and
\[D^{def}_{\mathcal{S}^{pt},r,{\bm{u}}} : \mathcal{T} \To (\mathcal{E}_{S_r})_{\bm{u}}\]
takes the derivative of the holomorphic pearly tree equation with respect to changes in the perturbation datum. 
We should really work in a local trivialization of $\mathcal{S}^{pt}$ over a small set $U$, as we did in Section \ref{subsec:pertdatafam}, but we gloss over this point to make things readable.

We claim that the universal operator $D^{univ}_{\mathcal{S}^{pt},r,\bm{u}}$ is surjective. 
Let $S$ denote the pearly tree with modulus $r$.
The codomain of $D^{univ}_{\mathcal{S}^{pt},r,{\bm{u}}}$ is a direct sum
\[ T_{\bm{ev}_m({\bm{u}})} \bm{L}^{F(S)} \oplus L^p(S^p,\Omega^{0,1}_S \otimes u_p^* TX) \oplus \bigoplus_{L \in \bm{L}} L^p(S^e(L),u_L^*TL). \]
The operator $D_{\mathcal{S}^{pt},r,\bm{u}}$ always maps 
\[W^{1,p}(S^e(L),u_L^*TL) \To L^p(S^e(L),u_L^*TL)\]
surjectively, for each $L \in \bm{L}$ (the moduli spaces of Morse flowlines are always regular -- we are not imposing any boundary conditions here).

The space of deformations $\mathcal{T}$ maps surjectively to 
\[ L^p(S^p,\Omega^{0,1}_S \otimes u_p^* TX)\]
(see \cite[Section 9k]{seidel08}).
To complete the proof of surjectivity, we show that the tangent space to the zero set of the universal section
\[ d^{univ}_S|_{S^p} : \mathcal{T} \times  W^{1,p}_{loc}(S^p,X) \To L^p(S^p,\Omega^{0,1}_S \otimes u_p^* TX)\]
maps surjectively to 
\[T_{\bm{ev}_m(\bm{u})} \bm{L}^{F(S)},\]
using a modification of an argument given in \cite[Section 3.4]{mcduffsalamon}. 
The essential observation is that the group of Hamiltonian diffeomorphisms fixing the Lagrangians in $\bm{L}$ acts on the space of perturbation data and associated holomorphic pearly trees with labels $\bm{L}$.

Let $h: S^p \To \mathcal{H}$ be a smooth function which is locally (in the $z$ coordinates on $S^p$) equal to a constant $H \in \mathcal{H}$ outside of $\Omega \cap S^p$, and such that $h|_{C}$ vanishes on the Lagrangian $L_C$, for any boundary component $C$ of $S^p$ with label $L_C$.
Denote by $\phi_z :X \To X$ the time-$1$ flow of the Hamiltonian $h(z)$, for $z \in S^p$. 
Then we can define a map from
\[ \mathcal{T} \times W^{1,p}_{loc}(S^p,X)\]
to itself by
\begin{eqnarray*}
 u_p(z) &\mapsto& \phi_z (u_p(z)), \\
 K(z) &\mapsto& \phi_z ^*K(z) - dh(z), \\
 J(z) &\mapsto& J(z) \circ \phi_z
 \end{eqnarray*}
where $dh(z)$ denotes the differential of $h(z)$ with respect to the coordinates $z$ on $S^p$. 
In particular, $dh(z)$ is supported in $\Omega$, so the new perturbation datum still lies in $\mathcal{T}$.
One can show that this action preserves the section $d^{univ}_S|_{S^p}$ and in particular preserves its zero set.

By our definition of $\Omega$, for each flag $f \in F(S)$ we can choose a curve in $\Omega$ that cuts the pearl containing $m(f)$ into two regions, one of which contains the marked point $m(f)$ and no other punctures or marked points. 
We can make these curves disjoint for different $f$ (see Figure \ref{fig:Omega}).
Then we can define $h_f: S^p \To \mathcal{H}$ which is supported in the region containing $m(f)$, and constant equal to some Hamiltonian $H_f$ in the portion of that region that lies outside of $\Omega$.
By making different choices of the functions $H_f$, we can independently move the points $\phi_{m(f)}(u_p(m(f)))$ in any direction we please, so the linearization of the evaluation map is surjective from the tangent space to the zero set of $d^{univ}_{S}|_{S^p}$ onto $T_{\bm{ev}_m(\bm{u})}\bm{L}^{F(S)}$.
This completes the proof of surjectivity of the universal linearized operator.

\begin{figure}
\centering
\includegraphics[width=0.9\textwidth]{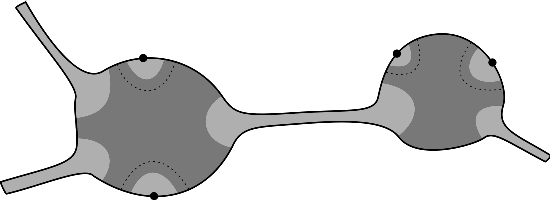}
\caption{The region $\Omega \cap \mathcal{S}^p_r$ (shaded dark grey) inside $\mathcal{S}^p_r$ (shaded light grey), for some $r \in \mathcal{R}^{pt}$. 
Note that $\Omega$ avoids all thin regions. 
The solid circles denote marked points. 
For each marked point $m(f)$, there is a curve inside $\Omega$ (drawn as a dotted line) which separates $m(f)$ from all other marked points and punctures.
\label{fig:Omega}}
\end{figure}

Therefore, the universal moduli spaces $\mathcal{M}^{univ}_{\mathcal{S}^{pt}}$ are Banach manifolds.
Similarly, one can show that the universal moduli spaces
\[ \mathcal{M}^{univ}_{\mathcal{S}^{pt}(T_S,B)}(\bm{y}) :=  \{ (t,r,{\bm{u}}): t \in \mathcal{T},  (r,{\bm{u}}) \in \mathcal{M}^t_{\mathcal{S}^{pt}(T_S,B)}(\bm{y})\}\]
are Banach manifolds for each $(T_S,B)$ (see Definitions \ref{defn:medge}, \ref{defn:tsb}).
The regular values of the projections of each of these universal moduli spaces to $\mathcal{T}$ are of the second category, by the Sard-Smale theorem (see Remark \ref{rmk:frechet}). 
Taking the intersection of regular values of the projection, over all $(T_S,B)$, shows that for a generic choice of deformed perturbation datum in $\mathcal{T}$, the moduli spaces
\[ \mathcal{M}_{\mathcal{S}^{pt}(T_S,B)}(\bm{y})\]
are all simultaneously regular.
This was our definition of regularity of the moduli space $\mathcal{M}_{\mathcal{S}^{pt}}(\bm{y})$ (see Definition \ref{defn:tsbreg}).
\end{proof}

\begin{rmk}
\label{rmk:frechet}
We have glossed over one technical issue: the space of admissible deformed perturbation data is not a Banach space as we have defined it, but rather a Fr\'{e}chet space, and hence the Sard-Smale theorem does not apply.
To fix this, we should work with the Banach spaces of $C^l$ perturbation data, then take the intersection over all $l$ (see \cite[Section 3.1]{mcduffsalamon} for details).
\end{rmk}

\subsection{$A_{\infty}$ structure maps}
\label{subsec:prods}

In this section we give our definition of the Fukaya category. 
We do not discuss signs, but they work in essentially the same way as in \cite{seidel08} (using Remark \ref{rmk:orient}).

We make a choice of Floer data and a consistent universal choice of perturbation data, and assume that all moduli spaces $\mathcal{M}_{\mathcal{S}^{pt}}(\bm{y})$ (as well as those used in the definition of the Floer differential) are regular.

We define the differential 
\[ \mu^1: CF^*(L_0,L_1) \To CF^*(L_0,L_1)\]
to be the standard Floer differential if $L_0,L_1$ are distinct, and to be the Morse differential (for Morse {\bf co}homology) for $(h_L,g_L)$ if $L_0 = L_1  = L$. 

Given Lagrangian labels $\bm{L} = (L_0,\ldots,L_k)$, we define the higher products
\[ \mu^k: CF^*(L_{k-1},L_k) \otimes \ldots \otimes CF^*(L_0,L_1) \To CF^*(L_0,L_k)[2-k]\]
as follows: given an associated set of generators $\bm{y} = (y_0,\ldots,y_k)$, such that
\[ i(y_0) = i(y_1)+ \ldots + i(y_k) + 2 - k,\]
we define the coefficient of $y_0$ in 
\[ \mu^k(y_k,\ldots,y_1)\]
to be the count of points in the moduli space
\[ \mathcal{M}_{\mathcal{S}^{pt}}(\bm{y})\]
(which is $0$-dimensional by Proposition \ref{prop:msman}), with appropriate signs.
Note that the condition on degrees of the $y_j$ mean that the maps $\mu^k$ respect the $\Z$-grading in the appropriate sense for an $A_{\infty}$ category.

\begin{prop}
\label{prop:ainfstruc}
The operations $\mu^k$ satisfy the $A_{\infty}$ associativity equations, with signs and $\Z$-gradings.
\end{prop}
\begin{proof}
The proof follows familiar lines: given a set of generators $\bm{y}$ associated to Lagrangian labels $\bm{L}$, we consider the $1$-dimensional component of the moduli space $\bar{\mathcal{M}}_{\bar{\mathcal{S}}^{pt}}(\bm{y})$. 
The signed count of its boundary components is $0$.
By the results outlined in Section \ref{subsec:consistent}, the codimension-$1$ boundary strata of $\bar{\mathcal{M}}_{\bar{\mathcal{S}}^{pt}}(\bm{y})$ consist of those stable holomorphic pearly trees modeled on trees $T$ with one internal edge. 
The fact that their signed count is $0$ means that the coefficient of $y_0$ in
\[ \sum_{a,b} (-1)^{\star}\mu^{k-a+1}(y_k,\ldots,y_{a+b+1},\mu^a(y_{a+b},\ldots,y_{b+1}),y_b,\ldots,y_1)\]
is $0$, where 
\[ \star = i(y_1)+ \ldots + i(y_b) - b.\] 
This means exactly that the $A_{\infty}$ associativity equations hold.
\end{proof}

\begin{prop}
The Fukaya category is independent, up to quasi-isomorphism, of the choices of strip-like ends, Floer data and perturbation data made in its definition.
\end{prop}
\begin{proof}
Compare \cite[Chapter 10]{seidel08}.
\end{proof}

\begin{prop}
\label{prop:endo}
The $A_{\infty}$ algebra $CF^*(L,L)$ is quasi-isomorphic to the differential graded cohomology algebra $C^*(L)$.
\end{prop}
\begin{proof}
We can choose the Hamiltonian perturbations of the moduli spaces used to define $CF^*(L,L)$ to be zero, so that all pearls are constant by exactness of $L$. 
It is not difficult to show that transversality can be achieved with this class of perturbation data, by perturbing $V$.
The definition of $CF^*(L,L)$ then coincides with the definition of the $A_{\infty}$ algebra $CM^*(L)$ given in \cite[Section 2.2]{abouzaidplumb} (by counting Morse flow trees on $L$). 
The result now follows from \cite[Section 3]{abouzaidplumb}.
\end{proof}

\subsection{Compatibility with other definitions}
\label{subsec:compat}

In this section, we explain why our definition of the Fukaya category (which we denote, for the purposes of this section, by $\mathcal{F}uk^1(X)$) is quasi-equivalent to that in \cite{seidel08} (which we denote by $\mathcal{F}uk^2(X)$).
We define an auxiliary $A_{\infty}$ category, $\mathcal{F}uk^{12}(X)$, which contains two objects, $L^1$ and $L^2$, for each object $L$ of the usual Fukaya category. 
We define Floer data for each pair $(L^1,L^1)$ to consist of a Morse-Smale pair on $L$, but for all other pairs of objects $(L_0^i,L_1^j)$, including the case $L_0 = L_1$, we define the Floer data as if the objects were distinct in Definition \ref{defn:floerdata} (i.e., the Floer datum consists of a Hamiltonian component whose time-$1$ flow makes $L_0$ and $L_1$ transverse, and  an almost-complex structure component).
We define the $A_{\infty}$ structure coefficients by counting holomorphic pearly trees as before, but we only allow Morse flowlines if an edge has labels $L^1$ on opposite sides for some $L$.

There are $A_{\infty}$ embeddings
\[ \mathcal{F}uk^1(X) \hookrightarrow \mathcal{F}uk^{12}(X) \hookleftarrow \mathcal{F}uk^2(X)\]
defined by $L \mapsto L^1$, $L \mapsto L^2$ respectively.

\begin{prop}
The objects $L^1, L^2$ are quasi-isomorphic, for any $L$.
\end{prop}
\begin{proof}
The Piunikhin-Salamon-Schwarz isomorphism \cite{pss} gives isomorphisms on the level of cohomology,
\[ HF^*(L^1,L^2) \cong HF^*(L^2,L^1) \cong HF^*(L^2,L^2) \cong H^*(L),\]
and says that the product
\[ HF^*(L^1,L^2) \otimes HF^*(L^2,L^1) \To HF^*(L^2,L^2)\]
agrees with the cup product on cohomology (note that the moduli spaces defining this product involve no holomorphic pearly trees, only disks).

In particular, if we choose morphisms
\[ f_{12} \in CF^*(L^1,L^2)\mbox{ , }\,\, f_{21} \in CF^*(L^2,L^1)\]
corresponding to the identity in cohomology, then the PSS isomorphism tells us that the product
\[ \mu^2(f_{21},f_{12}) \in CF^*(L^2,L^2)\]
corresponds to the identity in cohomology. 
Thus, because $HF^*(L^1,L^1)$ and $HF^*(L^2,L^2)$ have the same rank (both are isomorphic to $H^*(L)$), the morphisms $f_{12}$ and $f_{21}$ induce isomorphisms on cohomology.

Thus, $L^1$ and $L^2$ are quasi-isomorphic, as required.
\end{proof}

\begin{cor}
The embeddings 
\[ \mathcal{F}uk^1(X) \hookrightarrow \mathcal{F}uk^{12}(X) \hookleftarrow \mathcal{F}uk^2(X)\]
are quasi-equivalences, and in particular, the $A_{\infty}$ categories $\mathcal{F}uk^1(X)$ and $\mathcal{F}uk^2(X)$ are quasi-equivalent.
\end{cor}
\begin{proof}
See \cite[Section 10a]{seidel08}.
\end{proof}

\section{Computation of $\mathcal{A}$}
\label{sec:Acalc}

The aim of this section is to prove Theorem \ref{thm:zcoeffs2}, which identifies the cohomology algebra of $\mathcal{A}$ as an exterior algebra, and Proposition \ref{prop:versal}, which gives a description of $\mathcal{A}$ up to quasi-isomorphism.

The outline of the section is as follows: Section \ref{subsec:fpearlyt} gives a Morse-Bott description of $CF^*(L^n,L^n)$. 
We define an $A_{\infty}$ category $\mathscr{C}$ with two objects: one is the Lagrangian immersion $L^n: S^n \To \mathcal{P}^n$, and the other is the Lagrangian immersion $L': S^n \To \CP{n}$ which is the double cover of the real locus $\RP{n}$. 
The situation is analogous to that in Section \ref{subsec:compat}, in which we explained why our Morse-Bott description of the Fukaya category using pearly trees was equivalent to the standard one using disks. 
Namely, we will define the $A_{\infty}$ structure maps so that the $A_{\infty}$ endomorphism algebra of $L^n$ counts holomorphic disks as in Section \ref{subsec:afuk}, and in particular is the same as $\mathcal{A}$, while the $A_{\infty}$ endomorphism algebra of $L'$ counts Morse-Bott objects which we call `admissible flipping holomorphic pearly trees'. 

Recall that one can think of a pearly tree as a degeneration of holomorphic disks, as the Hamiltonian part of the Floer datum for the pair $(L,L)$ converges to $0$ (see Remark \ref{rmk:limit}). 
Similarly, one should think of $L'$ as the limit of $L^n_{\epsilon}$ as $\epsilon \To 0$, i.e., the double cover of $\RP{n} \subset \CP{n}$ by $S^n$ (recall that $L^n_{\epsilon}$ is constructed as the graph of an exact 2-valued $1$-form $\epsilon df$ in the cotangent disk bundle $D^*_{\eta}\RP{n}$ embedded in $\CP{n}$). 
One should think of a flipping holomorphic pearly tree as a degeneration of a holomorphic pearly tree with boundary on $L^n_{\epsilon}$, in the limit $\epsilon \To 0$. 

Because we wish to consider only holomorphic pearly trees which lie inside $\mathcal{P}^n$ (i.e., do not intersect the boundary divisors), we must impose an additional condition (`admissibility') on our flipping holomorphic pearly trees. 
Thus, although the admissible flipping holomorphic pearly trees themselves may intersect the boundary divisors, they should be thought of as degenerations of holomorphic pearly trees which avoid the boundary divisors.

We show that, for sufficiently small $\epsilon > 0$, the objects $L'$ and $L^n_{\epsilon}$ of this $A_{\infty}$ category are quasi-isomorphic, and hence that we can compute $\mathcal{A} := CF^*(L^n_{\epsilon},L^n_{\epsilon})$ up to quasi-isomorphism by computing $\mathcal{A}' := CF^*(L',L')$. 

In Section \ref{subsec:a'prop} we describe some features of pearly trees, which help us to explicitly identify the moduli spaces of flipping holomorphic pearly trees that give the structure coefficients of $\mathcal{A}'$. 
This is possible because the pearls involved are just holomorphic disks in $\CP{n}$ with boundary on $\RP{n}$ (with some additional restrictions), hence well-understood.

In Section \ref{subsec:calc}, we carry this out. 
In particular, we prove Theorem \ref{thm:zcoeffs2}, which identifies the cohomology algebra of $\mathcal{A}'$ (and hence $\mathcal{A}$) as an exterior algebra. 
We also identify certain higher $A_{\infty}$ structure maps of $\mathcal{A}'$.

Finally, in Section \ref{subsec:def}, we show that $\mathcal{A}'$ is versal in the class of $A_{\infty}$ algebras with cohomology algebra the exterior algebra, and the equivariance and grading properties established in Section \ref{sec:A}.
This identifies $\mathcal{A}'$ (and hence $\mathcal{A}$) up to quasi-isomorphism, in the sense that any $A_{\infty}$ algebra in the same class must be quasi-isomorphic to $\mathcal{A}$.

\subsection{Flipping pearly trees}
\label{subsec:fpearlyt}

For the purposes of this section, we think of $\CP{n}$ as the hyperplane
\[\left\{\sum_j z_j = 0\right\} \subset \CP{n+1},\]
$\RP{n}$ as its real locus, $L': S^n \To \CP{n}$ the composition of the double cover of $\RP{n}$ with the inclusion $\RP{n} \hookrightarrow \CP{n}$, and $\{x_j\}$ the real coordinates on $S^n$.
We define an $A_{\infty}$ category $\mathscr{C}$ with two objects: one is the Lagrangian immersion $L^n: S^n \To \mathcal{P}^n$, and the other is the Lagrangian immersion $L': S^n \To \CP{n}$ just defined. 

\begin{defn}
We define Floer data and morphism spaces for the pairs of objects $(L_0,L_1) = (L^n,L^n), (L^n,L')$ or $(L',L^n)$ as in Definition \ref{defn:gen}.
\end{defn}

\begin{defn}
The Floer datum for the pair $(L',L')$ consists of {\bf two} Morse functions on $S^n$: one is $h$, a function whose only critical points are a maximum $p_{[n+2]}$ and minimum $p_{\phi}$. 
The other is $f$, the function constructed in Definition $\ref{defn:f}$, which has critical points $p_K$ for each proper, non-empty subset $K \subset [n+2]$, as shown in Corollary \ref{cor:critf}. 
Both, when paired with the standard round metric $g$ on $S^n$, form a Morse-Smale pair. 
One then defines
\[CF^*(L',L') := CM^*(h) \oplus CM^*(f) \cong \bigoplus_{K \subset [n+2]} \C \langle p_K \rangle.\]
We equip it with the $\Q$-grading
\[ i(p_K) := \frac{n}{n+2} |K|\]
(compare Corollary \ref{cor:grading}).
\end{defn}

\begin{rmk}
Given a complex volume form $\eta$ on $\mathcal{P}^n$, we can define a $\Z$-grading on the morphism spaces $CF^*(L_0,L_1)$ as usual.
\end{rmk}

\begin{defn}
We call generators of $CF^*(L',L')$ corresponding to critical points of $f$ {\bf flipping generators}, and those corresponding to critical points of $h$ {\bf non-flipping generators}. 
\end{defn}

\begin{defn}
Suppose we are given a set of Lagrangian labels $\bm{L}$, consisting only of the objects $L'$ and $L^n$ of $\mathscr{C}$.
We define a pearly tree with labels $\bm{L}$ to be a pearly tree as in Definition \ref{defn:pearlytree}, except that we only allow edges labeled $L'$ (not $L^n$).
\end{defn}

\begin{defn}
We define a perturbation datum $(K,J,V)$ for the family of pearly trees as in Sections \ref{subsec:pertdatafix}, \ref{subsec:pertdatafam}, with one difference. 
Namely, the part of the perturbation datum $V$ (associated to the edges, which all have label $L'$) now consists of two components: the `flipping component'
\[V^f \in C^{\infty}(S^e, \mathcal{V}_{S^n})\]
and the `non-flipping component'
\[V^{nf} \in C^{\infty}(S^e, \mathcal{V}_{S^n}).\]
We require that
\[V^f = V^{nf} = 0\]
on an internal edge $e$ of length $l_e \le 1$, and
\[ V^f = \nabla f \mbox{ and } V^{nf} = \nabla h\]
on an external edge or an edge $e$ of length $l_e \ge 2$.
\end{defn}

\begin{defn} (Compare Definition \ref{defn:perthol}) Given a set of Lagrangian labels $\bm{L}$ and associated generators $\bm{y}$, we define a {\bf flipping holomorphic pearly tree} with labels $\bm{y}$ to consist of the following data: 
\begin{itemize}
\item A designation of certain edges as {\bf flipping} and the remaining edges as {\bf non-flipping}, such that external flipping edges are labeled by flipping generators and external non-flipping edges are labeled by non-flipping generators. We call the marked points attached to flipping edges {\bf flipping marked points} and those attached to non-flipping edges {\bf non-flipping marked points}; 
\item A smooth map
\[u_e: S^e \To S^n\]
satisfying
\begin{eqnarray*}
Du_e-V^f &=& 0 \mbox{ on flipping edges, and}\\
Du_e - V^{nf} &=& 0 \mbox{ on non-flipping edges;}
\end{eqnarray*}
\item  A smooth map
\[u_p : S^p \To \CP{n}\] 
satisfying
\[ (Du_p - Y)^{0,1} = 0,\]
such that
\[u_p(C) \in \mathrm{im}(L_C) \mbox{ for each boundary component $C$ of $S^p$ with label $L_C$;}\]
\item A lift $\tilde{u}_C$ of the map $u_p|_C: C \To \mathrm{im}(L_C)$ to $S^n$,
\[\xymatrix{
& S^n \ar[d]^{L_C}\\
C \ar[ru]^{\tilde{u}_C} \ar[r]_-{u_p|_C} & \mathrm{im}(L_C),}
\]
for each boundary component $C$ with label $L_C$, 
\end{itemize}
satisfying the following conditions:
\begin{itemize}
\item $\tilde{u}_C$ is continuous except at flipping marked points, where it changes sheets of the covering;
\item We have
\[\tilde{u}^{\pm}(m(f)) = u_e(b(f)) \mbox{ for all $f \in F^{\pm}(S)$,}\]
where we denote by $\tilde{u}^{+}$, respectively $\tilde{u}^-$, the right, respectively left, limit of $\tilde{u}$ (this is necessary because $\tilde{u}$ is discontinuous exactly at the flipping marked points), and where $F^+(S)$, respectively $F^-(S)$, denotes the subset of flags whose orientation agrees, respectively disagrees, with the orientation of the tree; 
\item The external edges are asymptotic to the generators $\bm{y}$, in the same sense as in Definition \ref{defn:ascon}.
\end{itemize}
\end{defn}

Recall that $\mathcal{P}^n$ is obtained from $\CP{n}$ by removing the divisor $D$ which is the union of the divisors $D_j = \{z_j = 0\}$ for $j = 1, \ldots ,n+2$. 
We wish to count only flipping holomorphic pearly trees that do not `intersect' the divisors $D_j$. 
We now explain how to do this in a well-defined way.

\begin{defn}
\label{defn:strip}
Given a flipping holomorphic pearly tree $\bm{u}$ as defined above, one obtains a well-defined homology class $[\bm{u}] \in H_2(\CP{n},L^n)$ as follows:
\begin{itemize}
\item Start with the continuous map $\bm{u}: S \To \CP{n}$ associated with the flipping holomorphic pearly tree. 
\item Glue a thin strip along the boundary of the flipping pearly tree (see Figure \ref{fig:strip});
\item If the boundary component or edge has label $L^n$, then it already gets mapped to $L^n$, so we map the strip into $\CP{n}$ by making it constant along its width. 
\item If the boundary component or edge has label $L'$, then by construction, there is a continuous lift of the boundary of the strip to $S^n$. Namely, it is given by the lift $\tilde{u}_C$ along a boundary component $C$ of a pearl with label $L'$; by a flowline of $\nabla f$ and its antipode along the boundary of a strip coming from a flipping edge; and by a flowline of $\nabla h$ on both sides of the boundary of a strip coming from a non-flipping edge.
\item Thus, we can map the strip into $\CP{n}$ by letting it interpolate between the zero section and the graph of $\epsilon df$ in the Weinstein neighbourhood $D^*_{\eta} S^n$ used in the construction of $L^n_{\epsilon}$. 
Thus, boundary components of the strip with label $L'$ now lie on $L^n_{\epsilon}$.
\end{itemize}

We now define the intersection number $\bm{u} \cdot D_j$ to be the topological intersection number of this class $[\bm{u}] \in H_2(\CP{n},L^n)$ with $D_j \in H_{2n-2}(\CP{n})$.
We say that a flipping holomorphic pearly tree $\bm{u}$ is {\bf admissible} if $\bm{u} \cdot D_j = 0$ for all $j$.
\end{defn}

\begin{prop}
\label{prop:calcint}
Let $\bm{u}$ be a flipping holomorphic pearly tree.
Then the intersection numbers $\bm{u} \cdot D_j$ are non-negative.
Furthermore, in nice situations they can be calculated:
Suppose that the boundary lifts $\tilde{u}_C$ of each boundary component $C$ with label $L'$ are transverse to the real hypersurface $D_j^{\R} \subset S^n$, and no flipping marked points lie on $D_j^{\R}$.
Then one can calculate $\bm{u} \cdot D_j$ by counting the usual intersection number for internal intersections of each pearl $u_v$ with $D_j$ (this is positive by positivity of intersections), $+1$ for each time a flipping edge of $\bm{u}$ crosses $D_j^{\R}$, and $+1$ for each time a boundary lift $\tilde{u}_C$ crosses $D_j^{\R}$ in the negative direction.
\end{prop}
\begin{proof}
We observe that the first statement follows from the second: in the transverse situation the intersection number is non-negative because the only contributions are positive. 
We can put ourselves in the transverse situation by making a small perturbation of the divisor $D_j$. Namely, define a 1-parameter family of divisors
\[ D_j^{t} := \left\{z_j + t \sum_k \alpha_k z_k = 0\right\}\]
for $t \in [0,\delta]$, where $\alpha_j \in \R$ and $\delta > 0$ is real and sufficiently small that the real part $(D_j^t)^{\R}$ remains transverse to the gradient vector field $\nabla f$, and hence $D_j^t$ avoids the Lagrangian $L^n$ (by Lemma \ref{lem:transdiv}). 
We also make $\delta$ small enough that $D_j^t$ avoids all other Lagrangian labels of the flipping holomorphic pearly tree.
Therefore the intersection number $\bm{u}\cdot D_j^t$ remains constant, so we can compute $\bm{u} \cdot D_j$ by computing $\bm{u} \cdot D_j^{\delta}$.
That $D_j^{\delta}$ can be made transverse to the boundary lifts $\tilde{u}_C$ is an easy application of Sard's theorem.
Furthermore, one can easily make $D_j^{\delta}$ avoid all marked points and critical points of pearls (since these are isolated).

Now we prove the second statement.
Internal intersections of $\bm{u}$ with $D_j$ contribute the usual intersection number (which is positive by positivity of intersections, recalling that the almost-complex structure is standard near the divisors $D_j$).
The other intersections happen near boundary components of $\bm{u}$ with label $L'$:
\begin{itemize}
\item If a flipping edge crosses $D_j^{\R}$, one can see that the image of the surrounding strip under projection to the $z_j$ plane looks like Figure \ref{fig:morseint}, hence contributes $+1$ to the intersection number;
\item If a non-flipping edge crosses $D_j^{\R}$, the image of the strip under projection to the $z_j$ plane looks like Figure \ref{fig:morseint} except that the strip gets folded in two, so that both edges get sent to the same sheet of $L^n$, and the contribution to the topological intersection number is $0$;
\item If a boundary lift $\tilde{u}_C$ crosses $D_j^{\R}$ positively, the projection of the strip and nearby disk to the $z_j$ plane looks like Figure \ref{fig:diskint1} (the projection is a holomorphic map, which by assumption has no singularities near the divisor $D_j$, and its boundary crosses $D_j^{\R}$ positively, hence maps to the upper half plane in a neighbourhood of this point).
There is a `fold' along the real axis, and one can see that the contribution to the topological intersection number with $D_j$ is $0$;
\item If a boundary lift $\tilde{u}_C$ crosses $D_j^{\R}$ negatively, the projection of the strip and nearby disk to the $z_j$ plane looks like Figure \ref{fig:diskint2} (as before, because the disk is holomorphic, non-singular, and its boundary crosses $D_j^{\R}$ negatively, it must get sent to the lower half plane in a neighbourhood of this point).
Thus the contribution to the topological intersection number with $D_j$ is $+1$.
\end{itemize}
This completes the proof.
\end{proof}

\begin{cor}
\label{cor:adm}
In an admissible flipping holomorphic pearly tree, the flipping edges can not cross the hypersurfaces $D_j^{\R}$ and the boundary lifts can only cross $D_j^{\R}$ in the positive direction.
\end{cor}

\begin{figure}
\centering
\subfigure[Adding a strip to a flipping pearly tree, to define its homology class in $H_2(\CP{n},L^n)$.]{
\includegraphics[width=0.45\textwidth]{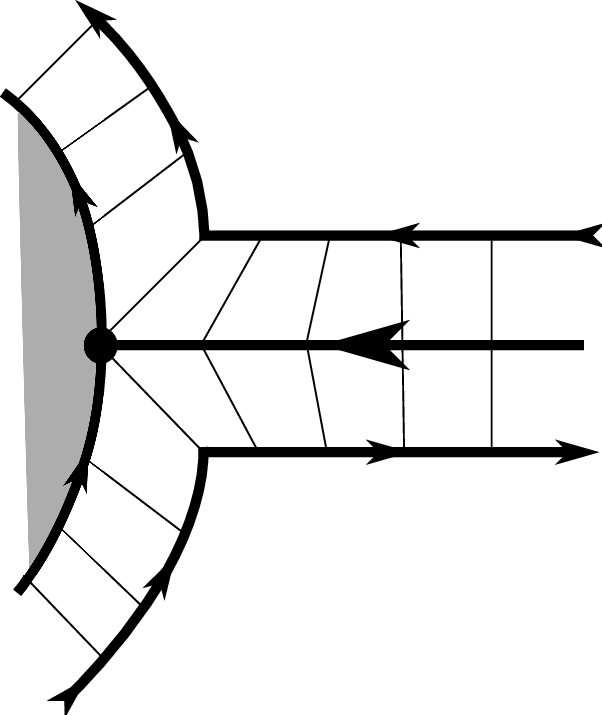}
\label{fig:strip}}
\hfill
\subfigure[Projection of the strip surrounding a flipping edge crossing the hypersurface $D_j^{\R}$ transversely, to the $z_j$ plane. The topological intersection number with $D_j$ (which corresponds to the point $0$ in this projection, drawn as a solid circle) is $+1$.]{
\includegraphics[width=0.45\textwidth]{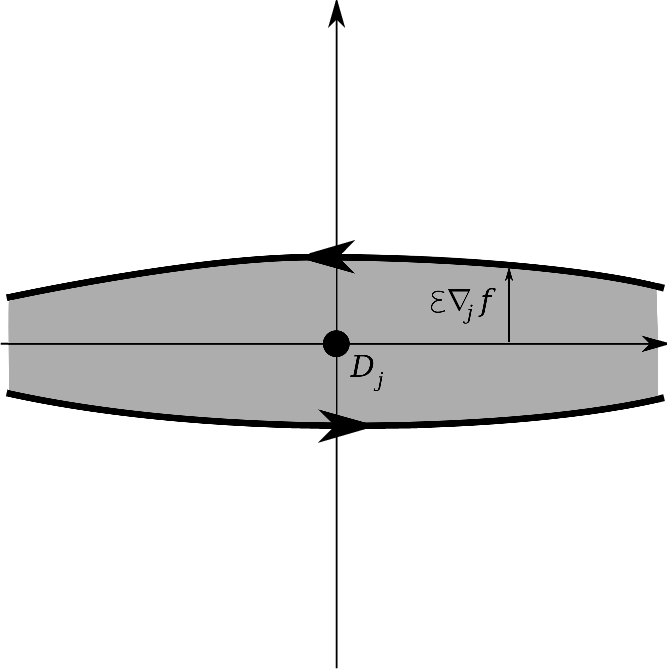}
\label{fig:morseint}}
\hfill
\subfigure[Projection of part of the disk and strip near a positive crossing of a boundary lift $\tilde{u}_C$ with $D_j^{\R}$, to the $z_j$ plane. There is a `fold' along the real axis, so the topological intersection number with $D_j$ is $0$.]{
\includegraphics[width=0.45\textwidth]{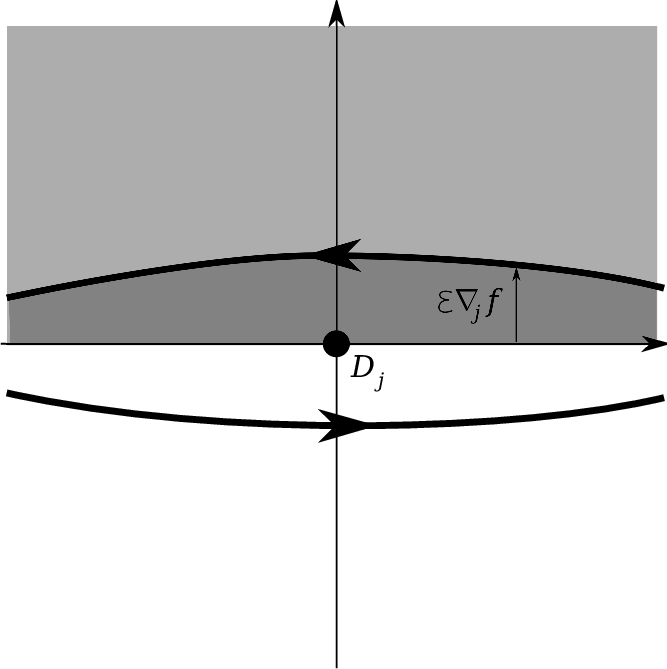}
\label{fig:diskint1}}
\hfill
\subfigure[Projection of part of the disk and strip near a negative crossing of a boundary lift $\tilde{u}_C$ with $D_j^{\R}$, to the $z_j$ plane. The topological intersection with $D_j$ is $+1$.]{
\includegraphics[width=0.45\textwidth]{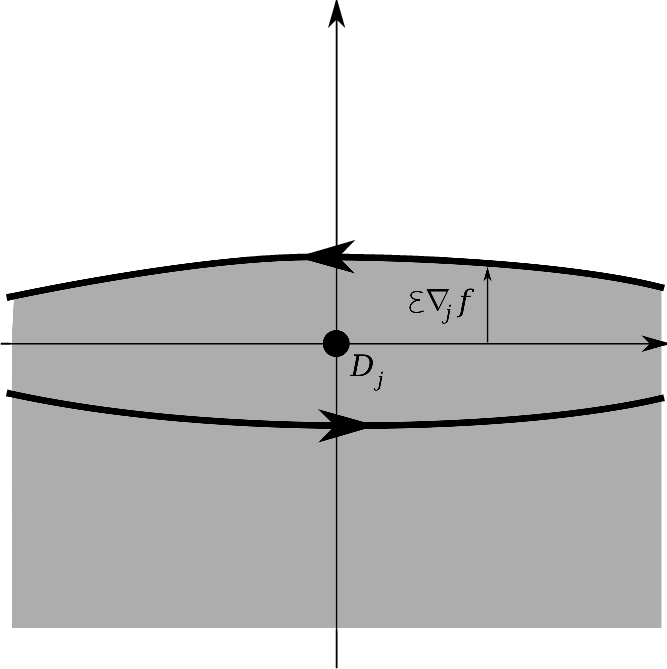}
\label{fig:diskint2}}
\caption{Defining and calculating $\bm{u} \cdot D_j$.
\label{fig:calcint}}
\end{figure}

\begin{defn}
We define the moduli space $\mathcal{M}_{\mathcal{S}^{fpt}}(\bm{y})$ of  admissible flipping holomorphic pearly trees with asymptotic conditions $\bm{y}$, by analogy with Definition \ref{defn:ascon}.
\end{defn}

\begin{rmk}
\label{rmk:smoothapprox}
We remark that it follows from the proof of Proposition \ref{prop:calcint} that, if $\bm{u}$ is an admissible flipping holomorphic pearly tree, then its homology class $[\bm{u}]$ can be represented by a smooth disk in $\mathcal{P}^n$ with boundary on $L^n$. 
Namely, we perturb the divisors $D_j$ to put ourselves in the transverse situation as described. 
The disk defining $[\bm{u}]$ can only intersect the divisors $D_j$ when a boundary lift $\tilde{u}_C$ crosses $D_j^{\R}$ in the positive direction. 
It is obvious from Figure \ref{fig:diskint1} that the disk can be perturbed to avoid the divisor in this case. 
\end{rmk}

It follows that admissible flipping pearly trees inherit any properties of holomorphic disks in $\mathcal{P}^n$ with boundary on $L^n$ that depend only on the topology. 
For example, the energy of an admissible flipping holomorphic pearly tree is given by the differences of symplectic action functionals of input and output generators, and in particular is constant in the moduli space  $\mathcal{M}_{\mathcal{S}^{fpt}}(\bm{y})$.
Furthermore, we can prove the following:

\begin{prop}
\label{prop:dimfpt}
Suppose that $\bm{L}$ is a set of Lagrangian labels and $\bm{y}$ an associated set of generators.
Then, for generic choice of perturbation data, $\mathcal{M}_{\mathcal{S}^{fpt}}(\bm{y})$ is a manifold of dimension
\[\mathrm{dim}(\mathcal{M}_{\mathcal{S}^{fpt}}(\bm{y})) = i(\bm{y}) +  k - 2.\]
\end{prop}
\begin{proof}
The proof follows that of Proposition \ref{prop:msman} -- we must construct charts from the moduli spaces $\mathcal{M}_{\mathcal{S}^{fpt}(T_S,B)}(\bm{y})$ for each $(T_S,B)$ as in Definition \ref{defn:tsb}, and glue the pieces $\mathcal{M}_{\mathcal{S}^{fpt}(T_S,B)}(\bm{y}) \times (-\epsilon, \epsilon)^B$ together to obtain a manifold, using an analogue of Proposition \ref{prop:charts}.

The dimension is given by the index of the Fredholm operator used to cut out the moduli space. 
One might worry that the index theory of Cauchy-Riemann operators depends on a choice of holomorphic volume form $\eta$ on $\mathcal{P}^n$, and our holomorphic pearls can intersect the boundary divisors $D_j$, where $\eta$ is not defined. 
However, this is dealt with by Remark \ref{rmk:smoothapprox}, which shows how to construct a smooth disk in $\mathcal{P}^n$ with boundary on $L^n$, near any given admissible holomorphic flipping pearly tree. 
One can show that the Fredholm index of the operator cutting out the moduli space of flipping pearly trees is equal to the index of the pseudo-holomorphic curve equation on the nearby disk, which depends only on the homology class of the disk in $\mathcal{P}^n$ relative to its Lagrangian boundary conditions. 
This is sufficient to prove the dimension formula.

Now observe that, when a new Morse edge with label $L'$ is created as in Figure \ref{fig:morseedge}, there are two possibilities: either the lifts $\tilde{u}$ of the two boundary components of the strip on the left are antipodes, in which case a flipping edge is created, or they coincide, in which case a non-flipping edge is created.
With this convention, the gluing maps of Proposition \ref{prop:charts} define boundary lifts $\tilde{u}_C$ as well as the map $u_p$. 
They also preserve the homology class of Definition \ref{defn:strip}, and hence admissibility.
\end{proof}

\begin{defn}
We define a {\bf stable flipping holomorphic pearly tree} by analogy with the definition of stable pearly trees (Definition \ref{defn:stablepearly}). 
The only difference is for edges of trees $T$ with both sides labeled $L'$: these can be broken Morse flowlines of $f$ (for flipping edges) or $h$ (for non-flipping edges).
We define a {\bf stable admissible flipping holomorphic pearly tree} to be a stable flipping holomorphic pearly tree, each component of which is admissible.
\end{defn}

\begin{rmk}
We observe that the admissibility condition rules out sphere bubbling in families of admissible flipping holomorphic pearly trees: any sphere bubble must have intersection number $0$ with the divisors $D_j$ by admissibility, and hence have trivial homology class. 
But then its symplectic area is $0$, so it must be constant.
\end{rmk}

\begin{prop}
The moduli space of stable admissible flipping holomorphic pearly trees has the structure of a compact manifold with corners.
\end{prop}
\begin{proof}
As in Section \ref{subsec:afuk}, we run into the problem that we can not appeal to a Gromov compactness theorem for immersed Lagrangians. 
Furthermore, we can not bypass this problem by passing to the cover $\widetilde{\mathcal{P}}^n$ of $\mathcal{P}^n$ defined in Corollary \ref{cor:cover}, as we did in Section \ref{subsec:afuk}, because the image of the Lagrangian immersion $L'$ does not lie in $\mathcal{P}^n$. 
Even if we considered the corresponding branched cover of $\CP{n}$ (branched around the divisors $D_j$), the Lagrangian immersion $L'$ would only lift to a piecewise smooth embedded Lagrangian, with `edges' along the branching divisors $D_j$. 
Again, there is no Gromov compactness theorem that deals with piecewise smooth Lagrangians.

Instead, consider the quadric
\[ Q^n := \left\{ \sum_{j=0}^{n+2} z_j^2 = 0, \sum_{j=1}^{n+2} z_j = 0\right\} \subset \CP{n+2},\]
and the branched double cover 
\begin{eqnarray*}
\rho: Q^n &\To & \CP{n} \\
\rho([z_0:\ldots:z_{n+2}]) &=& [z_1:\ldots:z_{n+2}].
\end{eqnarray*}
The cover is branched along the divisor
\[ \widetilde{Q}^n := \left\{ \sum_j z_j^2 = 0\right\} \subset \CP{n}.\] 

The real locus of $Q^n$ in the affine chart $z_0 = i$ is the unit sphere $S^n$, and $\rho|_{S^n}$ is the double cover of the real locus $\RP{n}$ of $\CP{n}$. 
It is well-known that there is a symplectomorphism
\[ T^* S^n \To Q^n \setminus \{z_0 = 0\},\]
sending the zero section to the real locus. 
This sends the radius-$\eta$ disk bundle $D^*_{\eta}S^n$ to a neighbourhood of $\RP{n}$, as in the construction of $L^n$ (Section \ref{subsec:ln}).
Thus, the lifts of $L^n$ and $L'$ to $T^* S^n \subset Q^n$ are embedded.
$L^n$ lifts as the graphs of the exact one-forms $\pm \epsilon df$, and $L'$ lifts to the zero section via the identity and via the antipodal map.

For any flipping holomorphic pearly tree $ \bm{u} \in \mathcal{M}_{\mathcal{S}^{fpt}}(\bm{y})$, the topological intersection number $[\bm{u}] \cdot \widetilde{Q}^n$ depends only on the generators $\bm{y}$ (compare Proposition \ref{prop:pearlenergy}). 
We can arrange that positivity of intersection with $\widetilde{Q}^n$ holds in our moduli space, for appropriate choice of perturbation datum, and then each flipping holomorphic pearly tree in the moduli space intersects $\widetilde{Q}^n$ some finite number of times, which is bounded above by the topological intersection number. 
Then the lifts of flipping holomorphic pearly trees $ \bm{u} \in \mathcal{M}_{\mathcal{S}^{fpt}}(\bm{y})$ to the branched cover $Q^n$ are branched over some finite number of points, hence have bounded genus. 
Gromov compactness for curves with bounded genus and boundary (see, for example, \cite{ye94,pansu}) then implies that the lifted family has a convergent subsequence, which corresponds to a convergent subsequence downstairs. 

This shows that a sequence of admissible flipping holomorphic pearly trees has a subsequence converging to a stable flipping holomorphic pearly tree whose intersection number with each divisor $D_j$ is $0$. 
The intersection number of the stable flipping holomorphic pearly tree with $D_j$ is the sum of intersection numbers of each component flipping holomorphic pearly tree with $D_j$. 
Since these are all non-negative by Proposition \ref{prop:calcint}, they must all be $0$. 
Thus the limit stable flipping holomorphic pearly tree is also admissible, and we have proven compactness.
\end{proof}

We define $A_{\infty}$ structure maps $\mu^k$ as in Section \ref{subsec:prods}, by counting rigid flipping holomorphic pearly trees. 
The proof that they satisfy the $A_{\infty}$ associativity equations essentially follows that of Proposition \ref{prop:ainfstruc}. 
The proof that the $A_{\infty}$ product is $\Q$-graded relies on Proposition \ref{prop:dimfpt}.

\begin{prop}
\label{prop:ll'}
For sufficiently small $\epsilon >0$, the objects $L'$ and $L^n_{\epsilon}$ are quasi-isomorphic.
\end{prop}
\begin{proof}
We observe that $\RP{n}$ and $L^n_{\epsilon}$ intersect transversely in the points $p_K$. 
Therefore we can choose the Hamiltonian component of the Floer datum for the pairs $(L', L^n)$ and $(L^n,L')$ to be $0$. 
The morphism space $CF^*(L',L^n)$ is generated by pairs of points $(p,q) \in S^n \times S^n$ that get sent to the same point by the respective Lagrangian immersions defining $L', L^n$. 
Thus $p$ is a critical point of $f$, and $q$ is either equal to $p$ or its antipode. 
As we saw in Corollary \ref{cor:critf}, there is a critical point $p_K$ of $f$ for each proper non-empty subset $K \subset [n+2]$.
Therefore, we can label the generators of $CF^*(L',L^n)$ as $p_K^{M} := (p_K,p_K)$ and $p_K^S:= (p_K,a(p_K))$ ($M$ stands for `Morse' because the generators $p_K^M$ correspond to the Morse cohomology of $L^n$, and $S$ stands for `self-intersection' because the generators $p_K^S$ correspond to the self-intersections of $L^n$). 
So, additively, 
\[CF^*(L',L^n) \cong CM_M^*(f) \oplus CM_S^*(f)\]
and similarly for $CF^*(L^n,L')$.
One can check that the gradings of these generators are
\[ i(p_K^S) = \frac{n}{n+2}|K|, \,\,\,\, i(p_K^M) = n - \mu_M(p_K) = n+1-|K|.\]

Now observe that we have natural inclusions
\begin{eqnarray*}
CM_M^*(f) &\overset{\varphi_1}{\hookrightarrow}& CF^*(L',L^n), \\
CM_M^*(f) & \overset{\varphi_2}{\hookrightarrow}& CF^*(L^n,L')
\end{eqnarray*}
as graded vector spaces.

\begin{lem}
\label{lem:morsepearly}
For sufficiently small $\epsilon > 0$, the inclusions $\varphi_j$ are chain maps.
\end{lem}
\begin{proof}
We first observe that, for sufficiently small $\epsilon >0$, the holomorphic strips 
\[u: Z \To \CP{n}\]
used to define the differential 
\[\mu^1: CF^*(L',L^n) \To CF^*(L',L^n)\]
must remain entirely within the Weinstein neighbourhood $D^*_{\eta} \RP{n}$ used in the construction of $L^n_{\epsilon}$.
To see why, suppose that $u$ passes through some point $p$ of distance $> \eta$ from $\RP{n}$. 
Then for sufficiently small $\epsilon>0$, the ball $B(p; \eta /2)$ is disjoint from $L^n_{\epsilon}$ and $L'$.
Therefore, by the monotonicity lemma (see \cite[3.15]{lawson74}), the symplectic area of the intersection of $u$ with the ball $B(p;\eta / 2)$ is at least $c (\eta / 2)^2$ for some constant $c$. 
However, the symplectic area of $u$ is given by the difference in symplectic actions of the generators (see Remark \ref{rmk:smoothapprox} and its sequel), which is proportional to $\epsilon$ and hence can be made arbitrarily small. 
Thus, for sufficiently small $\epsilon>0$, the strips never leave the Weinstein neighbourhood $D^*_{\eta} \RP{n}$.

Now we observe that any strip $u$ contributing to the differential on $CF^*(L',L^n)$ lifts to the double cover $D^*_{\eta} S^n \To D^*_{\eta} \RP{n}$, because it comes equipped with a lift of one boundary component to $S^n$ by definition. 
This lifted strip contributes to the differential
\[ \mu^1: CF^*(S^n, \Gamma (\epsilon df) ) \To CF^*(S^n, \Gamma (\epsilon df) )\]
in the Fukaya category of $T^* S^n$.
Conversely, any strip $u$ contributing to the differential on $CF^*(S^n,\Gamma(\epsilon df))$ projects to a strip contributing to the differential on $CF^*(L', L^n_{\epsilon})$. 
The only thing to check is that these projected strips are all admissible -- for this one needs a certain amount of control on the topology of $u$. 
It was proven in \cite[Proposition 9.8]{fukayaoh} that, given $\delta > 0$, there exists $\epsilon_0 > 0$ such that for any strip contributing to the differential on $CF^*(S^n, \Gamma (\epsilon df) )$, with $\epsilon < \epsilon_0$, there is a Morse flowline of $f$,
\[ \gamma: \R \To S^n\]
such that
\[ d(u(s,t), \gamma(\epsilon s)) < \delta \mbox{ for all $s,t$.}\]
Because Morse flowlines of $f$ cross the hypersurfaces $D_j^{\R}$ positively, it follows from Proposition \ref{prop:calcint} that all such strips are admissible.

It follows that the inclusion 
\[CF^*(S^n, \Gamma (\epsilon df) ) \hookrightarrow CF^*(L',L^n)\]
(where the left hand side is a morphism space in the Fukaya category of $T^*S^n$ and the right hand side is a morphism space in the Fukaya category of $\mathcal{P}^n$ as we have defined it) is a chain map. 
Now the Lagrangians $S^n, \Gamma(\epsilon df)$ in $T^* S^n$ are Hamiltonian isotopic, hence quasi-isomorphic in the Fukaya category of $T^* S^n$. 
So there is a quasi-isomorphism
\[CF^*(S^n, \Gamma (\epsilon df) ) \cong CF^*(S^n,S^n) \cong CM^*(f)\]
(the second quasi-isomorphism comes from Proposition \ref{prop:endo}).
Thus, there is a chain map
\[ CM_M^*(f) \cong CF^*(S^n, S^n) \cong CF^*(S^n, \Gamma( \epsilon df)) \hookrightarrow CF^*(L',L^n)\]
as required.
\end{proof}

Now consider the elements
\[ f_1 \in CF^*(L^n,L') , \,\,\,\, f_2 \in CF^*(L',L^n)\]
that correspond to the identity in $CM^*_M(S^n)$. 
Explicitly,
\[ f_1 = \sum_{j=1}^{n+2} p_{\{j\}}^M\]
(and the same for $f_2$). 

\begin{lem}
For sufficiently small $\epsilon > 0$, we have 
\begin{eqnarray*}
\mu^1(f_j) &=& 0 \mbox{ for $j=1,2$, and}\\
\mu^2(f_1,f_2) &=& p_{\phi} \in CF^*(L',L')
\end{eqnarray*}
\end{lem}
\begin{proof}
The fact that $\mu^1(f_j) = 0$ follows from Lemma \ref{lem:morsepearly}. 
We now prove that $\mu^2(f_1,f_2) = p_{\phi}$.

Observe that $i(f_1) = i(f_2) = 0$, so $i(\mu^2(f_1,f_2)) = 0$. 
Therefore, $p_{\phi}$ is the only term that can appear in the product $\mu^2(f_1,f_2)$. 
Its coefficient is the signed count of points in the moduli space of flipping holomorphic pearly trees which are holomorphic strips  running between some intersections $p^M_{\{j\}}$ and $p^M_{\{k\}}$ of $L^n$ and $L'$, with one marked point on the boundary labeled $L'$ which gets sent to $p_{\phi}$ (see Figure \ref{fig:comp}). 
As we saw in the proof of Lemma \ref{lem:morsepearly}, such strips must lie inside the Weinstein neighbourhood $D^*_{\eta} \RP{n}$, and lift canonically to the double cover $D^*_{\eta}S^n$. 
The lift is a holomorphic pearly tree contributing to the product
\[ \mu^2: CF^*(\Gamma( \epsilon df),S^n) \otimes CF^*(S^n, \Gamma( \epsilon df)) \To CF^*(S^n,S^n).\]
Conversely, by the same argument as in the proof of Lemma \ref{lem:morsepearly}, any holomorphic pearly tree contributing to this product projects to an admissible flipping holomorphic pearly tree contributing to the product $\mu^2(f_1,f_2)$. 

\begin{figure}
\centering
\includegraphics[width=0.9\textwidth]{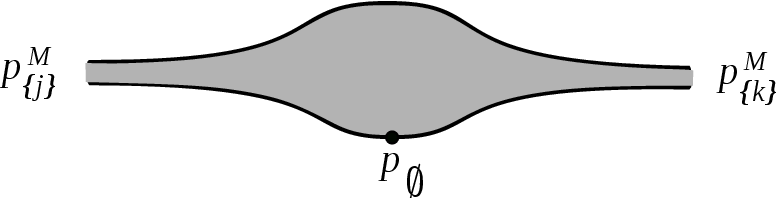}
\caption{The flipping holomorphic pearly trees whose count gives the coefficient of $p_{\phi}$ in $\mu^2(f_1,f_2)$. The solid circle denotes a non-flipping point.
The upper half of the boundary gets sent to $L^n$, and the lower half to $L'$.
\label{fig:comp}}
\end{figure}

It now follows from the quasi-isomorphisms (in the Fukaya category of $T^*S^n$)
\[ CF^*(S^n, \Gamma( \epsilon df)) \cong CF^*(S^n,S^n) \cong CF^*( \Gamma( \epsilon df),S^n)\]
and
\[CF^*(S^n,S^n) \cong CM^*(S^n)\]
that, on the level of cohomology,
\[ [\mu^2(f_1,f_2)] = [p_{\phi}]\]
(product of identity with identity is identity in $CM^*(S^n)$).
But $CF^0(L',L')$ has only the single generator $p_{\phi}$, so we have
\[ \mu^2(f_1,f_2) = p_{\phi}\]
as required. 
\end{proof}

Because $CF^*(L',L')$ and $CF^*(L^n,L^n)$ have the same rank (by Corollary \ref{cor:arank}), it follows that $f_1$ and $f_2$ induce mutually inverse isomorphisms on the level of cohomology, and therefore are mutually inverse quasi-isomorphisms in the category $\mathscr{C}$.
This completes the proof that $L'$ and $L^n$ are quasi-isomorphic, for sufficiently small $\epsilon > 0$.
\end{proof}

\subsection{Properties of the $A_{\infty}$ algebra $\mathcal{A}' := CF^*(L',L')$}
\label{subsec:a'prop}

We define the $A_{\infty}$ algebra $\mathcal{A}' := CF^*(L',L')$. 
It follows from Proposition \ref{prop:ll'} that $\mathcal{A}$ and $\mathcal{A}'$ are quasi-isomorphic $A_{\infty}$ algebras. 
Henceforth we will only be concerned with computing the $A_{\infty}$ structure of $\mathcal{A}'$. 
In particular, we will assume that our flipping holomorphic pearly trees have all boundary components labeled $L'$. 

\begin{lem}
\label{lem:pearltop}
If $\bm{u}$ is an admissible flipping holomorphic pearly tree with associated morphisms $\bm{y} = (p_{K_0},\ldots,p_{K_k})$, then
\[ \sum_{j=1}^k e_{K_j} = e_{K_0}\]
in $M$.
\end{lem}
\begin{proof}
The proof is identical to that of Proposition \ref{prop:top}, since the proof relies only on the homology class $[\bm{u}] \in H_2(\CP{n},L^n)$, which is determined by the admissibility condition.
\end{proof}

\begin{lem}
\label{lem:a'prop}
$\mathcal{A}'$ inherits the following properties of $\mathcal{A}$:
\begin{itemize}
\item It is $\mathbb{T}$-equivariant in the same sense as in Corollary \ref{cor:tact}, by Lemma \ref{lem:pearltop};
\item It has the $\Q$-grading given by $n/(n+2)$ times the normal $\Z$-grading, as in Corollary \ref{cor:grading};
\item As a consequence of these two properties, it satisfies the analogue of Corollary \ref{cor:prodzero}, namely the only non-zero $A_{\infty}$ products are $\mu^{2+nq}$ for $q \in \Z_{\ge 0}$;
\item It satisfies the analogue \ref{cor:cohsigns} (i.e., it is supercommutative).
\end{itemize}
\end{lem}

We now establish some results about flipping holomorphic pearly trees which will be used in Section \ref{subsec:calc} to identify the moduli spaces that give rise to the $A_{\infty}$ structure coefficients of $\mathcal{A}'$.

\begin{prop}
\label{prop:pearlenergy}
For $K \subset [n+2]$, define
 \[ |K|' = \left\{\begin{array}{ll}
			\frac{n+2}{2} & K = \phi, [n+2] \\
			|K| & \mbox{ otherwise.}
			\end{array}\right.
\]
If $\bm{u}$ is an admissible flipping holomorphic pearly tree with labels $\bm{y} = (p_{K_0},\ldots,p_{K_k})$, then the homology class of $\bm{u}$ in $H_2(\CP{n},\RP{n}) \cong \Z$ is given by the formula
\[d_{\bm{u}} =  2\frac{|K_0|' - \sum_{j=1}^k |K_j|' }{n+2} + k-1. \]
\end{prop}
\begin{proof}
Note that the Fubini-Study symplectic form $\omega$ acts on $H_2(\CP{n},\RP{n})$, with value $2\pi$ on the generator.
It follows that 
\[ \omega(\bm{u}) = 2\pi d_{\bm{u}},\]
so we can compute $d_{\bm{u}}$ by computing $\omega(u)$.
 
Recall that we add a strip to $\bm{u}$ to obtain a disk $\tilde{u}:(D,\partial D) \To (\CP{n},L^n)$.
Note that the symplectic area of the strip we add is $\mathcal{O}(\epsilon)$.
So we can compute $\omega(\bm{u})$ by evaluating $\omega(\tilde{u})$ in the limit $\epsilon \To 0$.

The Fubini-Study form is given by the K\"{a}hler potential
\[ \rho = \log \left( \sum_{j=1}^{n+2} \left| \frac{z_j}{z_1} \right|^2 \right) = \log \left( \sum_{j=1}^{n+2} e^{2r_j} \right) - 2 r_1\]
on $\CP{n+1}\setminus D_1$, where $z_j = \exp(r_j + i\theta_j)$.
Thus
\[\omega = dd^c \rho,\]
(recall that $d^c \rho = d\rho \circ J$), so we define
\begin{eqnarray*}
\alpha &=& d^c \rho \\
 &=& \frac{ \sum_{j=1}^{n+2} 2 e^{2r_j} d^c r_j}{\sum_{j=1}^{n+2} e^{2r_j}} - 2 d^c r_1\\
 &=& -2\frac{\sum_{j=1}^{n+2} e^{2r_j} d \theta_j}{\sum_{j=1}^{n+2} e^{2r_j}} + 2 d \theta_1.
 \end{eqnarray*}
 Then $\omega = d \alpha$.
 Of course this is really $\pi^* \alpha$, where $\pi: \C^{n+2} - \{ 0\} \To \CP{n+1}$ is the projection.
 
Because $\tilde{u} \cdot D_1 = 0$ by admissibility, we can deform $\tilde{u}$ to avoid $D_1$ then apply Stokes' theorem to obtain
\[ \int_{\partial D} \tilde{u}^* \alpha = \int_D \tilde{u}^* \omega. \]
 
 Now recall the lift of $L^n$ to $\C^{n+2}$ that arose in the construction of $L^n$, namely
 \begin{eqnarray*}
 \left\{ \sum_{j=1}^{n+2} x_j^2 = 1, \sum_{j=1}^{n+2} x_j = 0\right\} & \To & \C^{n+2} \\
 (x_1,\ldots, x_{n+2}) & \mapsto & \left(x_1 + i\epsilon f_1,\ldots,x_{n+2} + i\epsilon f_{n+2}\right) + \mathcal{O}(\epsilon^2).
 \end{eqnarray*}
 We can lift $\partial D$ to $\C^{n+2}$ (the result will not be a cycle, because when $\partial D$ changes sheets of $L^n$ the lift stops and reappears at the antipode).
 Call the lift $l$.
 Then
 \[ \int_{\partial D} \alpha = \int_{\pi_* l} \alpha = \int_l \pi^* \alpha.\]
 
 Observe that on the lift of $L^n$, $d\theta_k$ is small everywhere except for when $r_k$ is small, and when $r_k$ is small then
 \[ \frac{e^{2r_k}}{\sum_{j=1}^{n+2} e^{2r_j}}\]
 is small.
 Thus the first term in $\pi^* \alpha$ is negligible.
 So
 \[ \int_l \pi^* \alpha =  \int_l 2 d\theta_1 + \mathcal{O}(\epsilon).\]
 
 The projection of the lift of the point $p_K$ to the angular variables is $\pi e_K$ (now thought of as living in $\widetilde{M}_{\R}$ rather than $M_{\R}$).
 Thus, as the lift of $\partial D$ travels from $p_{K_i}$ to $p_{\bar{K}_{i+1}}$, the contribution to the integral is (to order $\epsilon$)
 \[\int_{p_{K_j}}^{p_{\bar{K}_{j+1}}} 2 d\theta_1 = 2\pi e_1 \cdot \left(e_{\bar{K}_{j+1}} - e_{K_j}\right).\]
 An exception occurs when $K_j \mbox{ (respectively }\bar{K}_{j+1}) = \phi \mbox{ or }[n+2]$, in which case $p_{K_j}\mbox{ (respectively }p_{\bar{K}_{j+1}}) $ represents the bottom or top cohomology class of $L^n$, so $\partial D$ does not change sheets of $L^n$ as it passes through $p_{K_j}\mbox{ (respectively }p_{\bar{K}_{j+1}}) $.
 In this case we should simply replace $e_1\cdot (e_{K_j})\mbox{ (respectively }e_1 \cdot (e_{\bar{K}_{j+1}}))$ in the expression above by $0$.
 
For the moment, assume that $K_j \neq \phi \mbox{ or } [n+2]$. 
Adding up and regrouping the contributions of each part of $\partial D$, and recalling that $p_{K_0}$ is the `outgoing' point, we obtain:
 \begin{eqnarray*}
 \int_{\partial D} \alpha & = & 2 \pi e_1\cdot \left(e_{K_0} - e_{\bar{K}_0} + \sum_{j=1}^k e_{\bar{K}_j} - e_{K_j} \right) + \mathcal{O}(\epsilon) \\
 &=& 2\pi \left( 2e_1\cdot \left( e_{K_0} - \sum_{j=1}^k e_{K_j} \right) + k-1 \right) + \mathcal{O}(\epsilon)\\
 &=& 2\pi \left( 2\frac{ |K_0| - \sum_{j=1}^k |K_j| }{n+2} + k-1 \right)+ \mathcal{O}(\epsilon)
 \end{eqnarray*}
 (in the last step we used the fact that the vector is a multiple of $e_{[n+2]}$ by Proposition \ref{prop:top}).
 
 Now if $K_j = \phi \mbox{ or } [n+2]$, recall that we must replace  $e_1\cdot (e_{\bar{K}_j} - e_{K_j})$ by $0$ in the first two lines above.
 This is equivalent to replacing $|K_j|$ by $|K_j|'$ in the final line.
This completes our proof.
 \end{proof}

\begin{defn}
Given an admissible flipping holomorphic pearly tree, it is useful to label certain points on its boundary with proper, non-empty subsets of $[n+2]$, as follows:
At each flipping marked point, the boundary immediately before and after the point get sent (by the lift $\tilde{u}$ of the boundary) to antipodal points of $S^n \setminus D^{\R}$.
Thus they lie in the antipodal regions $S^n_K, S^n_{\bar{K}}$ respectively, for some $K \subset [n+2]$ (recall that $S^n_K$ is defined to be the region where $x_j<0$ for $j \in K$ and $x_j>0$ for $j \notin K$). 
We will ignore the case where a flipping marked point lies on some $D_j^{\R}$, but it presents no real additional problem in our subsequent arguments.
We label the point immediately before our flipping marked point with $K$, and the point immediately after with $\bar{K}$.
Non-flipping marked points do not get labels. 
\end{defn}

\begin{rmk}
\label{rmk:labelpearl}
We observe that, because of the condition that Morse flowlines do not cross the hypersurfaces $D_j^{\R}$ (by Corollary \ref{cor:adm}), the labels at opposite ends of an internal flipping Morse flowline are identical. 
Furthermore, at a flipping marked point connected by an incoming edge to the flipping generator $p_K$, the label immediately before is $K$ and the label immediately after is $\bar{K}$.
Also, by Corollary \ref{cor:adm}, the boundary lifts can only cross the hypersurfaces positively.
So as we follow the boundary around anti-clockwise between two adjacent flipping marked points, the label at the beginning of the segment contains (not necessarily strictly) the label at the end of the segment.
Suppose the pearl corresponding to vertex $v$ of the underlying tree has degree $d_v \in H_2(\CP{n},\RP{n}) \cong \Z$. 
Then it must intersect $D_j$ $d_v$ times, and none of the intersections can be internal by admissibility, so the boundary lift must intersect $D_j^{\R}$ $d_v$ times.
It follows that
\[ \sum_{j \,\mathrm{ mod }\, k_v} e_{\bar{K}_{j-1}} - e_{K_{j}} = d_v e_{[n+2]}\]
in $\widetilde{M}$, where $K_1,\ldots,K_{k_v}$ are the labels given to the points immediately before the flipping points (traversing the boundary of the pearl in positive direction) on the pearl corresponding to $v$.
It follows quickly that
\[ \sum_{j=1}^{k_v} e_{K_j} = \frac{k_v - d_v}{2}e_{[n+2]}\]
for each pearl.
Figure \ref{fig:n5} shows a possible labeling of a flipping holomorphic pearly tree.
\end{rmk}

\begin{figure}
\centering
\includegraphics[width=0.8\textwidth]{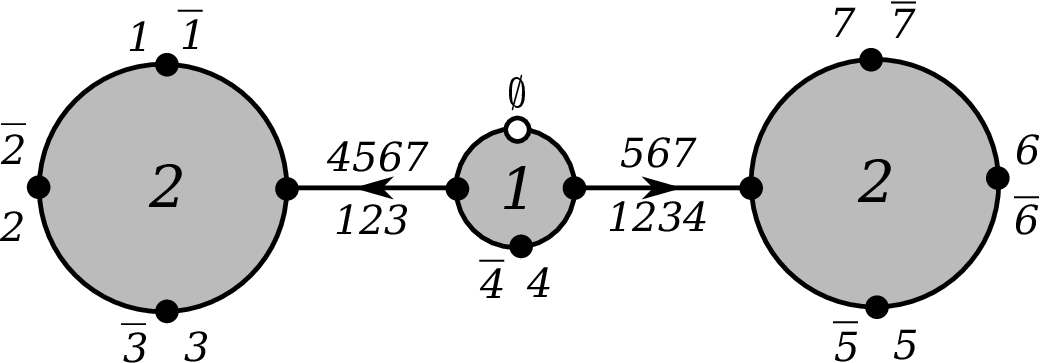}
\caption{An example of a legal labeling of a flipping holomorphic pearly tree, which might contribute to the coefficient of $p_{\phi}$ in the $A_{\infty}$ product $\mu^7(p_{\{1\}},\ldots,p_{\{7\}})$. 
We have illustrated a simple case, in which all external flowlines are constant because the points $p_{\{j\}}$ are maxima of the Morse function $f$.
The external label `$1$' means the set $\{1\}$, while `$\overline{1}$' means the complement $\{2,3,4,5,6,7\}$.
The big label `$1$' in the middle of a pearl means that the pearl has degree $1$.
\label{fig:n5}}
\end{figure}

\begin{rmk}
\label{rmk:autreg}
We will choose the almost-complex structure component of our perturbation data to be equal to the standard integrable complex structure $J_0$, and the Hamiltonian perturbation to be identically $0$.
Then the pearls in a flipping holomorphic pearly tree with labels $L'$ are holomorphic disks with boundary on $\RP{n}$, hence they can be `doubled' to a holomorphic sphere by the Schwarz reflection principle.
It follows from \cite[Proposition 7.4.3]{mcduffsalamon} that the moduli space of holomorphic spheres in $\CP{n}$, in a given homology class, is automatically regular. 
The moduli space of pearls is the real part of the moduli space of spheres, hence also regular. 
It follows that for every $(r, \bm{u}) \in \mathcal{M}_{\mathcal{S}^{fpt}}(\bm{y})$, the linearized operator
\[D^h_{\mathcal{S}^{fpt},r,\bm{u}}: T_{(r,\bm{u})} \left(\mathscr{B}_{\mathcal{S}^{fpt}} \right) \To (\mathscr{E}_{S_r})_{\bm{u}}\]
of Definition \ref{defn:extlin} is automatically surjective.
Thus, to show that a moduli space $\mathcal{M}_{\mathcal{S}^{fpt}}(\bm{y})$ of flipping holomorphic pearly trees is regular, we need only check that the evaluation map
\[ \bm{ev}: \mathrm{ker}(d_{\mathcal{S}^{fpt}}) \To T_{\bm{u}} \left((S^n)^{F(S)}\right)\]
is surjective at each $(r, \bm{u}) \in \mathcal{M}_{\mathcal{S}^{fpt}}(\bm{y})$. 
Note that $\mathrm{ker}(d_{\mathcal{S}^{fpt}})$ is the space of holomorphic pearls and Morse flowlines, without the constraint $\bm{ev}(\bm{u}) \in \Delta^S$.
\end{rmk}

\begin{defn}
\label{defn:fixloc}
The following notation will be useful.
If $K_1,\ldots,K_k$ are disjoint subsets of $[n+2]$, we define 
\[F_{K_1,K_2,\ldots,K_k} := \{ \bm{x} \in S^n: x_l = x_m \mbox{ for all }l,m \in K_i, \mbox{ for all }i\}.\]
\end{defn}

\begin{rmk}
Observe that 
\[F_{K,\bar{K}} = \{p_K,p_{\bar{K}}\}.\]
As we saw in Lemma \ref{lem:morsecells}, the unstable manifold $\mathcal{U}(K)$ of $p_K$ is an open subset of $F_{\bar{K}}$, and the stable manifold $\mathcal{S}(K)$ is an open subset of $F_K$.
\end{rmk}

\subsection{Computation of $\mathcal{A}'$}
\label{subsec:calc}

In this section we compute the $A_{\infty}$ structure of $\mathcal{A}'$.

First, we observe that the analogue of Corollary \ref{cor:m1m2top} holds for $\mathcal{A}'$. 
I.e., $\mu^1 = 0$ and the only possibly non-zero $\mu^2$ products are
\[ \mu_{\mathcal{A}'}^2(p_{K_1},p_{K_2}) = a'(K_1,K_2) p_{K_1 \sqcup K_2}\]
for disjoint $K_1,K_2$. 
The proof is exactly the same, using the corresponding properties of $\mathcal{A}'$ given in Lemma \ref{lem:a'prop}.

\begin{prop}
\label{prop:ak1k2}
We have
\[a'(K_1,K_2) = \pm 1.\]
\end{prop}
\begin{proof}
Let $K_3:= \overline{K_1 \sqcup K_2}$, so $K_1 \sqcup K_2 \sqcup K_3 = [n+2]$.
If any of $K_1,K_2,K_3$ are $\phi$ or $[n+2]$, the result is easy as the corresponding holomorphic disks are constant.
If that is not the case, then $a'(K_1,K_2)$ is given by a count of flipping holomorphic pearly trees.
The homology class of such a flipping holomorphic pearly tree is
\[ \left( 2\frac{ |K_1 \sqcup K_2|' - |K_1|' - |K_2|' }{n+2} + 2-1 \right) = 1 \]
by Proposition \ref{prop:pearlenergy}.
Therefore the corresponding flipping holomorphic pearly tree has two incoming and one outgoing legs, and a single pearl with the homology class of half of a line in $\CP{n}$ with boundary on $\RP{n}$.

The real part of such a pearl is a line. 
Thus, $a'(K_1,K_2)$ counts lines passing through the unstable manifolds $\mathcal{U}(K_1)$, $\mathcal{U}(K_2)$, $\mathcal{U}(K_3)$.
Recall from Lemma \ref{lem:morsecells} that the unstable manifolds $\mathcal{U}(K_i)$ are contained in the linear spaces $F_{\bar{K}_i}$ (see Definition \ref{defn:fixloc}).

Given points $p_1 \in F_{\bar{K}_1}$ and $p_2 \in F_{\bar{K}_2}$, the line through $p_1$ and $p_2$ is contained in the linear space $F_{\bar{K}_1 \cap \bar{K}_2} = F_{K_3}$.
This space intersects $F_{\bar{K}_3}$ transversely at $p_{K_3}$.
Therefore there is a unique line (namely $F_{K_1,K_2,K_3}$) that intersects $\mathcal{U}(K_1)$, $\mathcal{U}(K_2)$, $\mathcal{U}(K_3)$ (at $p_{K_1}$, $p_{K_2}$, $p_{K_3}$ respectively), and the intersections are transverse so the flipping holomorphic pearly tree is regular.

We check that it is admissible, using Proposition \ref{prop:calcint}.
Firstly, the Morse flowlines are constant at the $p_{K_i}$, hence do not cross the hypersurfaces $D_j^{\R}$.
Secondly, the boundary lifts as
\[ p_{K_1} \leadsto p_{K_2 \sqcup K_3} \To p_{K_2} \leadsto p_{K_1 \sqcup K_3} \To p_{K_3} \leadsto p_{K_1 \sqcup K_2} \To p_{K_1}\]
where $\To$ denotes a straight line connecting two points and $\leadsto$ denotes changing sheet.
This lift clearly crosses all hypersurfaces $D_j^{\R}$ positively (since the label at the beginning of a straight line always contains the label at the end), so the flipping holomorphic pearly tree is admissible and regular.

Thus $a'(K_1,K_2) = \pm 1$ as required.
\end{proof}

We are now in a position to prove Theorem \ref{thm:zcoeffs}. 
It is implied by the following:

\begin{thm}
\label{thm:zcoeffs2}
The cohomology algebra of $\mathcal{A}$ is
\[H^*(\mathcal{A}) \cong \wedgestar \widetilde{M}_{\C}\]
as $\Z_2$-graded associative $\C$-algebras.
The isomorphism is given by
\[ p_{K} \mapsto \sigma_K \underset{j \in K}{\wedge} e_j,\]
for some sign $\sigma_K = \pm 1$.
\end{thm}
\begin{proof}

We define a homomorphism of $\C$-algebras from the tensor algebra of $\widetilde{M}_{\C}$ to the cohomology algebra of $\mathcal{A}$, by
\begin{eqnarray*} 
\bigoplus_{k=1}^{\infty} (\widetilde{M}_{\C})^{\otimes k} & \To & H^*(\mathcal{A}), \\
e_j & \mapsto & p_{\{j\}} \mbox{ for all $j \in [n+2]$.}
\end{eqnarray*}
By Corollary \ref{cor:cohsigns}, this descends to a homomorphism
\[ \wedgestar \widetilde{M}_{\C} \To H^*(\mathcal{A}).\]
It follows from Proposition \ref{prop:ak1k2} that the elements $p_{\{j\}}$ generate the algebra $H^*(\mathcal{A}')$, and hence the corresponding elements generate $H^*(\mathcal{A})$, by Proposition \ref{prop:ll'}.
Therefore this homomorphism is surjective, so because both sides have the same rank it must be an isomorphism.
\end{proof}

Now we consider the next non-trivial $A_{\infty}$ product in $\mathcal{A}'$, $\mu^{n+2}$.
We aim to compute 
\[\mu^{n+2}(p_{\{\sigma(1)\}},\ldots,p_{\{\sigma(n+2)\}}),\]
where $\sigma$ is a permutation of $[n+2]$ (these are the important products to compute in order to apply deformation theory, because they determine the deformation class of the $A_{\infty}$ structure  (see Section \ref{subsec:def}).

\begin{prop}
\label{prop:mn2}
In $\mathcal{A}'$, we have
\[\mu^{n+2}(p_{\{\sigma(1)\}},\ldots,p_{\{\sigma(n+2)\}}) = \pm p_{\phi},\]
for exactly one permutation $\sigma$ of $[n+2]$.
For all other permutations, the result is $0$.
A different choice of the point $p_{\phi}$ (the minimum of the Morse function $h$) will lead to a different permutation $\sigma$.
\end{prop}
\begin{proof}
First, note that $p_{\phi}$ is the only term that can appear in this product, for grading reasons (Corollary \ref{cor:grading}). 

Note also that $\mathcal{U}(p_{\{j\}}) = \{p_{\{j\}}\}$ and $\mathcal{S}(p_{\phi}) = \{p_{\phi}\}$,
so the external gradient flowlines of the flipping holomorphic pearly trees contributing to the coefficient of $p_{\phi}$ in this product are constant.
We split the proof into two parts: counting the flipping holomorphic pearly trees with a single `pearl' (we show that these give the desired answer) and proving that there are no `multiple-pearl trees' contributing to the product.

For the first part, Proposition \ref{prop:pearlenergy} shows that a disk contributing to this product must have degree $n$. 
By pairing such a disk with its conjugate we obtain a degree-$n$ curve through the $n+3$ points $p_{\{1\}},\ldots,p_{\{n+2\}}, p_{\phi}$.
It is a classical theorem of Veronese that there is a unique rational normal curve through $n+3$ generic points in $\CP{n}$.
A constructive proof is given in \cite[p. 10]{harrisbook}.
We just need to check that this curve satisfies the conditions required for the definition of an admissible flipping holomorphic pearly tree -- namely, the curve should be real, and its real part should admit a lift to $S^n$ which changes sheet at each point $p_{\{j\}}$ and crosses the hypersurfaces $D_k^{\R}$ positively.

By the construction in \cite{harrisbook}, we can parametrize our curve as $u: \CP{1} \To \CP{n}$,
\begin{eqnarray*}
u(z) & := & \left( \begin{array}{cccc}
				n+1 & -1 &\ldots & -1 \\
				-1 & n+1 & \ldots & -1 \\
				\vdots & \vdots & \ddots & \vdots \\
				-1 & -1 & \ldots & n+1
				\end{array} \right) 
				\left( \begin{array}{c}
						(z-\nu_1)^{-1} \\
						(z-\nu_2)^{-1} \\
						\vdots \\
						(z-\nu_{n+2})^{-1}
						\end{array} \right) \\
		& =& \left[ \frac{n+1}{z-\nu_1} - \sum_{j \neq 1} \frac{1}{z-\nu_j} : \frac{n+1}{z-\nu_2} - \sum_{j \neq 2} \frac{1}{z-\nu_j} : \ldots :   \frac{n+1}{z-\nu_{n+2}} - \sum_{j \neq n+2} \frac{1}{z-\nu_j} \right]. 
\end{eqnarray*}
Observe that this curve has degree $n$: if we clear denominators, the leading coefficients $z^{n+1}$ in all factors cancel, leaving polynomials of degree $n$.
Furthermore, we have
\[u(\nu_j) = [-1:-1:\ldots:n+1:\ldots:-1] = p_{\{j\}}.\]
We choose the $\nu_j$ so that $u(0) = p_{\phi}$, i.e.,
\[\left( \begin{array}{cccc}
				n+1 & -1 &\ldots & -1 \\
				-1 & n+1 & \ldots & -1 \\
				\vdots & \vdots & \ddots & \vdots \\
				-1 & -1 & \ldots & n+1
				\end{array} \right) 
				\left( \begin{array}{c}
						\nu_1^{-1} \\
						\nu_2^{-1} \\
						\vdots \\
						\nu_{n+2}^{-1}
						\end{array} \right) = p_{\phi}.\]

Note that this parametrization automatically gives a lift of the boundary $\RP{1}$ to $\R^{n+1}\setminus\{0\}$ and hence to $S^n$.
Furthermore, the parametrization changes sheets exactly at the flipping points $\nu_j$, because the sign of the dominant term $(z-\nu_j)^{-1}$ changes there.
We just have to check that it crosses all of the real hypersurfaces $D_k^{\R}$ positively.
This is true because if
\[ \frac{n+1}{z-\nu_k} - \sum_{j \neq k} \frac{1}{z-\nu_j} = 0,\]
then the derivative
\[ -\frac{n+1}{(z-\nu_k)^2} + \sum_{j \neq k} \frac{1}{(z-\nu_j)^2} > 0\]
by the quadratic-arithmetic mean inequality (alternatively one can graph the function).

Thus, the two halves of this curve are the only disks that can contribute to such a product, and only one passes through $p_{\phi}$ (the other has the opposite lift of the boundary, hence passes through the antipode of $p_{\phi}$).
The permutation $\sigma$ is determined by the ordering of the coordinates of the chosen point $p_{\phi}$.
 
It is clear from our construction that this pearl is regular. 
Namely, because we have exhibited a construction of a degree-$n$ curve through $n+3$ arbitrary generic points in $\RP{n}$, if we fix all boundary points $p_{\{j\}}, p_{\phi}$ except for one, then the evaluation map at the remaining point is transverse to the point.

Now we proceed with the second part of the proof, namely showing that multiple-pearl trees do not contribute.
Suppose we have a contribution from a multiple-pearl tree.
The tree must contain a pearl with exactly one internal edge attached. 
Without loss of generality it has input flipping generators $p_{\{1\}},\ldots,p_{\{k+1\}}$ and a single Morse flowline attached at point $q$, as shown in Figure \ref{fig:deg} (it may also have the `output' point $p_{\phi}$ on its boundary, but whether it does or not is irrelevant to the following argument).
If $q$ is non-flipping then it follows from Remark \ref{rmk:labelpearl} that $k=n+1$, so this is not a multiple-pearl tree.
If $q$ is flipping, then it follows by Remark \ref{rmk:labelpearl} that it has degree $k$, where we assume $k<n$.

\begin{figure}
\centering
\includegraphics[width=0.6\textwidth]{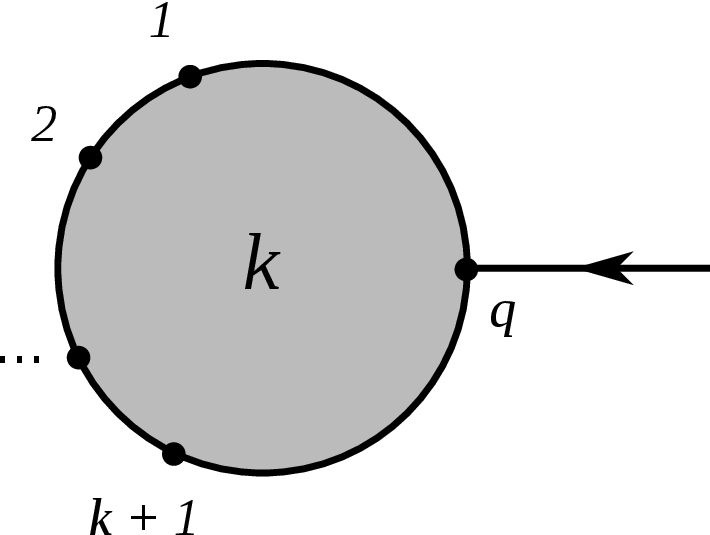}
\caption{Part of a multiple-pearl tree that may contribute to $\mu^{n+2}$.
The label `$j$' on a marked point means that marked point gets mapped to $p_{\{j\}}$, while the big label `$k$' in the middle of the pearl means that pearl has degree $k$.
\label{fig:deg}
}
\end{figure}

Any degree-$k$ curve in $\CP{n}$ is contained in a linear subspace of dimension $k$ (this can be proved by induction on $n$: choose any $k+1$ points on the curve and a hyperplane through those points, then the hyperplane intersects the degree-$k$ curve in more than $k$ points so the curve is contained in the hyperplane by Bezout's Theorem).
In our case, there is a unique dimension-$k$ linear subspace through the points $p_{\{1\}},\ldots,p_{\{k+1\}}$, namely $F_{\overline{[k+1]}}$ (to clarify: $\overline{[k+1]} = \{k+2,\ldots,n+2\}$).

Therefore our pearl is a degree-$k$ curve in a $k$-dimensional projective space, so by the first half of the argument, the evaluation map at $q$ runs over an open subset of $F_{\overline{[k+1]}}$.
But this subspace is preserved by the Morse flow of $f$, by the equivariance of $f$ with respect to the $S_{n+2}$ action.
Hence the Morse flow at $q$ is parallel to the evaluation map, so the evaluation map at $p_{\phi}$ has dimension (at least) $1$ less than expected.
Thus, for a generic choice of $p_{\phi}$, the moduli space will be empty.

Thus the only contributions to the product come from the single-pearl tree, which gives the advertised result.
\end{proof}

\begin{rmk}
We observe that the final argument, in which we showed that multiple-pearl flipping holomorphic pearly trees do not contribute to the product, remains true even if we make a small change in our perturbation data: 
observe that, by Remark \ref{rmk:labelpearl}, $q$ lies in the region $S^n_{[k+1]}$.
If we perturb the holomorphic curve equation by a small amount, the perturbed evaluation map at $q$ can be made arbitrarily $C^0$-close to the unperturbed one. 
Thus, the image of the perturbed evaluation map at $q$ is contained in an arbitrarily small open neighbourhood of $F_{\overline{[k+1]}} \cap S^n_{[k+1]}$.

Now the Morse flowline emanating from $q$ remains inside the region $S^n_{[k+1]}$, since flipping flowlines cannot cross the hypersurfaces by Corollary \ref{cor:adm}.
But $F_{\overline{[k+1]}}\cap S^n_{[k+1]}$ is exactly the intersection of the unstable manifold of $p_{[k+1]}$ with $S^n_{[k+1]}$, so the flowline remains inside an arbitrarily small open neighbourhood of $F_{\overline{[k+1]}} \cap S^n_{[k+1]}$.
Given that, for generic $p_{\phi}$, the evaluation map at the other end of the Morse flowline misses $F_{\overline{[k+1]}}$, it also misses a sufficiently small neighbourhood of it.
Therefore, for a sufficiently small perturbation, the moduli space remains empty.
\end{rmk}

\subsection{Versality of $\mathcal{A}'$}
\label{subsec:def}

We aim to prove Theorem \ref{thm:mirrsym} by applying the techniques of \cite[Section 3]{seidel03}, in the equivariant setting.
All our conventions on signs and gradings are taken from that paper.
We review some necessary definitions and results.

\begin{defn}
\label{defn:a}
Consider the $\Q$-graded algebra 
\[A := \wedgestar \left( \widetilde{M}_{\C} \right),\]
where the grading is given by $n/(n+2)$ times the normal ($\Z$-)grading.
Define an action of the character group of $M$, 
\[\mathbb{T} := \mathrm{Hom}(M,\C^*),\]
on $A$ by 
\[ \alpha \cdot e := \alpha(e)e.\]
Let $\mathfrak{A}(A)$ denote the set of $\Q$-graded, $\mathbb{T}$-equivariant $A_{\infty}$-algebras with underlying graded vector space $A$, $\mu^1 = 0$ and
\[ \mu^2(a_2,a_1) = (-1)^{|a_1|}a_2 \wedge a_1.\]
\end{defn}

\begin{prop}
\label{prop:hha}
Recall that the ($\mathbb{T}$-equivariant) Hochschild cohomology of $A$ is given by the Hochschild-Kostant-Rosenberg isomorphism \cite{hkr}:
\[ HH^{s+t}(A,A)^{t,\mathbb{T}} \cong \bigoplus _{\frac{2}{n+2}s + \frac{n}{n+2}j = s+t} \left(Sym^s(\widetilde{M}_{\C} ^{\vee}) \otimes \Lambda ^j \left(\widetilde{M}_{\C}  \right)\right)^{\mathbb{T}}.\]
For $d>2$, we have
\[HH^2(A,A)^{2-d,\mathbb{T}} = \left\{ \begin{array}{ll}
                                          \C \cdot W & \mbox{ for $d = n+2$} \\
                                          0 & \mbox{ otherwise,}
                                          \end{array} \right.\]
where $W = z_1\ldots z_{n+2} = z^{e_{[n+2]}}$ is the superpotential of the mirror, viewed as an element of the symmetric tensor product $Sym^{n+2}(\widetilde{M}_{\C}^{\vee})$.
\end{prop}
\begin{proof}
Suppose we have a generator
\[ z^a \underset{k \in K}{\wedge} e_k \in HH^2(A,A)^{2-d,\mathbb{T}}.\]
Here $a \in \widetilde{M}_{\ge 0}^{\vee}$, $K \subset [n+2]$ and $d = \mathrm{deg}(z^a) > 2$.
$\mathbb{T}$-equivariance simply says that
\[ a = e_K + qe_{[n+2]}\]
for some $q \in \Z_{\ge 0}$ (here we identify $\widetilde{M}^{\vee}$ with $\widetilde{M}$ in the natural way).
To lie in $HH^2$ we must have
\begin{eqnarray*}
2&=&\frac{2}{n+2} \mathrm{deg}(z^a) + \frac{n}{n+2}|K| \\
&=& \frac{2}{n+2}(|K| + q(n+2)) + \frac{n}{n+2}|K| \\
&=& |K| + 2q.
\end{eqnarray*}
Now we have
\[2< d = \mathrm{deg}(z^a) = |K| + (n+2)q = 2 + nq,\]
hence $q>0$.
Therefore, we must have $K = \phi,q=1$ and $a = e_{[n+2]}$.
Thus the generator is $z^a = W$.
\end{proof}

\begin{prop}
\label{prop:versal}
$\mathcal{A}'$ is a versal element of $\mathfrak{A}(A)$, in the sense of a $\mathbb{T}$-equivariant version of \cite[Lemma 3.2]{seidel03}, with deformation class $\pm W \in HH^2(A,A)^{-n,\mathbb{T}}$.
In particular, any element of $\mathfrak{A}(A)$ with the same deformation class is quasi-isomorphic to $\mathcal{A}'$.
\end{prop}
\begin{proof}
The fact that $\mathcal{A}'$ lies in $\mathfrak{A}(A)$ follows from our previous results, namely Lemma \ref{lem:a'prop}:
\begin{itemize} 
\item $\mu^1 = 0$ as the only non-zero $A_{\infty}$ products are $\mu^{2+nq}$ for $q \in \Z_{\ge 0}$;
\item the underlying algebra is $A$ (Theorem \ref{thm:zcoeffs2});
\item the grading on $A$ is $n/(n+2)$ times the usual grading;
\item it is equivariant with respect to the action of $\mathbb{T}$.
\end{itemize}

The fact that $\mathcal{A}'$ is versal follows from the results:
\begin{itemize}
\item $\mu^k = 0$ for $2<k<n+2$ (by the analogue of Corollary \ref{cor:prodzero});
\item The first non-trivial higher product $\mu^{n+2}$ satisfies
\[\mu^{n+2}(e_1,\ldots,e_{n+2})= \pm 1\]
(without loss of generality) but is $0$ on all other permutations of the generators $e_i$ (Proposition \ref{prop:mn2}).
Therefore the deformation class of $\mathcal{A}'$ in $HH^2(A,A)^{-n}$ is given (by the HKR isomorphism) by
\[\mu^{n+2}(\bm{z},\ldots,\bm{z}) = \pm z_1\ldots z_{n+2} = \pm W(z),\]
where $\bm{z} = \sum_j z_j e_j$.
Combining this with Proposition \ref{prop:hha} gives the result.
\end{itemize}
\end{proof}

\section{Matrix factorizations}
\label{sec:matfact}

We now consider the other side of mirror symmetry. 
Recall (from the Introduction) that the putative mirror to $\mathcal{P}^n$ is the Landau-Ginzburg model $(\mathrm{Spec} (R),W)$, 
where
\begin{eqnarray*}
R &:= & \C \left[ \widetilde{M} \right] \\
 W &= &z^{e_{[n+2]}}.
 \end{eqnarray*}
Observe that there is a natural action of $\mathbb{T}$ on $R$ that preserves $W$ (recall $\mathbb{T}  := \mathrm{Hom}(M,\C^*)$).

Also recall (from the Introduction) that the $B$-model on $(\mathrm{Spec} (R), W)$ is given by the triangulated category of singularities of $W^{-1}(0)$, which is quasi-equivalent (by \cite[Theorem 3.9]{orlov04}) to the category $MF(R,W)$ of matrix factorizations of $W$.
The object corresponding to our Lagrangian $L^n$ is the skyscraper sheaf at the origin,
\[ \mathcal{O}_0 \in D^b_{\mathrm{Sing}}(W^{-1}(0)).\]
Henceforth, we work entirely in the category $MF(R,W)$.
We abuse notation, and denote also by $\mathcal{O}_0$ the matrix factorization corresponding to $\mathcal{O}_0$ under the above quasi-equivalence.

To prove Theorem \ref{thm:mirrsym}, we must show that the differential $\Z_2$-graded algebra of endomorphisms of $\mathcal{O}_0$,
\[\mathcal{B} := \mathrm{Hom}^*_{MF(R,W)}(\mathcal{O}_0,\mathcal{O}_0),\]
is quasi-isomorphic to $\mathcal{A}$.

It is explained in \cite{dyckerhoff09} how to compute a minimal $A_{\infty}$ model for the endomorphism algebra of $\mathcal{O}_0$.
That paper focuses on the case where $W$ has an isolated singularity at $0$, which is certainly not true in our case, but the computation of the minimal $A_{\infty}$ model does not rely on this assumption.
We briefly review the construction, explaining how the $\mathbb{T}$-action enters the picture.

The matrix factorisation corresponding to $\mathcal{O}_0$ is the Koszul resolution of $\mathcal{O}_0$
\[ R \otimes \wedgestar \widetilde{M}\]
with the deformed differential
\[ \delta := \iota_{u} + v \wedge \cdot \]
where
\begin{eqnarray*}
 u &=& \sum_{j=1}^{n+2} z_j \theta_j^{\vee} \in R \otimes \wedgestar \widetilde{M}^{\vee} \\
v &=& \sum_{j=1}^{n+2} a_j \frac{W}{z_j} \theta_j \in R \otimes \wedgestar \widetilde{M}
\end{eqnarray*}
where $\{\theta_j\}$ is a relabeling of the canonical basis for $\widetilde{M}$, $\{ \theta_j^{\vee} \}$ is the dual basis of $\widetilde{M}^{\vee}$, and $a_j$ are numbers adding up to $1$. 
Alternatively, we can write this matrix factorisation as
\[ \left( R \left\langle \theta_1,\ldots,\theta_{n+2} \right\rangle, \delta \right),\]
where
\[ \delta = \sum_j z_j \del{}{\theta_j} + a_j \frac{W}{z_j} \theta_j.\]

The endomorphism algebra of $\mathcal{O}_0$ is the algebra
\[ R \otimes \wedgestar \widetilde{M}^{\vee} \otimes \wedgestar \widetilde{M}.\]
This can be thought of as the commutative algebra of differential operators
\[ \mathcal{B} := R \left\langle \theta_1,\ldots,\theta_{n+2},\del{}{\theta_1},\ldots, \del{}{\theta_{n+2}} \right\rangle\]
with the differential given by $d = [\delta, -]$.
One can check that
\begin{eqnarray*}
d(\theta_j) &=& z_j \\
d \left( \del{}{\theta_j} \right) &=& a_j \frac{W}{z_j}.
\end{eqnarray*}
Thus the cohomology algebra $H^*(\mathcal{B},d)$ is generated by the elements
\[ \bar{\partial}_j : = \del{}{\theta_j} - a_j \frac{W}{z_j z_k} \theta_k\]
for some $k \neq j$ (this is proven in \cite{dyckerhoff09} by constructing an explicit homotopy contracting $\mathcal{B}$ onto the subcomplex generated by the $\bar{\partial}_j$).
The generators $\bar{\partial}_j$ supercommute, so the cohomology algebra can be naturally identified with
\[ A = \wedgestar(\widetilde{M}_{\C})\]
via
\[ \bar{\partial}_j \mapsto e_j.\]
This proves that
\[H^*\left( \mathrm{Hom}^*_{MF(R,W)}(\mathcal{O}_0,\mathcal{O}_0)\right) \cong \wedgestar \C^{n+2}\]
as $\Z_2$-graded associative $\C$-algebras.

We observe that the action of $\mathbb{T}$ extends in the natural way to $\mathcal{B}$, and that $\delta$ is invariant under the action of $\mathbb{T}$, so the differential algebra structure of $\mathcal{B}$ is $\mathbb{T}$-equivariant.

Furthermore, observe that if we assign $\Q$-gradings
\[ |z_j| = \frac{2}{n+2}, \,\,\, |\theta_j| = -\frac{n}{n+2}, \,\,\, \left| \del{}{\theta^j} \right| = \frac{n}{n+2},\]
then the product structure on $\mathcal{B}$ respects the grading (because $|\theta_j| + |\partial / \partial \theta_j| = 0$), and the differential on $\mathcal{B}$ has degree $|\delta| = +1$.
Therefore $(\mathcal{B},d)$ is a $\mathbb{T}$-equivariant differential ($\Q$-)graded algebra.
Observe that the grading on the cohomology algebra $A$ is $n/(n+2)$ times the usual one, as
\[ | \bar{\partial}_j | = \frac{n}{n+2}.\]

In \cite[Section 4]{dyckerhoff09}, it is shown how to construct a homotopy contracting $\mathcal{B}$ onto its cohomology, and hence (via the homological perturbation lemma) a minimal $A_{\infty}$ model for $\mathcal{B}$.
The homotopy used is manifestly $\mathbb{T}$-equivariant in our setting (see \cite{dyckerhoff09} to check this), so the resulting minimal model is also $\mathbb{T}$-equivariant.
Furthermore, the homotopy has degree $0$ with respect to the grading introduced above, so the $\Q$-grading is preserved under the perturbation lemma construction (in the sense that the $A_{\infty}$ product $\mu^k$ has degree $2-k$ with respect to this grading).
Thus we obtain a $\mathbb{T}$-equivariant, $\Q$-graded minimal $A_{\infty}$ model for $\mathcal{B}$, which we shall denote by $\mathcal{B}'$.
It is clear from our discussion that $\mathcal{B}'$ satisfies the necessary conditions to lie in $\mathfrak{A}(A)$.

\begin{prop}
\label{prop:aprimeversal}
$\mathcal{B}'$ is a versal element of $\mathfrak{A}(A)$, in the same $\mathbb{T}$-equivariant sense as in Proposition \ref{prop:versal}.
It has the same deformation class as $\mathcal{A}'$.
\end{prop}
\begin{proof}
The fact that $\mathcal{B}'$ lies in $\mathfrak{A}(A)$ follows from the preceding discussion.
The fact that $\mathcal{B}'$ is versal with the same deformation class as $\mathcal{A}'$ follows from the results:
\begin{itemize}
\item $\mu^k = 0$ for $2<k<n+2$ because of the grading and $\mathbb{T}$-equivariance (exactly as in Corollary \ref{cor:prodzero});
\item The first non-trivial higher product $\mu^{n+2}$ satisfies
\[\mu^{n+2}(e_1,\ldots,e_{n+2})= \pm 1\]
for an appropriate choice of contracting homotopy $h$ (see \cite[Theorem 4.8]{dyckerhoff09})
 but is $0$ on all other permutations of the generators $e_j$ (by similar computations -- one can show that only one tree gives a non-zero contribution to such a product).
Therefore the deformation class of $\mathcal{B}'$ in $HH^2(A,A)^{-n}$ is given (by the HKR isomorphism) by
\[\mu^{n+2}(\bm{z},\ldots,\bm{z}) = \pm z_1\ldots z_{n+2} = \pm W(z),\]
where $\bm{z} = \sum_j z_j e_j$.
\end{itemize}
Combining this with Propositions \ref{prop:hha} and \ref{prop:versal} gives the result.
\end{proof}

\begin{cor}
There are quasi-isomorphisms
\[ \mathcal{A} \cong \mathcal{A}' \cong \mathcal{B}' \cong \mathcal{B}.\]
In particular, Theorem \ref{thm:mirrsym} is proved.
\end{cor}
\begin{proof}
That $\mathcal{A} \cong \mathcal{A}'$ follows from Proposition \ref{prop:ll'}.
That $\mathcal{A}' \cong \mathcal{B}'$ follows from Propositions \ref{prop:versal} and \ref{prop:aprimeversal}, by a $\mathbb{T}$-equivariant version of \cite[Lemma 3.2]{seidel03}.
That $\mathcal{B}' \cong \mathcal{B}$ follows by construction.
\end{proof}

\section{Applications}
\label{sec:app}

\subsection{Covers of $\mathcal{P}^n$}

We recall the behaviour of the Fukaya category with respect to covers, from \cite[Section 8b]{seidel03} and \cite[Section 9]{seidelg2}.

Suppose that 
\[\rho: M \To \Gamma\]
is a homomorphism onto a finite abelian group $\Gamma$.
Let $\mathcal{P}^n_{\Gamma} \To \mathcal{P}^n$ be the associated abelian cover, with covering group $\Gamma$ (recalling that $\pi_1(\mathcal{P}^n) \cong M$).
There is a natural action of $\Gamma^*$ on $\mathcal{A}$, inherited from the action of $\mathbb{T}$ on $\mathcal{A}$ and the embedding $\Gamma^{\ast} \hookrightarrow \mathbb{T}$ (here $\Gamma^{\ast} := \mathrm{Hom}(\Gamma,\C^*)$ is the character group of $\Gamma$). 

\begin{defn}
We define the object 
\[\tilde{L}^n \in Ob(D^{\pi}(\mathcal{F}uk(\mathcal{P}^n_{\Gamma})))\]
to be the direct sum of all lifts of $L^n$.
We define its $A_{\infty}$ endomorphism algebra
\[ \tilde{\mathcal{A}} := CF^*\left( \tilde{L}^n,\tilde{L}^n \right).\]
\end{defn}

\begin{prop}[See \cite{seidel03} or \cite{seidelg2}]
\label{prop:semid}
We have
\[ \tilde{\mathcal{A}} \cong \mathcal{A} \rtimes \Gamma^{\ast}.\]
\end{prop}

Now let us consider the mirror statement to Proposition \ref{prop:semid}.
Taking a cover with covering group $\Gamma$ corresponds, on the mirror, to considering $\Gamma^*$-equivariant objects (sheaves or matrix factorizations), where $\Gamma^*$ acts on $\C[\widetilde{M}]$ via the natural embedding $\Gamma^* \hookrightarrow \mathbb{T}$.

\begin{defn}
Define the object
\[ \tilde{\mathcal{O}}_0 := \mathcal{O}_0 \otimes \C[\Gamma^*] \in D^b_{\mathrm{Sing},\Gamma^*}(W^{-1}(0)).\]
We define its endomorphism algebra
\[ \tilde{\mathcal{B}} := \mathrm{Hom}^*_{D^b_{\mathrm{Sing}, \Gamma^*}}(\tilde{\mathcal{O}}_0,\tilde{\mathcal{O}}_0).\]
\end{defn}

Corresponding to Proposition \ref{prop:semid}, we have the result

\begin{prop}[See \cite{seidelg2}]
\label{prop:semid2}
We have
\[ \tilde{\mathcal{B}} \cong \mathcal{B} \rtimes \Gamma^{\ast}.\]
\end{prop}

\begin{cor}
\label{cor:mirrcov}
There is a quasi-isomorphism
\[ CF^*(\tilde{L}^n,\tilde{L}^n) \cong \mathrm{Hom}^*_{D^b_{\mathrm{Sing},\Gamma^*}}(\tilde{O}_0,\tilde{O}_0).\]
\end{cor}
\begin{proof}
Follows from Propositions \ref{prop:semid}, \ref{prop:semid2} and Theorem \ref{thm:mirrsym}.
\end{proof}

\subsection{Affine Fermat hypersurfaces}
\label{subsec:hmscy}

\begin{defn}
Let $X^n$ be the Calabi-Yau Fermat hypersurface 
\[ X^n := \{ z_1^{n+2}+ \ldots + z_{n+2}^{n+2} = 0\} \subset \CP{n+1} = \mathbb{P}\left( \widetilde{M}_{\C} \right) .\]
We define the divisor
\[ X^n_{\infty} := \bigcup_j \{z_j = 0\},\]
and the affine part,
\[ \widetilde{X}^n := X^n \cap M_{\C^*} = X^n \setminus X^n_{\infty}.\]
\end{defn}

There is a covering
\begin{eqnarray*}
\widetilde{X}^n & \overset{\pi}{\To} & \mathcal{P}^n \\
{[}z_1:\ldots:z_{n+2}{]} & \mapsto & {[}z_1^{n+2}: \ldots :z_{n+2}^{n+2}{]}
\end{eqnarray*}
with corresponding group homomorphism
\[\rho: M  \To  \Gamma_n := M \otimes \Z_{n+2}.\]

\begin{defn}
Define the map $\Gamma_n \To \Z_{n+2}$ by taking the sum of the entries (this is well-defined because $e_{[n+2]} \mapsto 0$).
Call its kernel $\tilde{\Gamma}_n$, so we have a short exact sequence
\[ 0 \To \tilde{\Gamma}_n \To \Gamma_n \To \Z_{n+2} \To 0.\]
\end{defn}

\begin{defn}
Let $Y^n$ be the singular Calabi-Yau hypersurface
\[ Y^n := \{W = 0\} \subset \CP{n+1},\]
where $W = z_1\ldots z_{n+2}$ as before.
There is an action of $\tilde{\Gamma}_n^*$ on $Y^n$, inherited from the action of $\Gamma_n^*$ (which comes from the embedding $\Gamma_n^* \hookrightarrow \mathbb{T}$), because the kernel of the map $\Gamma_n^* \To \tilde{\Gamma}_n^*$ acts trivially on projective space.
\end{defn}

\begin{thm}
\label{thm:hmscy}
There is a fully faithful $A_{\infty}$ embedding,
\[ \mathrm{Perf}_{\tilde{\Gamma}_n^*}(Y^n) \To D^{\pi}(\mathcal{F}uk(\widetilde{X}^n)),\]
where the left-hand side denotes the category of $\tilde{\Gamma}_n^*$-equivariant perfect complexes on $Y^n$.
\end{thm}
\begin{proof}
On the Fukaya category side, let $\tilde{L}^n \in Ob(D^{\pi}(\mathcal{F}uk(\widetilde{X}^n)))$ be the direct sum of all lifts of $L^n$ under the covering $\pi$ (there are $|\Gamma_n| = (n+2)^{n+1}$ of them).
By Proposition \ref{prop:semid}, we have
\[CF^*(\tilde{L}^n,\tilde{L}^n) \cong \mathcal{A}^n \rtimes \Gamma_n^*.\]

On the other side, we repeat the argument of \cite[Section 10d]{seidel03}. 
Namely, consider the Beilinson exceptional collection
\[F_k := \Omega^{n+1-k}(n+1-k)\]
for $k=0,\ldots,n+1$ on $\CP{n+1}$.
It was shown in \cite{beilinson78} that $F_k$ generate $D^bCoh(\CP{n+1})$, and that
\[ \mathrm{Hom}^*(F_j,F_k) \cong \Lambda^{k-j} \left(\widetilde{M}_{\C}\right)\]
concentrated in degree $0$.

Now let $\iota: Y^n \To \CP{n}$ denote the inclusion, $E_k: = \iota^* F_k$, 
\[ E := \bigoplus_{k=0}^{n+1} E_k,\]
and
\[ B:= \mathrm{Hom}^*_{Y^n}(E,E).\]
We observe that $E$ generates $\mathrm{Perf}(Y^n)$, by \cite[Lemma 5.4]{seidel03}.

\begin{lem}
\label{lem:serred}
With appropriate grading shifts, there is an isomorphism of $\Q$-graded algebras,
\[B \cong A \rtimes \Z_{n+2},\]
where $A$ is the $\Q$-graded exterior algebra of Definition \ref{defn:a}.
\end{lem}
\begin{proof}
We compute that (writing $P$ for $\CP{n}$ and $Y$ for $Y^n$)
\begin{eqnarray*}
\mathrm{Hom}_{Y}^*(E_j,E_k) & = & \mathrm{Hom}_{Y}^*(\iota^* F_j, \iota^*F_k) \\
& \cong & \mathrm{Hom}_P^*(F_j, \iota_* \iota^* F_k) \mbox{ (adjunction)}\\
& \cong & \mathrm{Hom}_P^*(F_j, \mathcal{O}_Y \otimes_{\mathcal{O}_P} F_k) \\
& \cong &  \mathrm{Hom}_P^*(F_j, \left\{ \mathcal{K}_P \overset{W}{\To} \mathcal{O}_P \right\} \otimes_{\mathcal{O}_P} F_k) \mbox{ (resolving $\mathcal{O}_Y$ as $\mathcal{O}_P$-module)} \\
& \cong &  \left\{\mathrm{Hom}_P^*(F_j, \mathcal{K}_P \otimes F_k ) \To \mathrm{Hom}_P^*(F_j, F_k)\right\}\\
&\cong& \mathrm{Hom}_P^*(F_k,F_j)^{\vee}[-n] \oplus \mathrm{Hom}_P^*(F_j,F_k) \mbox{ (Serre duality)}\\
& \cong & \Lambda^{j-k} \left( \widetilde{M}_{\C}^{\vee} \right)[-n] \oplus \Lambda^{k-j}\left( \widetilde{M}_{\C} \right).
\end{eqnarray*}
We can naturally identify
\[
 \Lambda^j\left( \widetilde{M}_{\C}^{\vee} \right) \xrightarrow{\iota_{e_1 \wedge \ldots \wedge e_{n+2}}} \Lambda^{n+2-j} \left( \widetilde{M}_{\C} \right),\]
and hence
\[ \mathrm{Hom}_Y^*(E_j,E_k) \cong \Lambda^{k-j} \left( \widetilde{M}_{\C} \right),\]
where $k-j$ is taken modulo $n+2$ (when $k=j$ we have both $\Lambda^0 \oplus \Lambda^{n+2}$).
One can check that the composition rule is the obvious one.
We have thus computed that
\[B \cong \Lambda \left( \widetilde{M}_{\C} \right) \rtimes \Z_{n+2}\]
as an algebra.

However, this is not an isomorphism of graded algebras: for example, the morphisms in $B$ are concentrated in degrees $0$ and $n$. 
To fix this, we shift $E_k$ by the rational number $nk/(n+2)$ (compare \cite{seidel03,seidelg2,caldararutu10}), so that 
\[ \mathrm{Hom}\left(E_j \left[\frac{nj}{n+2}\right],E_k\left[\frac{nk}{n+2}\right]\right) \cong \Lambda^{k-j} \left(\widetilde{M}_{\C}\right)\]
is concentrated in degree $n(k-j)/n+2$, where $k-j$ is taken modulo $n+2$.
If we correspondingly multiply the standard grading on $\Lambda ( \widetilde{M}_{\C})$ by $n/(n+2)$, then the isomorphism above becomes an isomorphism of graded algebras.
\end{proof}

Now we obtain an $A_{\infty}$ structure on $B$, by applying the homological perturbation lemma \cite{guglamstash,kontsoib01} to the \v{C}ech complex whose cohomology computes $B$.
We denote the resulting $A_{\infty}$ algebra by $\mathcal{B}$.

\begin{lem}
\label{lem:semidquasi}
The algebra isomorphism of Lemma \ref{lem:serred} lifts to a quasi-isomorphism of $A_{\infty}$ algebras
\[\mathcal{B} \cong \mathcal{A}^n \rtimes \Z_{n+2}.\]
\end{lem}
\begin{proof}
The proof is similar to the arguments of Section \ref{sec:matfact}, in that we apply a $\mathbb{T}$-equivariant version of \cite[Lemma 3.2]{seidel03}.
First, we observe that there is a natural action of $\mathbb{T}$ on $\CP{n}$ (it is the algebraic torus action on the toric variety $\CP{n}$). 
Furthermore, $Y^n = W^{-1}(0)$ is a $\mathbb{T}$-equivariant divisor (it corresponds to the boundary of the moment polytope), and the sheaves $F_k$ are $\mathbb{T}$-equivariant.
It follows that the $\mathbb{T}$ action descends to $Y^n$ and the sheaves $E_k$, and therefore that the $A_{\infty}$ structure of $\mathcal{B}$ is $\mathbb{T}$-equivariant.

As in Corollary \ref{cor:prodzero}, by considering the grading, we can show that the only non-zero $A_{\infty}$ products $\mu^l$ in $\mathcal{B}$ occur for $l = 2 + nq$.
We use \cite[Proposition 4.2]{seidel03} and essentially identical arguments to those proving Proposition \ref{prop:hha} to compute the $\mathbb{T}$-equivariant Hochschild cohomology
\begin{eqnarray*}
HH^{2}(A \rtimes \Z_{n+2},A \rtimes \Z_{n+2})^{2-d,\mathbb{T}} &\cong &\bigoplus _{\frac{2}{n+2}d + \frac{n}{n+2}j = 2} \left(Sym^d(\widetilde{M}_{\C} ^{\vee}) \otimes \Lambda ^j \left(\widetilde{M}_{\C}  \right)\right)^{\mathbb{T}}\\
&\cong &  \left\{ \begin{array}{ll}
                                          \C \cdot W & \mbox{ for $d = n+2$} \\
                                          0 & \mbox{ otherwise}
                                          \end{array} \right.
\end{eqnarray*}
(observe that the right hand side would usually be restricted to the $\Z_{n+2}$-equivariant part, but because we are already considering only $\mathbb{T}$-equivariant structures, the $\Z_{n+2}$-equivariance is subsumed in the $\mathbb{T}$-equivariance).

It is clear from the fact that $\mathcal{A}^n$ is versal (Proposition \ref{prop:versal}) that $\mathcal{A}^n \rtimes \Z_{n+2}$ is versal (i.e., has non-zero deformation class in the above Hochschild cohomology group).
The proof of the fact that $\mathcal{B}$ is versal carries over exactly as in \cite[Lemma 10.8]{seidel03}. 
This completes the proof, by the $\mathbb{T}$-equivariant version of \cite[Lemma 3.2]{seidel03}.
\end{proof}

Now we apply the analogue of Proposition \ref{prop:semid2} to the $\tilde{\Gamma}_n^*$-equivariant sheaf
\[ \tilde{E} := E \otimes \C[\tilde{\Gamma}_n^*].\]
Combined with Lemma \ref{lem:semidquasi}, it shows that we have $A_{\infty}$ quasi-isomorphisms
\begin{eqnarray*}
\mathrm{Hom}^*_{Y^n,\tilde{\Gamma}_n^*}(\tilde{E},\tilde{E}) &\cong& \mathcal{B} \rtimes \tilde{\Gamma}_n^* \\
&\cong & \left(\mathcal{A}^n \rtimes \Z_{n+2}\right) \rtimes \tilde{\Gamma}_n^*\\
&\cong& \mathcal{A}^n \rtimes \Gamma_n^* \\
& \cong& CF^*(\tilde{L}^n,\tilde{L}^n).
\end{eqnarray*}

Therefore, if we map $\tilde{E} \mapsto \tilde{L}^n$, we define a quasi-isomorphism of the $A_{\infty}$ subcategories generated by these respective objects.
Since $\tilde{E}$ generates the category $\mathrm{Perf}_{\tilde{\Gamma}_n^*}(Y^n)$ (by the equivariant version of \cite[Lemma 5.4]{seidel03}), this extends to the desired $A_{\infty}$ embedding.
\end{proof}

\begin{rmk}
If we could prove that $\tilde{L}^n$ split-generates $\mathcal{F}uk(\widetilde{X}^n)$, we would have shown that this is a quasi-equivalence of $A_{\infty}$ categories.
So far we have been unable to prove this.
\end{rmk}

\begin{rmk}
By Corollary \ref{cor:mirrcov}, there is a quasi-isomorphism between the subcategory of
\[D^b_{\mathrm{Sing},\Gamma_n^*}(W^{-1}(0))\]
 generated by $\mathcal{O}_0 \otimes \C[\Gamma_n^*]$ and the subcategory of $D^{\pi}(\mathcal{F}uk(\widetilde{X}^n))$ generated by $\tilde{L}^n$.
It would be nice to prove Theorem \ref{thm:hmscy} by using a version of \cite[Theorem 3.11]{orlov05} to compare $B$-branes on $W^{-1}(0)$ with perfect complexes on $Y^n$, but such a result does not exist in the literature (the theorem does not apply in our case because $Y^n$ is singular).
\end{rmk}

\subsection{Projective Fermat hypersurfaces} 

This paper was conceived as a step towards a proof of Homological Mirror Symmetry for hypersurfaces in projective space (not necessarily Calabi-Yau). 
The author has made considerable progress in this direction, which will appear in a forthcoming preprint \cite{Sheridan2011}. 
In this section we give a brief outline of our anticipated results in the Calabi-Yau case, borrowing the terminology of \cite{seidel03}.

\begin{rmk}
Independently, Nohara and Ueda \cite{noharaueda} have considered the important special case of the quintic threefold, using the results of this paper together with techniques of \cite{seidel03}. 
\end{rmk}

The relative Fukaya category $\mathcal{F}uk(X^n, X^n_{\infty})$ is a one-parameter deformation of the affine Fukaya category $\mathcal{F}uk(\widetilde{X}^n)$.
Thus the $A_{\infty}$ endomorphism algebra of the object $\tilde{L}^n$ in the relative Fukaya category is a one-parameter deformation of $\mathcal{A}^n \rtimes \Gamma_n^*$, which we denote by $\tilde{\mathcal{A}}^n_q$ ($q$ is the parameter of the deformation). 

On the mirror side, we have the corresponding one-parameter deformation of the singular hypersurface $Y^n := W^{-1}(0)$, given by $Y^n_q := W_q^{-1}(0)$, where
\[ W_q = z_1 \ldots z_{n+2} + q\left(z_1^{n+2}+\ldots+z_{n+2}^{n+2}\right).\]
Observe that $W_q$ is preserved by the action of $\tilde{\Gamma}_n^*$, so this group acts on $Y^n_q$.
We denote the $A_{\infty}$ endomorphism algebra of the restriction of the Beilinson exceptional collection to $Y^n_q$ by $\mathcal{B}_q$. 
It is a one-parameter deformation of $\mathcal{B}$.

Following the approach to one-parameter deformations of $A_{\infty}$ algebras of \cite{seidel03}, we can show that $\mathcal{B}_q$ is a versal deformation of $\mathcal{B}$ (here we must take the $\Gamma_n^*$ equivariance, the fractional grading, and also the $S_{n+2}$ action into account). 

We then prove that $\tilde{\mathcal{A}}^n_q$ is a versal deformation of $\mathcal{A}^n \rtimes \Gamma_n^*$, in the same class as $\mathcal{B}_q$, so the two endomorphism algebras are quasi-isomorphic (up to a formal change of variables in the parameter $q$). 
Finally, we need to prove generation results on both sides, in order to show that there is a quasi-equivalence of $D^b Coh_{\Gamma_n^*}(Y^n_q)$ with $D^{\pi}\mathcal{F}uk(X^n)$.

\appendix

\section{Proof of Lemma \ref{lem:thesign2}}

\subsection{Note on conventions}

We will use results from \cite{Sheridan2017}, whose sign conventions differ slightly from those of \cite{seidel08} which we use in this paper. 
Namely, if $\mu^*$ denotes the $A_\infty$ structure maps in the Fukaya category, defined using the sign conventions of \cite{seidel08}, and $\eta^*$ are those defined using the sign conventions of \cite{Sheridan2017}, then we have
\[ \eta^k(a_1,\ldots,a_k) = \mu^k(a_k,\ldots,a_1).\]

\subsection{Morse critical points}

Let $(V,\Omega)$ be a symplectic vector space, and $\cG V$ its Lagrangian Grassmannian.
Let $\tau: V \to V$ be an anti-symplectic involution, i.e., a linear map satisfying $\tau^*\Omega = -\Omega$ and $\tau^2 = \id$. 
Let $\Lambda$ be the $+1$ eigenspace of $\tau$, and let us denote the $-1$ eigenspace of $\tau$ by $i\Lambda$ (this is purely notational: our constructions do not depend on a choice of complex structure on $V$). 
Both $\Lambda$ and $i\Lambda$ are Lagrangian, and $\Omega$ defines a perfect pairing $\Lambda \otimes i\Lambda \to \R$. 

Now let $B$ be a nondegenerate quadratic form on $\Lambda$. 
Let $\Lambda = \Lambda_+ \oplus \Lambda_-$ be the orthogonal decomposition into positive/negative eigenspaces with respect to some inner product on $\Lambda$. 
The Morse index $\mu(B)$ is defined to be the dimension of $\Lambda_-$. 
Let $A: \Lambda \to i\Lambda$ be the corresponding linear isomorphism, such that $B(v) = \Omega(v,Av)$. 
There is a corresponding path $\rho: [0,1] \to \cG V$, given by $\rho(t) = \mathrm{graph}((2t-1)A)$: this path has Maslov index $\mu(B)$ (compare \cite[\S 2d (iv)]{seidel99}). 
We will call it the `short' path between these two Lagrangian subspaces.

Associated to this short path $\rho$ we have the orientation operator $D_{H,\rho}$ on the upper half-plane $H$ (see \cite[\S B.2]{Sheridan2017}). 
Evaluation at the boundary marked point $0 \in \partial H$ defines an isomorphism of graded lines
\begin{equation}
\label{eqn:Dmorse}
 \lambda(D_{H,\rho}) \cong \lambda(\Lambda_-)
 \end{equation}
(compare \cite[Equation (11.20)]{seidel08}; this isomorphism appears in the Lagrangian PSS isomorphism between Morse cohomology of $L$ and the Floer endomorphism algebra of $L$, see \cite[Equation (12.14)]{seidel08}).
In particular, as the isomorphism \eqref{eqn:Dmorse} is a graded isomorphism, the Fredholm index of $D_{H,\rho}$ is equal to the Morse index $\mu(B)$. 

We have an isomorphism
\begin{align}
\label{eqn:Dtau}
D_{H,\rho} & \to D_{H,\rho} \\
u(s+it) & \mapsto \tau(u(-s+it)),
\end{align}
which is well-defined because  $\tau(\rho(1-t)) = \rho(t)$.
This induces an isomorphism of determinant lines:
\begin{equation}
\label{eqn:lambtau}
 \lambda(D_{H,\rho}) \to \lambda(D_{H,\rho}).\end{equation}
The isomorphism \eqref{eqn:Dtau} clearly respects the isomorphisms \eqref{eqn:Dmorse}. 
Therefore, since $\tau$ acts trivially on $\Lambda_-$, it follows that the isomorphism \eqref{eqn:lambtau} is the identity. 

\subsection{Warmup: the case of cotangent bundles}

Let $L$ be a closed $n$-manifold and $X := T^*L$. 
The tangent bundle $TX$ is trivial as a complex vector bundle, so we have a canonical holomorphic volume form $\eta$ up to homotopy, which induces a $\Z$-grading of the Fukaya category (in the language of \cite{Sheridan2017}, it defines a morphism of grading data $p: \G(X) \to \Z$ which we can use to push forward the $\G(X)$-grading to a $\Z$-grading).
This allows us to define $\fuk_{\l\r}(X)$, a version of the Fukaya category equipped with compatible leftwards and rightwards shift functors (see \cite[\S B.4]{Sheridan2017}). 
We will denote $\fuk(X) := \fuk_{\l\r}(X)$.
We have an isomorphism of categories
\[ c^X: \fuk(X) \to \fuk(\bar{X})^{op}\]
by \cite[Lemma B.10]{Sheridan2017}, where $\bar{X} = (X,-\omega,\bar{\eta})$ is the `opposite' symplectic manifold to $X = (X,\omega,\eta)$, and `$op$' denotes the opposite $A_\infty$ category.

We have an isomorphism $\tau: (X,-\omega,\bar{\eta}) \xrightarrow{\sim} (X,\omega,\eta)$, which reverses the sign of the covector in each fibre. 
This gives an isomorphism of categories
\[ \tau: \fuk(\bar{X})^{op} \to \fuk(X)^{op}.\]

Following the convention of \cite[Appendix B]{Sheridan2017}, an object of $\fuk(X)$ is a tuple $(L,\iota,\tilde{\iota},P^\#)$ where $L$ is a smooth manifold, $\iota: L \hookrightarrow X$ a Lagrangian embedding, $\tilde{\iota}$ a grading, and $P^\#$ a Pin structure on $L$. 
If we assume that $L$ is simply-connected and equipped with a Pin structure $P^\#$, then the inclusion $\iota: L \hookrightarrow T^* L$ as the zero-section can be equipped with a grading $\tilde{\iota}$, so we obtain an object $(L,\iota,\tilde{\iota},P^\#)$ of $\fuk(X)$. 

Since $\tau \circ \iota = \iota$, there is a canonical isomorphism of graded Lagrangian branes, $j: \tau L^\# \to L^\#$. 
Thus we have an algebra isomorphism
\begin{align}
\label{eqn:taufix}
 \Hom^*_{\cF(X)}(L^\#,L^\#) &\xrightarrow{c^X} \Hom^*_{\cF(\bar{X})^{op}}(L^\#,L^\#)  \\
 & \xrightarrow{\tau} \Hom^*_{\cF(X)^{op}}(\tau L^\#, \tau L^\#)  \nonumber \\
 & \xrightarrow{j} \Hom^*_{\cF(X)^{op}}(L^\#,L^\#).  \nonumber
\end{align}

The endomorphism algebra of $L^\#$ depends on a choice of Floer data (see \cite[\S 8e]{seidel08}): generators of the endomorphism algebra correspond to time-$1$ Hamiltonian chords from $L$ to $L$.
If we choose the Hamiltonian component of the Floer datum to be an extension of a small Morse function $h: L \to \R$ to a Weinstein neighbourhood of $L$, then each such chord $y$ corresponds to a critical point  $x$ of $h$. 
The associated orientation line is
\begin{equation*} o_y := \lambda(D_{H,\rho}) \otimes \mathsf{Pin}(\rho),\end{equation*}
where $\rho$ is a short path of the kind considered above (it is a valid choice for the orientation operator because its Maslov index is equal to the degree of $y$: see \cite[Definitions B.4 and B.5]{Sheridan2017}). 
The isomorphism \eqref{eqn:taufix} is defined to send $o_y \mapsto o_y$ by the conjugation isomorphism $\lambda(D_{H,\rho}) \to \lambda(D_{H,\rho})$, tensored with the isomorphism $\mathsf{Pin}(\rho) \to \mathsf{Pin}(\rho)$ which reverses the direction of $\rho$, multiplied by $-1$. 
We have argued that the first isomorphism is the identity (since $\rho$ is homotopic to the short path), and the second obviously sends the trivial Pin structure to itself, hence is also the identity. 
Therefore \eqref{eqn:taufix} is $-\id$.

Using the definition of the opposite category \cite[\S A.1]{Sheridan2017}, this means
\begin{equation*}
 \mu^2(\alpha,\beta) = (-1)^{\maltese} \mu^2(\beta,\alpha)
\end{equation*}
where
\[\maltese = 1+1+1+|\alpha|'\cdot|\beta|' = |\alpha|\cdot|\beta| + |\alpha| + |\beta|.\]
In terms of the associative algebra structure $\alpha \cdot \beta := (-1)^{|\beta|} \cdot \mu^2(\alpha,\beta)$, this means that
\[\alpha \cdot \beta = (-1)^{|\alpha|\cdot|\beta|} \beta \cdot \alpha,\]
i.e., the endomorphism algebra of $L^\#$ is super-commutative.

On the other hand we have the Lagrangian PSS isomorphism \cite[\S 12e]{seidel08}:
\[ \Hom^*_{\fuk(X)}(L^\#,L^\#) \cong H^*(L),\]
which also implies that the endomorphism algebra of $L^\#$ is super-commutative (since the cup product on cohomology always is). 
This provides a useful consistency check.

\subsection{The immersed Lagrangian sphere}

We consider the object $L^\# := (S^n,\iota,\tilde{\iota},P^\#)$, where $\iota:S^n \to \cP^n$ is the immersion of the Lagrangian sphere, $\tilde{\iota}$ a grading, and $P^\#$ a Pin structure on $S^n$ (the non-trivial one, if $n=1$).
The endomorphism algebra $\Hom^*_{\fuk(\cP^n)}(L^\#,L^\#)$ has a basis $\{p_K\}$, where $p_K$ has degree $2\bm{n} \cdot e_K - |K|$ by Proposition \ref{prop:grading}.

Let $a: S^n \to S^n$ denote the antipodal map, and $\tau: \cP^n \to \cP^n$ complex conjugation; we observe that $\tau \circ \iota = \iota \circ a$, so a choice of isomorphism $P^\# \cong a^* P^\#$ determines an isomorphism of anchored Lagrangian branes, $j:L^\# \xrightarrow{\sim} \tau L^\#$. 
Thus we have an isomorphism
\begin{align}
\label{eqn:theiso} \Hom^*_{\fuk(\cP^n)}(L^\#,L^\#) & \xrightarrow{c^{\cP^n}} \Hom^*_{\fuk(\bar{\cP}^n)^{op}}(L^\#,L^\#) \\
& \xrightarrow{\tau} \Hom^*_{\fuk(\cP^n)^{op}}(\tau L^\#, \tau L^\#) \nonumber \\
& \xrightarrow{j} \Hom^*_{\fuk(\cP^n)^{op}}(L^\#,L^\#), \nonumber
\end{align}
analogous to \eqref{eqn:taufix}. 

\begin{lem}
\label{lem:thesign}
The isomorphism \eqref{eqn:theiso} sends
\[ p_K \mapsto (-1)^{1 + \bm{n} \cdot e_K} \cdot p_K.\]
\end{lem}
\begin{proof}
By definition, the isomorphism \eqref{eqn:theiso} sends
\[ p_K \mapsto (-1)^{\sigma_K} \cdot p_K\]
for some $\sigma_K \in \Z/2$, which we now determine.

First we deal briefly with the cases $K = \emptyset, [n+2]$ (although these cases can be deduced from the others by associativity of the product).
The generators $p_\emptyset, p_{[n+2]}$ correspond to $H^*(S^n)$. 
The antipodal map $a$ acts trivially on $H^0(S^n)$ and by $(-1)^{n-1}$ on $H^n(S^n)$. 
There is an additional $-$ sign that goes into the definition of $c^{\cP^n}$, so \eqref{eqn:theiso} sends 
\begin{align*}
p_{\emptyset} & \mapsto - p_{\emptyset}, \\
p_{[n+2]} & \mapsto (-1)^n \cdot p_{[n+2]}
\end{align*}
as required (recall that $\bm{n} \cdot e_{[n+2]} = n+1$).

Now we consider the cases $K \neq \emptyset, [n+2]$.
The orientation operator associated to $p_K$ is
\begin{equation} o(p_K) := \lambda(D_{H,\rho}) \otimes \mathsf{Pin}(\rho) \end{equation}
where $\rho$ is a path in the Lagrangian Grassmannian of Maslov index $-|K| + 2\bm{n} \cdot e_K$. 
However $p_K$ corresponds to a critical point of the function $f: S^n \to \R$ of Morse index $n+1-|K|$ by Corollary \ref{cor:morseindf}, so we can not use the short path of Lagrangian subspaces $\rho$ when defining the orientation operator associated to $p_K$.
 
However we would like to be able to use the particularly simple form of the isomorphism $c^\cP$ for short paths of Lagrangian subspaces, so we apply a shift to one copy of $L^\#$. 
For any $\jmu$ we have the following commutative diagram (using notation from \cite{Sheridan2017}):
\begin{equation}
\label{eqn:comm}
 \xymatrix{ \Hom^{*}_{\cF(\cP)}(L^\#,L^\#[\jmu]) \ar[d]^{c^\cP} \ar[r]^-{s_\r^{\jmu}} & \Ts_\jmu|\Hom^{*}_{\cF(\cP)}(L^\#,L^\#) \ar[d]^{c^\cP} 
 \\
\Hom^{*}_{\cF(\bar{\cP})^{op}}(L^\#[-\jmu],L^\#) \ar[d]^{\tau} \ar[r]^{s_\l^{\jmu}} &  \Ts_\jmu|\Hom^*_{\cF(\bar{\cP})^{op}}(L^\#,L^\#) \ar[d]^-{\tau} \\
\Hom^{*}_{\cF(\cP)^{op}}(\tau L^\#[-\jmu],\tau L^\#) \ar[d]^j \ar[r]^{s_\l^{\jmu}} &  \Ts_\jmu|\Hom^*_{\cF(\cP)^{op}}(\tau L^\#,\tau L^\#) \ar[d]^j  \\
\Hom^{*}_{\cF(\cP)^{op}}(L^\#[-\jmu],L^\#) \ar[d]^{s^{-\jmu}} \ar[r]^{s_\l^{\jmu}} & \Ts_\jmu|\Hom^*_{\cF(\cP)^{op}}(L^\#,L^\#)  \\
 \Hom^*_{\cF(\cP)^{op}}(L^\#,L^\#[\jmu]) \ar[ur]_{(-1)^{\dag_\jmu}\cdot s_\r^{\jmu}},  
 }
\end{equation}
where $\dag_\jmu = \jmu(\jmu-1)/2$.
In this diagram, the top square commutes by the fact that $c^{\cP^n}$ respects shift maps \cite[Lemma B.12]{Sheridan2017}. 
The next square commutes by naturality of shift maps under symplectomorphisms. 
The next square commutes by naturality of shift maps under isomorphism of branes. 
The bottom triangle commutes by \cite[Lemma B.3]{Sheridan2017}.

Let $\jmu_K = n+1-2\bm{n} \cdot e_K$. 
Let $q_K$ be the generator of $\Hom^*_{\fuk(\cP^n)}(L^\#,L^\#[\jmu_K])$ corresponding to $p_K$: i.e., $s_\r^{\jmu_K}(q_K) = 1|p_K$. 
The isomorphism running down the right-hand side of the diagram \eqref{eqn:comm} is \eqref{eqn:theiso}, so it sends $1|p_K \mapsto (-1)^{\sigma_K} \cdot 1|p_K$ by definition of $\sigma_K$.
By commutativity, the isomorphism running down the left-hand side sends $q_K \mapsto (-1)^{\dag_{\jmu_K} + \sigma_K} q_K$. 

The orientation operator associated to $q_K$ is
\begin{equation*} o(q_K) := \lambda(D_{H,\rho}) \otimes \mathsf{Pin}(\rho)\end{equation*}
where $\rho$ is a path of Maslov index $\deg(q_K) = \deg(p_K)+\jmu_K=n+1-|K|$, which is equal to the Morse index of the critical point corresponding to $p_K$: so now we may choose $\rho$ to be the `short' path.
Thus the isomorphism $o(q_K) \to o(q_K)$ is the tensor product of isomorphisms
\begin{align}
\lambda(D_{H,\rho}) & \xrightarrow{\id} \lambda(D_{H,\rho}), \\
\label{eqn:pinsiso}\mathsf{Pin}(\rho) & \to \mathsf{Pin}(\rho),
\end{align}
multiplied by $-1$.
To determine the sign of the isomorphism \eqref{eqn:pinsiso}, we must follow through the isomorphisms of Pin structures down the left side of \eqref{eqn:comm}. 

Let us suppose that $p_K$ represents a constant chord connecting $x \in S^n$ to $a(x) \in S^n$. 
Let $j_x: P^\#_x \to P^\#_{a(x)}$ denote our chosen isomorphism of Pin structures $P^\# \cong a^*P^\#$, for $x \in S^n$. 
$\mathsf{Pin}(\rho)$ is the $\Z/2$-torsor
\[ \mathsf{Iso}(P^\#_x, P^\#_{a(x)}) \otimes \lambda(T_{a(x)}S^n)^{\otimes \jmu_K},\]
where $\mathsf{Iso}(P^\#_x,P^\#_{a(x)})$ is the torsor of isomorphisms of principal homogeneous $\mathrm{Pin}_n$ spaces covering the antipodal map. 

The isomorphism running down the left side of \eqref{eqn:comm} acts trivially on $\lambda(T_{a(x)} S^n)$ (because parallel transport along the short path $\rho$ coincides with the differential of the antipodal map). 
It sends the isomorphism $f: P^\#_x \to P^\#_{a(x)}$ to the isomorphism $g$ which makes the following diagram commute:
\[ \xymatrix{P^\#_x \ar[r]^-{f} \ar[d]^{j_x} & P^\#_{a(x)} \ar[d]^{j_{a(x)}} \\
P^\#_{a(x)}  & P^\#_{x}\ar[l]^-{g}. }
\]
In particular, it sends $f = j_x$ to $g = j_{a(x)}^{-1}$. 
Thus the isomorphism running down the left side of \eqref{eqn:comm} acts by $(-1)^\ddag$ on $\mathsf{Pin}(\rho)$, where
\[ j_{a(x)} \circ j_x = (-1)^{\ddag} \cdot \id.\]

Now we claim that $\ddag = n(n+1)/2$. 
To prove this, observe that $j_{a(x)} \circ j_x = \id$ if and only if our Pin structure on $S^n$ is pulled back from one on $\RP{n}$. 
When $n \ge 2$, the Pin structure on $S^n$ is unique, so $j_{a(x)} \circ j_x = \id$ if and only if $\RP{n}$ admits a Pin structure. 
This happens if and only if $w_2(T(\RP{n}))$ vanishes, i.e., if and only if $n(n+1)/2$ is even. 
When $n=1$, of course $\RP{1}$ admits two Pin structures, but both pull back to the trivial Pin structure on $S^1$: since we have chosen the non-trivial Pin structure on $S^1$, it follows that $j_{a(x)} \circ j_x = -\id = (-1)^{(1)(1+1)/2} \cdot \id$ in this case also.

Combining the signs, we have shown that
\begin{align*}
 \sigma_K & = \dag_{\jmu_K} + \ddag + 1 \\
 &= \frac{(n+1-2\bm{n} \cdot e_K)(n-2\bm{n} \cdot e_K)}{2} + \frac{n(n+1)}{2} + 1 \\
 &= 1+ \bm{n} \cdot e_K
 \end{align*}
as required.
\end{proof}

\bibliographystyle{plain}
\bibliography{biblio}

\end{document}